\documentclass[10pt,a4paper]{amsart}
\usepackage[all,cmtip,ps]{xy}
\xyoption{rotate}
\usepackage{amsmath}
\usepackage{amsthm}
\usepackage{amssymb}
\usepackage{amscd} 
\usepackage{hyperref}
\usepackage{dashrule}
\usepackage[svgnames,hyperref]{xcolor}
\usepackage{empheq}
\usepackage[many]{tcolorbox}
\tcbset{highlight math style={enhanced,
  colframe=white!60!white,colback=lightgray!50!white,arc=4pt,boxrule=1pt}}
\newcommand{\QcX}{{\mathrm {Qcoh}}(X)}
\newcommand{\QcY}{{\mathrm {Qcoh}}(Y)}
\newcommand{\chX}{{\mathrm {coh}}(X)}
\newcommand{\chY}{{\mathrm {coh}}(Y)}

\newcommand{\Pvu}{{}^{-1}\mathrm {Perv}(Y/X)}
\newcommand{\Pvz}{{}^{0}\mathrm {Perv}(Y/X)}
\newcommand{\cA}{\mathcal{A}}
\newcommand{\cB}{\mathcal{B}}
\newcommand{\cC}{\mathcal{C}}
\newcommand{\cL}{\mathcal{L}}

\newcommand{\cE}{\mathcal{E}}

\def\cHom{{\mathcal H}\!o\!m}

\newcommand{\cI}{\mathcal{I}}

\newcommand{\cD}{\mathcal{D}}
\newcommand{\cK}{\mathcal{K}}
\newcommand{\cG}{\mathcal{G}}
\newcommand{\cM}{\mathcal{M}}
\newcommand{\cN}{\mathcal{N}}
\newcommand{\cO}{\mathcal{O}}
\newcommand{\cT}{\mathcal{T}}
\newcommand{\cF}{\mathcal{F}}
\newcommand{\cR}{\mathcal{R}}
\newcommand{\cH}{\mathcal{H}}
\newcommand{\cP}{\mathcal{P}}
\newcommand{\cU}{\mathcal{U}}

\newcommand{\cX}{\mathcal{X}}
\newcommand{\cY}{\mathcal{Y}}

\newcommand{\bR}{\mathbf{R}}
\newcommand{\bL}{\mathbf{L}}

\newcommand{\bZ}{\mathbb{Z}}
\newcommand{\point}{{\scriptscriptstyle\bullet}}

\DeclareMathOperator{\Coker}{Coker}

\DeclareMathOperator{\Hom}{Hom}

\DeclareMathOperator{\End}{End}
\DeclareMathOperator{\Ext}{Ext}

\DeclareMathOperator{\Ker}{Ker}

\DeclareMathOperator{\im}{Im}

\DeclareMathOperator{\id}{id}

\newcommand*{\N}{\mathcal{N}}
\DeclareMathOperator{\R}{\mathbf{R}}

\newcommand*{\leff}{\textrm{\textup{-eff}}}
\newcommand*{\reff}{\textrm{\textup{eff-}}}
\newcommand*{\proj}{\textrm{\textup{proj}}}
\newcommand*{\lproj}{\textrm{\textup{-proj}}}
\newcommand*{\rproj}{\textrm{\textup{proj-}}}

\newcommand*{\lMod}{\textrm{\textup{-Mod}}}
\newcommand*{\rMod}{\textrm{\textup{Mod-}}}

\newcommand*{\lcoh}{\textrm{\textup{-coh}}}
\newcommand*{\rcoh}{\textrm{\textup{coh-}}}
\newcommand*{\lfp}{\textrm{\textup{-fp}}}
\newcommand*{\rfg}{\textrm{\textup{fg-}}}

\newcommand*{\rfp}{\textrm{\textup{fp-}}}
\newtheorem{theorem}{Theorem}[section]
\newtheorem{lemma}[theorem]{Lemma}
\newtheorem{proposition}[theorem]{Proposition}
\newtheorem{corollary}[theorem]{Corollary}
\theoremstyle{definition}
\newtheorem{conjecture}[theorem]{Conjecture}
\newtheorem{definition}[theorem]{Definition}

\newtheorem{remark}[theorem]{Remark}
\newtheorem{example}[theorem]{Example}
\newtheorem{numero}[theorem]{}

\def\cDll#1{{\cD}^{\leq #1}}
\def\cDgg#1{{\cD}^{\geq #1}}
\def\cTll#1{{\cT}^{\leq #1}}
\def\cTgg#1{{\cT}^{\geq #1}}

\def\ra{{\rightarrow}}
\newcommand{\hH}{\mathrm{H}}

\def\backmatter{\renewcommand{\baselinestretch}{1}\normalfont}

\begin{document}
\title[N-Quasi-Abelian Categories vs  N-Tilting Torsion Pairs]{N-Quasi-Abelian Categories \\vs  \\N-Tilting Torsion Pairs\\
\tiny{\tiny{with an application to
Flops of higher relative dimension}}
}
\author{Luisa Fiorot 
}
\address{
Dipartimento di Matematica ``Tullio Levi-Civita'', 
Universit\`a degli studi di Padova, via Trieste 63, I-35121 Padova Italy}
\email{luisa.fiorot@unipd.it}

\begin{abstract}
It is a well established fact that the notions of quasi-abelian categories
 and 
tilting torsion pairs are equivalent. This equivalence fits in a wider picture including tilting pairs of $t$-structures.

\smallskip

Firstly, we extend this picture into a hierarchy of
$n$-quasi-abelian categories
 and $n$-tilting torsion
classes. We prove that
any  $n$-quasi-abelian category $\cE$ admits a ``derived'' category  
$D(\cE)$  endowed with a $n$-tilting pair   of $t$-structures
 such that the respective hearts are derived equivalent. 
 
Secondly, we describe the hearts of these $t$-structures as quotient categories of coherent functors, generalizing Auslander's Formula.

Thirdly, we apply our results to Bridgeland's theory of perverse coherent sheaves
for   flop contractions.
In Bridgeland's work, the   relative dimension $1$ assumption guaranteed that $f_*$-acyclic coherent sheaves form  a $1$-tilting torsion class,
whose associated heart is  derived equivalent to $D(Y)$.
We generalize this theorem to  relative dimension $2$.
  \end{abstract}

\maketitle
\vspace*{-\baselineskip}%
{\footnotesize\tableofcontents}
\section*{Introduction}

In \cite[3.3.1]{BBD} Beilinson, Bernstein and Deligne introduced the notion of a
$t$-structure obtained by tilting
the natural one on $D(\cA)$ (derived category of an abelian category $\cA$) with respect to a torsion pair $(\cX,\cY)$. In
 \cite{HRS} Happel Reiten and Smal{\o} developed  this procedure, they 
proved 
the  Tilting Theorem: 
whenever $\cX$ is a  \emph{tilting torsion class} (\ref{D:tilttorpair})
on 
$\cA$ 
there is a triangulated equivalence
$D(\cH)\simeq D(\cA)$ where $\cH$ is the heart of the tilted $t$-structure.

In \cite{Schn} J.-P. Schneiders 
associated to any quasi-abelian category $\cE$ (\ref{N4.1})
a triangulated category $D(\cE)$  endowed with a $1$-tilting pair   of $t$-structures
$(\cR,\cL)$ (\ref{D:1tpt}) 
such that  $\cE=\cH_\cR\cap\cH_\cL$.

Rump in \cite{Rual}, followed by
 Bondal and Van den Bergh in   \cite[App. B]{BonVdBergh},
established 
 an  equivalence between the previous notions:
 given an additive category  $\cE$, the 
following properties are equivalent:
1) $\cE$ is a $1$-quasi-abelian category, 2) $\cE$ is a $1$-tilting torsion class, 3) $\cE$ is a $1$-cotilting torsion-free class, 4) $\cE$ is the 
intersection of the hearts of a $1$-tilting
pair of $t$-structures $(\cR, \cL)$  
on $D(\cE)$.  

\bigskip
This paper contains three main results.

\subsection{}
We propose an higher analog of the previous equivalence:
given an additive category  $\cE$, the 
following properties are equivalent:
1) $\cE$ is a $n$-quasi-abelian (\ref{D:n-qa}), 2) $\cE$ is a $n$-tilting torsion class (\ref{D:ntiltorcl}), 3) $\cE$ is a $1$-cotilting torsion-free class, 4) $\cE$ is the 
intersection of the hearts of a $n$-tilting
pair of $t$-structures $(\cR, \cL)$  
on $D(\cE)$.  

 In particular, we prove that the derived category of an $n$-quasi-abelian category $\cE$ has two canonical $t$-structures
(the left and the right one).
We can view the hearts of these $t$-structures as canonical abelian envelopes for $\cE$.

\subsection{}
We establish a new description
for the hearts of these $t$-structures
 as Gabriel quotients of
the category of coherent functors with respect to a suitable Serre subcategory of
effaceable functors   (\ref{T:HvsCF}).

%
For an abelian  ($0$-quasi-abelian) category $\cA$ this result reduces to 
the {\emph Auslander's Formula} introduced by H. Krause in \cite{KRdaf}.

%
\subsection{}
Our main application is the generalization of Bridgeland's theory of perverse coherent sheaves.
Let consider 
$Y\stackrel{f}\to X$  a flop contraction with $X$ and $Y$ varieties over $\Bbb C$, $Y$ smooth
and $Y^+\stackrel{f^+}\to X$ its flop.
The Bondal-Orlov conjecture predicts that
the derived categories $D(Y)$ and $D({Y^+})$
(of coherent $\cO$-modules) are equivalent.
 Bridgeland proved the Bondal-Orlov conjecture for threefolds (\cite{Bri1})
and  Van den Bergh  proposed a different proof relaxing some hypotheses, but always assuming that $f$ has relative dimension $1$  (\cite{VdB}).

 Bridgeland considered the $t$-structures on $D(Y)$ 
obtained
by   tilting the natural $t$-structure with respect to the $1$-tilting torsion classes
\[
\cT_{0}:= \{T\in \chY \; | \, \bR f_* T\simeq f^*T\} \quad \hbox{and}\quad
\cT_{-1}:= \{T\in \chY  \; | \, \eta_T: f^*f_*T\twoheadrightarrow T\}.
\]

He denoted by
${\;}^{0}{\rm Per}(Y/X)$ and
  ${\;}^{-1}{\rm Per}(Y/X)$  their respective hearts.
These categories form the first main ingredient  
in Bridgeland's and Van den Bergh's proofs of the Bondal-Orlov conjecture in the relative dimension 1 case.

In these proofs the use of $1$-tilting torsion classes  is dictated by the geometry of the problem since the fibers of the flops  have dimension $\leq 1$. 
In the case of relative dimension $n$
we define
\[
\cT_{0}:= \{T\in \chY \; | \, \bR f_* T\simeq f^*T\} \quad \hbox{and}\quad
\cT_{-1}:= \{T\in \cT_{0} \; | \, \eta_T: f^*f_*T\twoheadrightarrow T\} 
\]
as $n$ analogues of the previous  classes.
We prove tha,t for $n=2$, these are $2$-tilting torsion classes. We denote by 
${\;}^{i}{\rm Per}(Y/X))$ the respective hearts and we prove that
$$
 D(Y)\simeq D({\;}^{0}{\rm Per}(Y/X))\simeq D({\;}^{-1}{\rm Per}(Y/X)).
 $$

\medskip

{{\it Acknowledgements.}  We are grateful to  P. J{\o}rgensen who sent us  
useful comments and remarks on the first redaction of this work.
}

\section{$1$-tilting torsion classes  }

In what follows any full subcategory $\cC'$ of an additive category $\cC$ will be strictly full
(i.e., closed under isomorphisms) and additive and we will use the notation $\cC'\subseteq \cC$ to indicate such a subcategory. Any functor between additive categories will be an additive functor.

\numero{\bf Torsion pairs in abelian categories }\label{num:torpair}(\cite{Di}). 
A \emph{torsion pair} in an abelian category $\cA$ is a pair $(\cX,\cY)$ of full subcategories of 
$\cA$ such that:
\begin{enumerate}
\item[\rm (i)] $\cA(X,Y)=0$,  for every $X\in \cX$ and every $Y\in\cY$
(where $\cA(X,Y)$ denotes the group of morphisms from $X$ to $Y$ in $\cA$).
\item[\rm (ii)] For any object $C\in\cA$ there exists a short exact sequence 
$0 \to X \to C \to Y \to 0$ in $\cA$ with $X\in\cX$ and $Y\in\cY.$

\end{enumerate}
The class $\cX$ is called the \emph{torsion class} while $\cY$ is called the \emph{torsion-free class}.
The pervious conditions imply that the ``inclusion" functor $i_{\cX}:\cX\to \cA$ has a right adjoint $\tau$ while $i_{\cY}:\cY\to\cA$ has a
left adjoint $\phi$; moreover the endo-functors $t:=i_\cX\circ\tau$ and $f:=i_\cY\circ \phi$ are radicals.
The class
$\cX$ (resp. $\cY$) is closed under extensions, quotients (resp. subobjects) representable direct sums
(resp. direct products).
%
As observed in \cite[5.4]{BonVdBergh} both $\cX$ and $\cY$ admit kernels and cokernels
such that: $\Ker_{\cX}=\tau\circ \Ker _{\cA}$, $\Coker_{\cX}=\Coker _{\cA}$;
$\Ker_{\cY} = \Ker _{\cA}$ and $\Coker_{\cY}=\phi\circ \Coker _{\cA}$.

\begin{definition}\label{D:tilttorpair} (\cite{HRS})
A  torsion pair $(\cX,\cY)$ is called 
\emph{tilting} if $\cX$ \emph{cogenerates} $\cA$ (i.e., every object in $\cA$ is a subobject of an object in $\cX$)
and 
$\cX$ is called a \emph{$1$-tilting torsion class} (in $\cA$).
Dually  $(\cX,\cY)$ is \emph{cotilting} if $\cY$ \emph{generates} $\cA$
(i.e., every object in $\cA$ is a quotient of an
object in $\cY$) and $\cY$ is called a \emph{$1$-cotilting torsion-free class}.
\end{definition}
%


\begin{lemma}\label{L:eq1ttc}
Let $\cA$ be an abelian category.
The full subcategory $\cE\stackrel{i_{\cE}}\hookrightarrow \cA$ is a $1$-tilting torsion class if and only if
\begin{enumerate}
\item $\cE$ cogenerates $\cA$;
\item $\cE$ is closed under extensions in $\cA$;
\item  $\cE$ is closed under representable direct sums in $\cA$;
\item for any exact sequence  $0\to A\to X\to B\to 0$ in $\cA$ with $X\in\cE$ and $A,B\in\cA$ we have $B\in\cE$;
\item $\cE$ has kernels.
\end{enumerate}
\end{lemma}
\begin{proof}
Any tilting torsion class 
satisfies these conditions.
On the other side let $\cE\stackrel{i_{\cE}}\hookrightarrow \cA$ be a full subcategory satisfying the previous conditions. 
Hence, by the first property,
we can co-present  any $A\in\cA$ as $A=\Ker_{\cA}f$ with $X_1\stackrel{f}\to X_2$ and $X_i\in\cE$ for $i=1,2$
and so, since the functor $\rMod\cE\ni\cA(i_{\cE}(\_), A)\cong\cE(\_ ,\Ker_{\cE}f)$,
we can define $\tau(A):=\Ker_{\cE}f$ (using the last property)
which gives a right adjoint for the functor $i_{\cE}$. 
We remark that the functoriality of this construction is guaranteed by the fact that if we change the
co-presentation of $A$ as $A=\Ker_{\cA}g$ with $Y_1\stackrel{g}\to Y_2$ and $Y_i\in\cE$ for $i=1,2$
there exists a unique isomorphism $\Ker_{\cE}f\stackrel{\phi}\simeq \Ker_{\cE}g$ such that the following triangle commutes:
$$\xymatrixcolsep{2em}
\xymatrixrowsep{1em}
\xymatrix{
\Ker_{\cE}f \ar[rr]^{\phi} \ar[rd]& & \Ker_{\cE}g \ar[ld] \\
& A & \\}
$$
The fourth property implies that for any $A\in\cA$ the co-unit of the adjunction
$\varepsilon_{A}:i_{\cE}\tau (A)\to A$ is a monomorphism.
So for any $A\in \cA$ we have a short exact sequence
$0\to i_{\cE}\tau(A)\stackrel{\varepsilon_A}\to A\to \Coker (\varepsilon_A)\to 0$.
Moreover
$\Coker_{\cA}(\varepsilon_A)\in \cE^\perp$ (see Appendix~\ref{N:oc} for the notion of orthogonal class)
since given any morphism $f:E\to \Coker_{\cA}(\varepsilon_A)$ with $E\in\cE$
its $\cA$ pull-back
$A\times_{\Coker_{\cA}(\varepsilon_A)}E$ belongs to $\cE$ (by the second property since it is an extension of $E$ by $i_{\cE}\tau(A)$),
 hence the pull-back morphism 
$f':A\times_{\Coker_{\cA}(\varepsilon_A)}E \to A$ factors (by adjunction) through  $i_{\cE}\tau(A)$
which implies that
 $f=0$.
\end{proof}

\begin{corollary}
Let $\cA$ be a well powered abelian category with arbitrary direct sums. The full subcategory $\cE\stackrel{i_{\cE}}\hookrightarrow \cA$ is a $1$-tilting torsion class if and only if
\begin{enumerate}
\item $\cE$ cogenerates $\cA$;
\item $\cE$ is closed under extensions in $\cA$;
\item  $\cE$ is closed under direct sums in $\cA$;
\item for any exact sequence  $0\to A\to X\to B\to 0$ in $\cA$ with $X\in\cE$ and $A,B\in\cA$ we have $B\in\cE$.
\end{enumerate}
\end{corollary}

We note that  the torsion pair $(\cA,0)$ in an abelian category $\cA$ is tilting while
$(0,\cA)$ is cotilting.
So the identity ${\rm id}_\cA:\cA\to\cA$ represents $\cA$ as a $1$-tilting torsion class
and also as a $1$-cotilting torsion-free class.

\smallskip

We will refer to Appendix~\ref{Ap2} for some generalities on $t$-structures.
In particular in order to assure that any category introduced in this work has Hom sets we 
will suppose in  the whole paper the following:

\numero{\bf Hypothesis HS}.
Given $\cE$
 a \emph{projectively complete} category (i.e., additive category such that any idempotent splits)
its derived
category $D(\cE):=D(\cE,{\cE}x_{\rm max})$ 
(endowed with its maximal Quillen exact structure see Appendix~\ref{QES}) has Hom sets.

In the following we will always suppose that $\cE$ is a projectively complete category.

\numero{\bf Happel-Reiten-Smal{\o} tilted $t$-structure.}\label{numero:HRS}
 \cite[Prop.~I.2.1, Prop.~I.3.2]{HRS} \cite[Prop.~2.5]{bridgeland2005t}.
Let $\cH_\cD$ be the heart of a non degenerate
$t$-structure $\cD=(\cDll0,\cDgg0)$ on a triangulated category $\cC$ and let 
$(\cX,\cY)$ be a torsion pair on $\cH_\cD$. Then the pair $\cT:=(\cTll0_{(\cX,\cY)},\cTgg0_{(\cX,\cY)})$ of full subcategories of $\cC$
\[
\begin{matrix}
\cT^{\leq 0}_{(\cX,\cY)}= & \{ C\in \cC\; | \; H_\cD^0(C)\in\cX,\; H_\cD^i(C)=0 \;  \forall i>0 \} \hfill \\
\cT^{\geq 0}_{(\cX,\cY)}=  &\{C\in \cC\; | \; H_\cD^{-1}(C)\in\cY,\; H_\cD^i(C)=0 \;  \forall i<-1 \} \\ \end{matrix}
\]
is a $t$-structure on $\cC$. 
Following \cite{bridgeland2005t} we say that $\cT$ is obtained 
\emph{by right tilting $\cD$ with respect to the torsion pair $(\cX,\cY)$} while
the $t$-structure 
${\overline{\cT}}:={\cT}[-1]$
is called the $t$-structure obtained
\emph{by left tilting $\cT$ with respect to the torsion pair $(\cX,\cY)$}. 
The right  
tilted heart is:
\[
{\mathcal H}_\cT
=  \{ C\in \cC\; | \; 
H_{\cD}^0(C)\in{\cX},\; 
H_{\cD}^{-1}(C)\in{\cY},\; H_{\cD}^i(C)=0 \;  \forall i\notin \{-1,0\} \} .
\]
In this paper we simply call tilting the right one.
In \cite[Lem. 1.1.2]{Pol} Polishchuk proved that given any pair of $t$-structures
$(\cD,\cT)$ on a triangulated category $\cC$ such that
$\cD^{\leq -1}\subseteq \cT^{\leq 0}\subseteq \cD^{\leq 0}$, the $t$-structure
$\cT$ is obtained by right tilting $\cD$ with respect to the torsion pair 
$(\cX:=\cH_\cT\cap\cH_\cD,\cH_\cT[-1]\cap\cH_\cD=:\cY)$ while
$\cD$ is obtained by left tilting $\cT$ with respect to the tilted torsion pair
$(\cY[1]=\cH_\cD[1]\cap\cH_\cT,\cH_\cD\cap\cH_\cT=:\cX)$.

\numero{\bf Notation}. \label{N:tf}
In this paper whenever we have a pair of $t$-structures $(\cD, \cT)$ on a triangulated 
category $\cC$ we will denote by $\delta^{\leq 0}$ the truncation functor with respect to
$\cD$ and by $\tau^{\leq 0}$ the one with respect to $\cT$.

\begin{theorem}{\bf 1-Tilting Theorem.} (\cite[Th.~I.3.3]{HRS}, \cite{Chen}).  \label{Th:1--tilt}
Given a tilting  torsion pair $(\cE,\cY)$ in $\cA$ 
there exists a triangle equivalence 
$D(\cH_{\cT})\stackrel{\simeq}\to D(\cA)$
(where $\cH_{\cT}$ is the heart of the $t$-structure obtained
by right tilting the natural $t$-structure 
 with respect to the torsion pair $(\cE,\cY)$)
which is compatible with the natural inclusion 
$\cH_\cT\subseteq D(\cA)$.
Moreover $(\cY[1],\cE)$ is a cotilting torsion pair in $\cH_\cT$.
\end{theorem}
%

\begin{definition}\label{D:1tpt}
A pair of $t$-structures
$(\cD,\cT)$ on a triangulated category $\cC$ is called \emph{$1$-tilting} if
the following two conditions  hold:
\begin{enumerate}
\item $\cD^{\leq -1}\subseteq \cT^{\leq 0}\subseteq \cD^{\leq 0}$;
\item denoting by $\cE:=\cH_\cD\cap\cH_\cT$,
the following equivalent conditions are satisfied:
\begin{description}
\item[(i)] 
$\cC\simeq  K(\cE)/\cN$ and
$D(\cH_\cD)\stackrel{\simeq}\hookleftarrow K(\cE)/\cN
\stackrel{\simeq}\hookrightarrow D(\cH_\cT)$
where $\cN$ is the null system of complexes in 
$K(\cE)$ acyclic in $\cH_{\cD}$
 or equivalently in
$\cH_\cT$;
\item[(ii)]  $\cC\simeq D(\cH_\cD)$ and $\cE$ cogenerates $\cH_\cD$;
\item[(iii)] $\cC\simeq D(\cH_\cT)$ and $\cE$ generates $\cH_\cT$.
\end{description}
\end{enumerate}
\end{definition}

\begin{proposition}\label{P:1t}
The pair $(\cD,\cT)$ is a $1$-tilting pair of $t$-structures if and only if 
$\cE:=\cH_\cD\cap\cH_\cT$ is a $1$-tilting torsion class (resp. $1$-tilting torsion-free class) 
in $\cH_\cD$ (resp. in $\cH_\cT$).
\end{proposition}
\begin{proof}
One implication is a consequence of the $1$-Tilting Theorem~\ref{Th:1--tilt}: if
$\cE:=\cH_\cD\cap\cH_\cT$ is a $1$-tilting torsion class (resp. $1$-tilting torsion-free class) 
in $\cH_\cD$ (resp. in $\cH_\cT$) we obtain that  
$(\cD,\cT)$ is a $1$-tilting pair of $t$-structures.
On the other side if $(\cD,\cT)$ is a $1$-tilting pair of $t$-structures by 
\cite[Lem.~1.1.2]{Pol} $\cE:=\cH_\cD\cap\cH_\cT$ is a torsion class in $\cH_\cD$ so
we have only to prove that
$\cE$ cogenerates $\cH_\cD$.
By hypothesis $K(\cE)/\cN\stackrel{\simeq}\hookrightarrow D(\cH_\cD)$ 
so any $A\in \cH_\cD$
can be represented by a complex $E^\point\in K(\cE)$
, hence $A\hookrightarrow \Coker_{\cH_\cD}(d^{-1}_{E^\point})\in \cE$ 
(and $ \Coker_{\cH_\cD}(d^{-1}_{E^\point})\in \cE$ since it is a quotient of a 
torsion object in $\cH_\cD$).
Dually if  $K(\cE)/\cN\stackrel{\simeq}\hookrightarrow D(\cH_\cT)$  we have that $\cE$ generates $\cH_\cT$
and it is a torsion-free class in $\cH_\cT$. 
\end{proof}
%
\begin{definition}(\cite{Schn}).\label{N4.1}
An additive category $\cE$ is called \emph{$1$-quasi-abelian} if it admits
kernels and cokernels, and 
any push-out of a kernel
is a kernel, and any pullback of a cokernel 
 is a cokernel. 
A  zero sequence 
$0\to E\stackrel{u}\to F\stackrel{v}\to G\to 0$
 is called \emph{exact} if and only if $(E,u)$ is the kernel of $v$ and
$(G,v)$ is the cokernel of $u$.
A complex $X^\point$ with entries in $\cE$ is called
\emph{acyclic} if each differential $d^n:X^n\to X^{n+1}$ decomposes in $\cE$ as
$d^n=m_n\circ e_n:
\xymatrix@-10pt{X^n\ar@{->>}[r]^{e_n} & D^n\ar@{>->}[r]^{m_n} & X^{n+1}}$
where $m_n$ is the kernel of $e_{n+1}$, and $e_{n+1}$ is the cokernel of
$m_{n}$  for any $n\in\Bbb Z$.
\end{definition}
\begin{remark}
The class of kernel-cokernel exact sequences provides the 
maximal Quillen exact
structure on $\cE$ if and only if $\cE$ is $1$-quasi-abelian (see Appendix~\ref{QES}
for the notion of maximal Quillen exact structure). 
\end{remark}
\numero{\bf Left and Right $t$-structures on the derived category of a quasi-abelian category}
\label{num:lefttstr} 
(\cite[\S 1.2]{Schn}).
Let $\cL\cK_\cE^{\leq 0}$ (resp. $\cR\cK_\cE^{\geq 0}$) denote the full subcategory of $K(\cE)$
formed by complexes which are 
isomorphic in $K(\cE)$ to complexes whose entries in each strictly positive 
(resp. strictly negative) degree are zero.
Let now suppose that $\cE$ admits kernels and cokernels, 
hence the pairs 
$\cL\cK_\cE:=(\cL\cK_\cE^{\leq 0},(\cL\cK_\cE^{\leq -1})^{\bot})$ 
and
$\cR\cK_\cE:=({}^{\bot}(\cR\cK_\cE^{\geq 1}),\cR\cK_\cE^{\geq 0})$
define two $t$-structures on $K(\cE)$ whose truncation functors are resp.:
{\small{\[ \begin{matrix} \hfill \tau_\cL^{\leq 0} E^\point := 
\cdots \longrightarrow & E^{-2} &\longrightarrow & E^{-1} &\longrightarrow 
&\stackrel{\point}{\Ker_{\cE}d^0}& \longrightarrow & 0  \longrightarrow 
\cdots \\  
\hfill  \tau_\cL^{\geq 1} E^\point := 
\cdots \longrightarrow & 0 &\longrightarrow & \Ker_{\cE}d^0 & \longrightarrow & 
\stackrel{\point}{E^0}& \longrightarrow  & E^{1} \longrightarrow \cdots \\
 \hfill\tau_\cR^{\leq -1} E^\point := 
\cdots\longrightarrow &E^{-1}&\longrightarrow 
& \stackrel{\point}{E^0}&\longrightarrow &\Coker_{\cE}d^{-1}&\longrightarrow & 
0 \longrightarrow  \cdots
 \\  
 \hfill \tau_\cR^{\geq 0} E^\point := 
\cdots\longrightarrow & 0 &\longrightarrow & 
\stackrel{\point}\Coker_{\cE} d^{-1}&\longrightarrow  & E^{1}&\longrightarrow 
& E^{2}\longrightarrow \cdots \\
\end{matrix}
\]}}(as  in  \ref{N:point} we use a point to indicate the object placed in degree 0).
The left $t$-structure $\cL\cK_\cE$ is the one considered by Schneiders in
\cite[Prop.~1.2.4]{Schn}. 
We will denote by $\cL\cK(\cE)$ (resp. $\cR\cK(\cE)$) the heart associated to 
the $t$-structure $\cL\cK_\cE$ (resp. $\cR\cK_\cE$). 
We have $\cE \simeq \cL\cK_\cE^{\leq 0}\cap \cR\cK_\cE^{\geq 0}=\cL\cK(\cE)\cap\cR\cK(\cE)$ 
 in $K(\cE$) and moreover 
$\cR\cK_\cE^{\leq -2}\subseteq\cL\cK_\cE^{\leq 0}\subseteq \cR\cK_\cE^{\leq 0}$
(since for any $E^\point\in K(\cE)$ its
$\tau_{\cR}^{\leq -2}(E^\point)\in \cL\cK_\cE^{\leq 0}$).
In $K(\cE)$ we have that $\cR\cK_\cE^{\leq -1}$ is contained in $\cL\cK_\cE^{\leq 0}$
if and only if any cokernel map is a split epimorphism or equivalently
any kernel map is a split monomorphism. 
If this is not the case in order to reduce the ``gap" (\cite[Definition 2.1]{FMT})
between the left and the right $t$-structures
(without changing the intersection $\cE$)
we can try to localize by a null system formed by
acyclic complexes with respect to a Quillen exact structure. 
In this case, if the previous  $t$-structures satisfy the conditions of Lemma~\ref{Prop:Schn},
they will induce a pair of $t$-structures $(\cR\cD_{(\cE,{\cE}x)},\cL\cD_{(\cE,{\cE}x)})$
on the localized category $D(\cE,{\cE}x)$.
In order to obtain 
$\cR\cD_{(\cE,{\cE}x)}^{\leq -1}\subseteq\cL\cD_{(\cE,{\cE}x)}^{\leq 0}\subseteq 
\cR\cD_{(\cE,{\cE}x)}^{\leq 0}$
we need to prove that 
for any $E^\point \in D(\cE,{\cE}x)$ the canonical morphism of complexes
$\alpha_{E^\point}:\tau_{\cL}^{\leq 0}( \tau_{\cR}^{\leq -1} E^\point )\to  
\tau_{\cR}^{\leq -1} E^\point $
is an isomorphism in $D(\cE,{\cE}x)$:
{\small{\[\xymatrixcolsep{1.6em}\xymatrixrowsep{0.8em}
\xymatrix@-4pt{
\tau_{\cL}^{\leq 0}(\tau_\cR^{\leq -1} E^\point) := 
\cdots\ar[r] \ar[d]_{ \alpha_{E^\point}}& E^{-1}\ar[r]\ar[d]  &
\mathop{\im_{\cE}(d^{-1})}\limits^{\point}_{\;}\ar[r] \ar[d] & 0\ar[r] \ar[d] & 0\ar[r]\ar[d] &
\cdots \\  
\tau_\cR^{\leq -1} E^\point := 
\cdots\ar[r] &E^{-1}\ar[r]^{d^{-1}} &
\mathop{E^0}\limits^{\point}_{\;}\ar[r] & \Coker_{\cE}(d^{-1}) \ar[r] & 0 \ar[r] & 
\cdots
 \\   }\]}}which is equivalent to require the acyclicity  of the mapping cone
$M(\alpha_{E^\point})$ (which is homotopically isomorphic to $\hbox{Ex}(d^0)$):
{\small{\[\xymatrixcolsep{0.7em}\xymatrixrowsep{0.8 em}
\xymatrix{
M(\alpha_{E^\point}):= 
\cdots\ar[r] \ar[d]^{\cong} &  E^{-2}\oplus E^{-1}\ar[r]\ar[d] &
E^{-1}\oplus \im_{\cE}(d^{-1})\ar[r] \ar[d]& \mathop{E^0}\limits^{\point}_{\;}\ar[r] \ar[d]&
 \Coker_{\cE}(d^{-1}) \ar[r] \ar[d]& 0 \ar[r] \ar[d]& 
\cdots
 \\  
\hbox{Ex}(d^0):= \cdots\ar[r] &0\ar[r] & \im_{\cE}(d^{-1})\ar[r]  & \mathop{E^0}\limits^{\point}_{\;}\ar[r] 
& \Coker_{\cE}(d^{-1}) \ar[r] & 0 \ar[r] &\cdots \\
}\]}}
Hence we would like to use a null system containing 
the complexes $\hbox{Ex}(d^0)$ for any $d^0:E^0\to E^1$
which is possible if and only if these short exact sequences
satisfy the axioms of a Quillen exact structure.
Therefore if
$\cE$ is a $1$-quasi-abelian category 
the previous truncation functors induce, by
 \cite[Lem.~1.2.17; 1.18]{Schn} (see
Proposition~\ref{Prop:Schn} and Lemma~\ref{L:toQES}), 
the $t$-structure $\cL\cD_\cE$ (resp. $\cR\cD_\cE$)
in the derived category
$D(\cE)=K(\cE)/\cN$. Moreover since 
$0\to \im_{\cE}(d^{-1})\to E^0 \to \Coker_{\cE}(d^{-1})\to 0$
 is a kernel-cokernel exact sequence it
 is exact for the
maximal Quillen exact structure on $\cE$
we deduce  
$\cR\cD_\cE^{\leq -1}\subseteq\cL\cD_\cE^{\leq 0}\subseteq \cR\cD_\cE^{\leq 0}$
and $\cE=\cL\cD_\cE^{\leq 0}\cap\cR\cD_\cE^{\geq 0}$.
The $t$-structure $\cL\cD_\cE$ (resp. $\cR\cD_\cE$) is 
called the \emph{left $t$-structure} 
(resp. the \emph{right $t$-structure}), 
whose aisle $\cL\cD_\cE^{\leq 0}$ (resp. co-aisle $\cR\cD_\cE^{\geq 0}$)
is the class of complexes isomorphic
in $D(\cE)$ to complexes whose entries in each strictly positive (resp. negative)
degree are zero. 
The heart of $\cL\cD_\cE$  (resp. $\cR\cD_\cE$)
is denoted by $\cL\cH(\cE)$ (resp.  $\cR\cH(\cE)$) and we denote by $I_\cL$  
(resp. $I_\cR$) the canonical embedding into
 $\cL\cH(\cE)$ (resp. $\cR\cH(\cE)$)
 \[
 \begin{matrix}
I_\cL:& \cE& \longrightarrow &\cL\cH(\cE) \qquad\\
 & E& \longmapsto & \begin{smallmatrix} 0\to & \stackrel{\point}E \qquad\\
 \end{smallmatrix}
 \end{matrix}
 \qquad
 \begin{matrix}
I_\cR:& \cE& \longrightarrow &\cR\cH(\cE) \qquad\\
 & E& \longmapsto & \begin{smallmatrix}  \stackrel{\point}E\to & 0 \qquad\\
 \end{smallmatrix}
 \end{matrix}
 \]
which preserves and reflects exact sequences. Moreover $\cE$ is stable under
extensions in $\cL\cH(\cE)$ (resp. $\cR\cH(\cE)$).

\begin{proposition}
Let $\cE$ be a $1$-quasi-abelian category.
The $t$-structures $\cL\cD_{\cE}=\cR\cD_{\cE}$ coincide if and only if
$\cE$ is an abelian category.
\end{proposition}
\begin{proof}
Let $\cE$  be a $1$-quasi-abelian category,
$\cL\cD_{\cE}=\cR\cD_{\cE}$ if and only if
for any complex
$E^\point\in D(\cE)$ 
the canonical map
$\beta_{E^\point}:\tau_\cL^{\leq 0}E^\point\to \tau_\cR^{\leq 0} E^\point\cong 
\tau_\cL^{\leq 1}\tau_\cR^{\leq 0} E^\point$
{\small{\[\xymatrixcolsep{1.6em}\xymatrixrowsep{0.8em}
\xymatrix{
\tau_\cL^{\leq 0} E^\point :=
\cdots \ar[r]\ar[d]_{\beta_{E^\point}} &  E^{-1} \ar[r]\ar[d] 
&\mathop{\Ker_{\cE}d^0}\limits^{\point}_{\;} \ar[r]\ar[d] & 0 \ar[r]\ar[d]&0 \ar[r]\ar[d] &
\cdots \\ 
 \tau_\cL^{\leq 1}\tau_\cR^{\leq 0} E^\point=
\cdots\ar[r] &   E^{-1}\ar[r]^{d^{-1}} &\mathop{E^0}\limits^{\point}_{\;}
\ar[r]& \im_{\cE}(d^0)\ar[r] & 0\ar[r] & \cdots
\\ }
\]}}is an isomorphism in $D(\cE)$ which holds true if and only if
the short sequence
$0\to \Ker_{\cE}d^0 \to E^0\to \im_{\cE}(d^0)\to 0$ is exact 
on $\cE$ i.e.;
$\cE$ is an abelian category.
\end{proof}

%
\begin{theorem}\label{Th:1tilt}\cite{Rual}, \cite[Prop. B.3]{BonVdBergh}.
Let $\cE$ be an additive category. The following properties are equivalent:
\begin{enumerate}
\item $\cE$ is a $1$-cotilting torsion-free class in an abelian category $\cA$;
\item $\cE$ is a $1$-tilting torsion class in an abelian category $\cA'$;
\item $\cE$  is a $1$-quasi-abelian category;
\item $\cE$ is the intersection of the hearts $\cH_\cD\cap\cH_\cT$
of a $1$-tilting pair of $t$-structures.
\end{enumerate}
Moreover $\cA\simeq \cL\cH(\cE)$, $\cA'\simeq\cR\cH(\cE)$ and $(\cD,\cT)=(\cR\cD_\cE,\cL\cD_\cE)$.
\end{theorem}
%
\begin{proof}
The equivalence between 
(1), (2) and (4)
is a consequence of  Theorem~\ref{Th:1--tilt} and
Proposition~\ref{P:1t}.
Given $\cE$ a $1$-quasi-abelian category as recovered in
\ref{num:lefttstr}
Schneiders proved that $(\cR\cD_\cE,\cL\cD_\cE)$ is a $1$-tilting
pair of $t$-structures with $\cL\cH(\cE)\cap\cR\cH(\cE)\simeq \cE$, so (3) implies (4).
On the other direction given any $1$-tilting pair of $t$-structures $(\cD,\cT)$
by Proposition~\ref{P:1t} 
the class $\cE:=\cH_\cT\cap\cH_\cD$ is a tilting torsion class in $\cH_\cD$, 
hence a $1$-quasi-abelian category and thus (4) implies (3).
\end{proof}

We have seen in \ref{num:torpair} that given any torsion pair
 $(\cX,\cY)$ in an abelian category $\cA$ both $\cX$ and $\cY$
 are $1$-quasi-abelian categories, so in particular $\cX$ is a $1$-tilting torsion class
 after a suitable replacement of the abelian category:

\begin{proposition}\label{P:t=tt}
Let $(\cX,\cY)$ be any torsion pair in an abelian category $\cA$.
Let consider $\cA_\cX$ to be the full subcategory of $\cA$ whose
objects are cogenerated by $\cX$. Then $\cA_\cX$ is abelian, the canonical
embedding functor $\cA_\cX\hookrightarrow \cA$ is exact and the pair $(\cX, \cY\cap\cA_\cX)$ is a $1$-tilting 
torsion pair in $\cA_\cX$ therefore $\cA_\cX\simeq \cR\cH(\cX)$.\\
Dually let  consider $\cA_\cY$ to be the full subcategory of $\cA$ whose
objects are generated by $\cY$. Then $\cA_\cY$ is abelian, the
functor $\cA_\cY\hookrightarrow \cA$ is exact and the pair $(\cX\cap\cA_\cY, \cY)$ 
is a $1$-cotilting torsion pair
in $\cA_\cY$ therefore $\cA_\cY\simeq\cL\cH(\cY)$.
\end{proposition}
%
\begin{proof}
Let us  prove that for any $X\stackrel{f}\to Y$ morphism in $\cA_\cX$,
its kernel and cokernel  in $\cA$ belong to $\cA_\cX$. 
By definition of $\cA_\cX$ there exist 
$X\stackrel{\alpha_X}\hookrightarrow T_X$ and 
$Y\stackrel{\alpha_Y}\hookrightarrow T_Y$ 
with $T_X,T_Y$ in $\cX$. 
Hence
$ \Ker_\cA(f)\hookrightarrow X\stackrel{\alpha_X}\hookrightarrow T_X$ implies 
$ \Ker_\cA(f)\in \cA_\cX$ while 
$ \Coker_\cA(f)\hookrightarrow \Coker_{\cA}(\alpha_Y f)\in\cX$, since $\cX$ is 
closed under quotients and  $T_Y\in\cX$.
Let $X\in\cA_\cX$ and let consider its short exact sequence
$0\to T(X)\to X\to F(X)\to 0$ where $T(X)$ (resp. $F(X)$) is its torsion 
(resp. torsion-free) part with respect to the torsion pair $(\cX,\cY)$ in $\cA$.
Then $T(X)\in\cX\subseteq \cA_\cX$, hence 
$F(X)\in\cA_\cX$ (since it is a cokernel 
of a morphism in $\cA_\cX$) which proves that $(\cX,\cY\cap\cA_\cX)$ is a torsion
pair in $\cA_\cX$.
The second statement follows dually.
\end{proof}

%

\section{$n$-Tilting Theorem}

\smallskip

\numero
\label{n:lista}
Let $\cC$ be a triangulated category endowed with a pair of  $t$-structures $(\cD,\cT)$:
$\cD^{\leq -n}\subseteq \cT^{\leq 0}\subseteq \cD^{\leq 0}$ and
$\cE:=\cH_\cT\cap\cH_\cD$.
The following statements hold true:
\begin{enumerate}
\item
any complex 
$\cdots\to0\to E^{-s}\to \cdots \to E^{-1}\to \stackrel{\point}{E^0}\to 0\to \cdots$
with $s\geq 0$
belongs to $\cT^{[-s,0]}\cap\cD^{[-s,0]}$ (\cite[Lem.~1.1]{FMS});
\item if $n\geq 1$,
given an exact sequence in $\cH_\cD$  (resp. $\cH_\cT$)
\[\xymatrixrowsep{0.5 em}\xymatrix@-7pt{
0\ar[r] & M\ar[r]^{g} & E_{-n+1}\ar[r]^{d_E^{-n+1}} & \cdots \ar[r]^{d_E^{-1}} & 
E_0\ar[r]^{f} & N \ar[r] & 0\\
}\quad E_{-i}\in\cE\quad  \forall i=0,\dots ,n-1\]
implies $N=\Coker_{\cH_\cD}d^{-1}_E\in\cE$
(resp. $M=\Ker_{\cH_\cT}d^{-n+1}_E\in\cE$). 
The argument of \cite[Lem.~1.2]{FMS}
gives a distinguished triangle
$M[n-1]\to [E_{-n+1}\to\cdots \to\stackrel{\point}{E_0}] \to N[0]\stackrel{+}\to$ hence
$M[n-1]\in \cH_\cD[n-1]\subseteq \cT^{\leq 1}$ and 
$[E_{-n+1}\to\cdots\to \stackrel{\point}{E_0}]\in \cT^{\leq 0}$ so
$N[0]\in \cH_\cD\cap  \cT^{\leq 0}=\cE$;
\item a complex $E^\point \in K(\cE)$ is acyclic in $\cH_\cD$
if and only if it is acyclic in $\cH_\cT$ and in this case for any $i$ we have
$\Ker_{\cH_\cD}d^i_{E^\point}\cong \Ker_{\cH_\cT}d^i_{E^\point}\in \cE$
(\cite[Prop.~1.3]{FMS});
\item $\cE$ is projectively complete
(any idempotent in $\cE$ splits in $\cH_\cD$ and it belongs to $\cH_\cT$ too);
$\cE$ is closed under extensions both in $\cH_\cD$ and $\cH_\cT$, hence the class of short exact sequences $0\to E_1\to E \to E_2\to 0$ 
(in $\cH_\cD$ or equivalently $\cH_\cT$)
provides a Quillen exact structure $(\cE,{\cE}x)$ on $\cE$.
\end{enumerate}

\begin{remark}
Let consider $\cC=D(\cH_\cD)$ and $\cE$ a cogenerating class in $\cH_\cD$.
By \cite[Lem.~1.4]{FMS} $\cE$ is generating in $\cH_\cT$ and
by point $(2)$ of \ref{n:lista}  any $A\in\cH_\cD$ admits
a copresentation of length at most $n$. Dually any $B\in\cH_\cT$
has a presentation of length  at most $n$.
\end{remark}

All the previous results combine into the following $n$ version of 
Theorem~\ref{Th:1--tilt}

\begin{theorem}{\bf n-Tilting Theorem.} (\cite[Th.~1.5]{FMS}  \label{Th:n--tilt})
Let $\cA$ be abelian category
such that its derived category $D(\cA)$ has Hom sets, let $\cD$ be the natural 
$t$-structure in $D(\cA)$ and  $\cT$ a 
$t$-structure such that $\cD^{\leq -n}\subseteq \cT^{\leq 0}\subseteq \cD^{\leq 0}$. 
Let us suppose that
 $\cE:=\cA\cap\cH_\cT$ cogenerates $\cA$, hence
there exists a triangle equivalence $E$:
\[
\xymatrixrowsep{1.2 em}
\xymatrix{
 & {K(\cE)\over\cN_{{\cE}x}}\ar@{_(->}[ld]_{\cong} \ar@{^(->}[rd]^{\cong} & \\
D(\cH_{\cT})\ar@{..>}[rr]^{\cong}_E& & D(\cA) \\
}\](where $\cN_{{\cE}x}$ is the null system of complexes in $K(\cE)$ acyclic in $\cA$ or equivalently in
$\cH_\cT$)  such that the restriction of $E$ to $\cH_\cT$ is naturally isomorphic to 
the inclusion 
$\cH_\cT\subseteq D(\cA)$. Moreover $\cE$ is generating in $\cH_\cT$.
\end{theorem}
%

\begin{definition}\label{D:ntpt}
A \emph{pair of $t$-structures}
$(\cD,\cT)$ on a triangulated category $\cC$ is called \emph{$n$-tilting} if
the following statements hold:
\begin{enumerate}
\item $\cD^{\leq -n}\subseteq \cT^{\leq 0}\subseteq \cD^{\leq 0}$;
\item the following equivalent conditions are satisfied:
\begin{description}
\item[(i)] given $\cE:=\cH_\cD\cap\cH_\cT$ we get $\cC\simeq  K(\cE)/\cN$ and
$D(\cH_\cD)\stackrel{\simeq}\hookleftarrow K(\cE)/\cN
\stackrel{\simeq}\hookrightarrow D(\cH_\cT)$
where $\cN$ is the null system of complexes in 
$K(\cE)$ acyclic in $\cH_{\cD}$
 or equivalently in
$\cH_\cT$;
\item[(ii)]  $\cC\simeq D(\cH_\cD)$ and $\cE$ cogenerates $\cH_\cD$;
\item[(iii)] $\cC\simeq D(\cH_\cT)$ and $\cE$ generates $\cH_\cT$.
\end{description}
\end{enumerate}
\end{definition}

If $\cD^{\leq -n}\subseteq \cT^{\leq 0}\subseteq \cD^{\leq 0}$ by 
Theorem~\ref{Th:n--tilt} we have that $(ii)$ 
 implies $(i)$ and $(iii)$, dually $(iii)$ 
 implies $(i)$ and $(ii)$ (by the cotilting version of Theorem~\ref{Th:n--tilt}) so $(ii)$ is equivalent to $(iii)$.
If $(i)$ holds $\cC\simeq D(\cH_\cD)$ and  $\cE$ cogenerates $\cH_\cD$
since any $A\in \cH_\cD$ can be represented by a complex $E^\point\in K(\cE)$
and so $A\hookrightarrow \Coker_{\cH_\cD}(d^{-1}_{E^\point})\in \cE$ 
($ \Coker_{\cH_\cD}(d^{-1}_{E^\point})\in \cE$ by (2) of \ref{n:lista}) which proves that $(i)$ implies $(ii)$.

We note that any $n$-tilting pair of $t$-structures is
also $m$-tilting for any $m\geq n$.

\begin{proposition}\label{P:DT=RL}
Let $(\cD,\cT)$ be a $n$-tilting pair of $t$-structures in 
a triangulated category
$\cC$. Hence the equivalence
$F:\cC\stackrel{\simeq}\to 
K(\cE)/\cN_{{\cE}x}=D(\cE,{\cE}x)$
(where the Quillen exact structure on $\cE$ is the one of \ref{n:lista} (4))
gives:
\begin{align*}
F(\cT^{\leq 0})=&\{X^\point\in K(\cE)\; |\; X^\point\cong E^\point_{\leq 0}\; \hbox{in }
K(\cE)/\cN_{{\cE}x} \hbox{ with }
E^\point_{\leq 0}\in \cL\cK_\cE^{\leq 0}\}=:\cL\cD_{(\cE,{\cE}x)}^{\leq 0}\\
F(\cD^{\geq 1})=&\{X^\point\in K(\cE)\; |\; X^\point\cong E^\point_{\geq 1}\; \hbox{in }
K(\cE)/\cN_{{\cE}x} \hbox{ with }
E^\point_{\geq 1}\in \cR\cK^{\geq 1}\}=:\cR\cD_{(\cE,{\cE}x)}^{\geq 1}.\\
\end{align*}
\end{proposition}
\begin{proof}
By definition $D(\cE,{\cE}x)={K(\cE)\over{\cN_{{\cE}x}}}$ (as introduced by Neeman 
in \cite{NeeDEC} see 
Appendix~\ref{N3.2}).
Since $(\cD,\cT)$ is $n$-tilting ,
we have that under the $n$-tilting equivalence
$D^{\leq 0}(\cH_\cT)$ corresponds to $\cT^{\leq 0}$, while $\cD^{\geq 1}$
corresponds to $D^{\geq 1}(\cH_\cD)$.
Moreover the class $\cE$ generates $\cH_\cT$ and so any object in 
$D^{\leq 0}(\cH_\cT)$ can be represented in $ K(\cE)/\cN_{{\cE}x}$ by a complex in $K^{\leq 0}(\cE)$.
On the other side since $\cE$ cogenerates $\cH_\cD$  any object in 
$D^{\geq 1}(\cH_\cD)$ can be represented in $ K(\cE)/\cN_{{\cE}x}$ by a complex in $K^{\geq 1}(\cE)$.
\end{proof}

\begin{remark}\label{R:m>n}
The proof of Theorem~\ref{Th:n--tilt} produces the desired equivalence on the
derived categories of the hearts passing trough an equivalence with
the triangulated category ${K(\cE)\over\cN_{{\cE}x}}=D(\cE,{\cE}x)$ where 
$\cE$ is the intersection of the hearts.
We remark that the role of the Quillen exact
structure is important in order to define $D(\cE,{\cE}x)$.
The previous proposition proves that the category $\cE$ encodes the data of the $t$-structures
since $(\cD,\cT)\simeq (\cR\cD_{(\cE,{\cE}x)},\cL\cD_{(\cE,{\cE}x)})$.
\end{remark}

\section{$2$-tilting torsion classes}\label{S:2}
As we will see soon the case $n=2$ is neatly easier than $n>2$ and so
we will first analyze this case in detail. 

\begin{lemma}\label{L:3.1}
Let $(\cD,\cT)$ be a $2$-tilting pair of $t$-structures in $\cC\simeq D(\cH_\cD)\simeq D(\cH_\cT)$.
Hence $\cE:=\cH_\cD\cap\cH_\cT$ is closed under extensions (both in  $\cH_\cD$ and $\cH_\cT$);
it  admits kernels and cokernels and
given $d:E\to F$ in $\cE$ we have 
$\Ker_{\cE}(d)=\Ker_{\cH_\cT}(d)\in\cE$ while 
$\Coker_{\cE}(d)=\Coker_{\cH_\cD}(d)\in\cE$.
Moreover the inclusion functor $i:\cE\hookrightarrow \cH_{\cD}$ admits a right adjoint $t:\cH_{\cD}\to \cE$
while the inclusion functor $j:\cE\hookrightarrow \cH_{\cT}$ admits a left adjoint $f:\cH_{\cT}\to \cE$.
\end{lemma}
%
\begin{proof}
Let $d:E\to F$ be a morphism in $\cE$, 
by point $(2)$ of \ref{n:lista} we have: $\Ker_{\cH_\cT}d\in\cE$  while $\Coker_{\cH_\cD}d\in\cE$ and so they provide   the kernel, resp. the cokernel, of $d$ in $\cE$. 
By hypothesis
$\cD^{\leq -2}\subseteq \cT^{\leq 0}\subseteq \cD^{\leq 0}$ so, by orthogonality  $\cT^{\geq 1}\subseteq \cD^{\geq -1}$.
Let $A\in\cH_\cD$, the distinguished triangle $\tau^{\leq 0}(A)\to A \to \tau^{\geq 1}(A){\stackrel{+}\to} $
proves that $\tau^{\leq 0}(A)\in\cD^{\geq 0}$ since $A\in\cH_{\cD}$
and $ \tau^{\geq 1}(A)\in \cT^{\geq 1}\subseteq \cD^{\geq -1}$ so
$t(A):=\tau^{\leq 0}(A)\in\cT^{\leq 0}\cap\cD^{\geq 0}=\cE$ (recall notation \ref{N:tf}).
Hence 
$\cH_\cD(i(E),A)=\cC(E,A)\simeq\cC(E,\tau^{\leq 0}A)=\cE(E,t(A))$
for any $E\in\cE$, which proves that $t$ is a right adjoint of $i$.
Dually the functor $\delta^{\geq 0}$ restricted to $\cH_\cT$ takes image
in $\cE$ and provides the left adjoint $f$ of $j$.
\end{proof}
%
 
 
%

At the $2$-level a tilting torsion class is given by the $2$-version of the characterization of Lemma~\ref{L:eq1ttc}:
\begin{definition}\label{D:2tiltorcl}
Let $\cA$ be an abelian category.
A full subcategory $\cE\hookrightarrow \cA$ is a
 \emph{$2$-tilting torsion class} if
\begin{enumerate}
\item $\cE$ cogenerates $\cA$;
 \item $\cE$ is closed under extensions in $\cA$;
\item $\cE$ has kernels;
\item for any exact sequence  $0\to A\to X_1\to X_2\to B\to 0$ in $\cA$ with $X_i\in\cE$ 
for $i\in\{1,2\}$ and $A,B\in\cA$ we have $B\in\cE$.
\end{enumerate}
Moreover $\cE$ is endowed with a canonical Quillen
exact structure whose short exact sequences are exact sequences in  $\cA$ with terms in $\cE$. Any $1$-tilting torsion class as in
Definition~\ref{D:tilttorpair} is also a $2$-tilting torsion class.

Dually a \emph{$2$-cotilting torsion-free class} in $\cA$ is  a full generating 
extension closed subcategory $\cE$ of $\cA$
admitting cokernels and 
closed under  kernels in $\cA$.
\end{definition}
%

\begin{proposition}\label{P:2t}
Given $(\cD,\cT)$ a $2$-tilting pair of $t$-structures the category
$\cE:=\cH_\cD\cap\cH_\cT$ is a $2$-tilting torsion class 
(resp. a $2$-tilting torsion-free class) in $\cH_\cD$
(resp. in $\cH_\cT$).
\end{proposition}
\begin{proof}
By Definition~\ref{D:ntpt}  $\cE$ cogenerates $\cH_\cD$
and generates $\cH_\cT$. By point (4) of \ref{n:lista}
we have that $\cE$ is closed under extensions both in 
$\cH_\cD$ and $\cH_\cT$.
Given a morphism $d:X_{1}\to X_2$ in $\cE$ by point (2) of 
\ref{n:lista} we deduce that $\Ker_{\cE}d\cong \Ker_{\cH_\cT}d\in\cE$
and $\Coker_{\cE} d\cong\Coker_{\cH_\cD}d\in\cE$ 
which concludes the proof.
\end{proof}
%



\begin{theorem}\label{T:3.5}
Let $\cA$ be an abelian category and let $\cD$ be the natural $t$-structure on the triangulated category $D(\cA)$.
Let  $i:\cE\hookrightarrow \cA$ be a $2$-tilting torsion class on $\cA$.
Hence $\cT^{\leq 0}:=\cD^{\leq -2}\star \cE\star \cE[1]$ is an aisle in $D(\cA)$ (see \ref{B2.1})
such that $\cE=\cA\cap\cH_\cT$ and the
pair $(\cD,\cT)$ is a $2$-tilting pair of $t$-structures.
We will say that the $t$-structure $\cT$ is obtained by \emph{tilting}
$\cD$ with respect to the $2$-tilting torsion class $\cE$.
\end{theorem}
%
\begin{proof}
We will prove in Lemma~\ref{L:ec} that
$\cT^{\leq 0}$ is extension closed;
 $\cT^{\leq 0}[1]\subseteq \cT^{\leq 0}$
since the suspension of any factor is contained in a factor.
By definition $\cD^{\leq -2}\subseteq \cT^{\leq 0}$,
since any factor is contained in $\cD^{\leq 0}$ (which is extension closed)
we have $\cT^{\leq 0}\subseteq \cD^{\leq 0}$.
Let us prove that the functor 
$i_{\cT^{\leq 0}}:\cT^{\leq 0}\to D(\cA)$ has a  right adjoint 
$\tau^{\leq 0}:D(\cA)\to\cT^{\leq 0}$. 
We will first define the restriction $\tau^{\leq 0}_{|\cD^{[-1,0]}}$,
next
we will prove how to extend this  functor
to the whole $D(\cA)$. 
Let us notice that the functor $i:\cE\to\cA$ has a right adjoint  $t$ defined as in 
Lemma~\ref{L:eq1ttc}:
for any $A\in\cA$ let consider a copresentation
$0\to A\to X_1\stackrel{f}\to X_2$ and let us pose
$t(A)=\Ker_{\cE}(f)$.
We note that for any
$L\in\cD^{\leq -2}\star \cE \star\cE[1]$
we have $H^0_\cD(L)\in\cE$ (since it is a cokernel
in $\cA$ between two objects in $\cE$).
Let $A\in\cA$,  for any $M\in\cT^{\leq 0}$
we have
$D(\cA)(M,A)\cong\cA(H^0_{\cD}(M),A)\cong \cA(H^0_{\cD}(M),t(A))\cong
D(\cA)(M,t(A))$.
So our truncation functor $\tau^{\leq 0}$ restricted
to $\cA$ coincides with $t$: $\tau^{\leq 0}_{|\cA}=t$.
Moreover the argument of \cite[Prop.~1.1]{KeVo} proves that
the mapping cone of the morphism $t(A) \to A$ belongs to $\cT^{\geq 1}$.
We remark that as in \cite[Prop.~1.3.3]{BBD} even if we have to choose a morphism 
in order to define this functor,
the functoriality of the construction is guaranteed by  the fact that for another choice
there exists a unique isomorphism compatible with this construction.

Let us now compute the restriction of $\tau^{\leq 0}$ to $\cD^{[-1,0]}$. 
Given any object in $D\in\cD^{[-1,0]}$ 
there exists $f:A\to B$ in $\cA$ such that 
$D$ is the mapping cone
of the morphism $f:A[0]\to B[0]$ in $D(\cA)$.
Since $\cE$ is cogenerating in $\cA$ there exists an immersion
$h:A\hookrightarrow E$ with $E\in\cE$ and so
$D$ is isomorphic in $D(\cA)$ to the   mapping cone
of $\overline{f}:E\to E\oplus_A B$.
Let define $\tau^{\leq 0}(D)$ to be mapping cone of $E\to t(E\oplus_A B)$. 
Let 
consider the following commutative diagram 
whose rows and columns are distinguished triangles (obtained by applying the
octahedron  axiom to the composition $E\to t(E\oplus_A B)\to E\oplus_A B$ see also  \cite[Prop.~1.1.11]{BBD})
{\small{\[ \xymatrixrowsep{0.9em}
\xymatrix{
E\ar[r]^(0.4){t(\overline{f})} \ar[d] &t(E\oplus_A B)\ar[r]\ar[d] & 
\tau^{\leq 0}(D)\ar[r]^(0.7){+}\ar[d] & \\
E\ar[r]^(0.4){\overline{f}}\ar[d] & E\oplus_A B\ar[r]\ar[d] & 
D\ar[r]^(0.7){+}\ar[d] & \\
0\ar[r]\ar[d]^{+}&\tau^{\geq 1}(E\oplus_A B)\ar[r] \ar[d]^{+}&
\tau^{\geq 1}(E\oplus_A B) \ar[r]^(0.7){+} \ar[d]^{+}& \\
\; & \; & \; & \\
}
\] }}

Since $\tau^{\geq 1}(E\oplus_A B)$ and $\tau^{\geq 1}(E\oplus_A B)[-1]$ belong to $\cT^{\geq 1}$;
for any $M\in \cT^{\leq 0}$ the long exact sequence associated to the right vertical distinguished triangle 
implies $D(\cA)(M,\tau^{\leq 0}(D))\cong D(\cA)(M,D)$.
Moreover $\tau^{\geq 1}(D)\cong \tau^{\geq 1}(E\oplus_A B) \in \cT^{\geq 1}$.

Now we are able to compute the truncation $\tau^{\leq 0}(D)$
for any $D\in\cD^{[-1,0]}$.  Since $\cT^{\leq 0}\subseteq \cD^{\leq 0}$
we have $\tau^{\leq 0}(X)\cong \tau^{\leq 0}(\delta^{\leq 0}(X))$
for any $X\in D(\cA)$ (one can see by the octahedron axiom that
the mapping cone of the composition 
$ \tau^{\leq 0}(\delta^{\leq 0}(X))\to\delta^{\leq 0}(X)\to X$ lyes in $\cT^{\geq 1}$).
Given $C\in\cD^{\leq 0}$; the
following commutative diagram (whose rows and columns are
distinguished triangles)
{\small{\[\xymatrixrowsep{0.8em}
\xymatrix{
\delta^{\leq -2}(C)\ar[r] \ar[d] & \tau^{\leq 0}(C)\ar[r]\ar[d] & 
\tau^{\leq 0}(\delta^{[-1,0]}(C))\ar[r]^(0.7){+}\ar[d] & \\
\delta^{\leq -2}(C)\ar[r] \ar[d]& C\ar[r]\ar[d] & \delta^{[-1,0]}(C)\ar[r]^(0.7){+}\ar[d] & \\
0\ar[r]\ar[d]^{+} & \tau^{\geq 1}(\delta^{[-1,0]}(C))\ar[r] \ar[d]^{+}& 
\tau^{\geq 1}(\delta^{[-1,0]}(C))\ar[r]^(0.7){+} \ar[d]^{+}& 
\\
\; & \; & \; & \\
}
\]}}permits us to compute  $\tau^{\leq 0}(C)$ for any $C\in\cD^{\leq 0}$.
We recall that whenever the  two rows and any column of such a digram are distinguished 
the third row is distinguished too \cite{BBD}.
The functoriality of this construction is guaranteed by the orthogonality of 
the classes $\cT^{\leq 0},\cT^{\geq 1}$.

Let us prove that $\cE=\cA\cap\cH_\cT$; we recall that
$\cA\simeq \cH_\cD$.
Let consider $A^\point\in\cA\cap\cH_\cT$,
hence $A^\point \in \cT^{\leq 0}=\cD^{\leq -2}\star\cE\star\cE[1]$ and so 
it fits into a distinguished triangle
$B^\point \to A^\point \to E^{[-1,0]}{\stackrel{+}\to}$ for suitable
$B^\point\in\cD^{\leq -2}$ and $ E^{[-1,0]}\in\cE\star\cE[1]$; but
since $A^\point\in\cD^{\geq 0}$
we deduce that $B^\point\in \cD^{\leq -2}\cap\cD^{\geq 0}={0}$ so
 $A^\point\in  \cE\star\cE[1]$.
Therefore $A^\point= [E^{-1}\stackrel{d}
\to\stackrel{\point}{E^0}]$
and $A^\point \cong H^0_\cD(A^\point)$ since $A^\point\in\cA\simeq\cH_\cD$, so  by
point (4) of Definition~\ref{D:2tiltorcl} we obtain $A^\point\in\cE$ which proves that
 $\cT^{\leq 0}\cap\cD^{\geq 0}=\cE$. 
  We can apply the
Tilting Theorem~\ref{Th:n--tilt} ($\cE$ cogenerates $\cA$) thus obtaining that $(\cD,\cT)$ 
is a $2$-tilting pair of $t$-structures.
\end{proof}
\begin{lemma}\label{L:ec}
Let $\cE$ be a $2$-tilting torsion class in an abelian category 
$\cA$. The full subcategory 
$\cT^{\leq 0}:=\cD^{\leq -2}\star \cE\star \cE[1]$
of $D(\cA)$ is  closed under extensions.
\end{lemma}
\begin{proof}
Let us denote by $\cD$ the natural $t$-structure on $D(\cA)$.
We refer to Appendix~\ref{B2.1} for the definition of $\star$.

{\it Step 1.}
Let us prove that $\cE[1]\star \cE\subseteq  \cE\star \cE[1]$.
Any $X^\point\in \cE[1]\star \cE\subseteq D^{[-1,0]}(\cA)$ 
can be represented by a complex $E^\point\in K^{\geq -1}(\cE)$
(since $\cE$ cogenerates $\cA$);
hence $E^\point$ fits in the distinguished triangle 
$H^{-1}_{\cD}(E^\point)[1]\to E^\point \to H^{0}_{\cD}(E^\point)\stackrel{+}\to $.
Since this is the unique triangle realizing $E^\point\in \cA[1]\star \cE$, we have
$H^{-1}_{\cD}(E^\point)=\Ker_{\cA}(d^{-1}_{E^\point})\in \cE$
and $H^0_{\cD}(E^\point)={\Ker_{\cA}(d^{0}_{E^\point})\over \im_{\cA}(d^{-1}_{E^\point})}\in \cE$.
The short exact sequence 
$0\to \Ker_{\cA}(d^{-1}_{E^\point})\to E^{-1}\to \im_{\cA}(d^{-1}_{E^\point})\to 0$
implies that $ \im_{\cA}(d^{-1}_{E^\point})\in\cE$ (by property (4) of 
Definition~\ref{D:2tiltorcl}) and
so $\Ker_{\cA}(d^{0}_{E^\point})\in \cE$ (since it is an extension of objects in $\cE$).
This proves that 
$X^\point\cong [E^{-1}\to \stackrel{\point}{\Ker_{\cA}(d^{0}_{E^\point})}]\in\cE\star\cE[1]$.

{\it Step 2.}
Thanks to the previous step,
in order to conclude the proof it is enough to prove that
$(\cE\star\cE[1])\star \cD^{\leq -2}\subseteq \cT^{\leq 0}$.
Let $Z^\point\in (\cE\star\cE[1])\star \cD^{\leq -2}\subseteq \cD^{\leq 0}$
hence $Z^\point$ fits in a distinguished triangle
$[E^{-1}\to \stackrel{\point}{E^0}] \to Z^\point \to Y^\point {\stackrel{+}\to}$
with $Y^\point\in K^{\leq -2}(\cA)$.
We can represent $Z^\point=[\cdots Y^{-3}\stackrel{d^{-3}}\to Y^{-2} \stackrel{d^{-2}}\to
E^{-1}\stackrel{d^{-1}}\to \stackrel{\point}{E^0}\to 0\to \cdots]\in K^{\leq 0}(\cA)$.
Let us prove that $\delta^{\geq -1}Z^\point\in \cE\star\cE[1]$.
Since  $\cE$ cogenerates $\cA$ there exists $F^{-1}\in\cE$ such that
$\Coker d^{-2}_{Z^\point}\stackrel{i}\hookrightarrow F^{-1}$; hence
 the complex
$\delta^{\geq -1}Z^\point=[\cdots 0\to  \Coker d^{-2}_{Z^\point}\to
\stackrel{\point}{E^{0}}\to 0\to \cdots ]\cong
[\cdots 0\to  F^{-1}\to
\stackrel{\point}{F^{-1}\oplus_{E^{-1}}{E^0}}\to 0\to \cdots ]\in\cE\star\cE[1]
$ since $F^{-1}\in \cE$ and by property (4) of 
Definition~\ref{D:2tiltorcl} ${F^{-1}\oplus_{E^{-1}}{E^0}}\in\cE$ too.
Therefore $Z^\point\in\cT^{\leq 0}$ since the triangle
$\delta^{\leq -2}Z^\point\to Z^\point\to \delta^{\geq -1}Z^\point\stackrel{+}\to $ is distinguished.
\end{proof}
%

\begin{remark}
Theorem~\ref{T:3.5} admits a dual version:
given $\cT$ a $t$-structure on $D(\cH_\cT)$ and $j:\cE\to \cH_\cT$ a
$2$-cotilting torsion-free class on $\cH_\cT$,
the class $\cD^{\geq 0}:=\cE[-1]\star\cE\star \cT^{\geq 2}$
is a co-aisle in $D(\cH_\cT)$ such that $\cH_\cT\cap\cH_\cD=\cE$.
Moreover the
pair $(\cD,\cT)$ is a $2$-tilting pair of $t$-structures.
\end{remark}

\begin{definition}\label{D:2qa}
A \emph{$2$-quasi-abelian category} is the data  
$(\cE,{\cE}x)$ of an additive category $\cE$ with a Quillen  exact structure ${\cE}x$ such that
$\cE$
admits kernels and cokernels.
\end{definition}

\begin{remark}\label{R:ttc-2qa}
Clearly any $1$-quasi-abelian category is also
$2$-quasi-abelian.
Any $2$-tilting torsion class $\cE$ is a $2$-quasi-abelian category since by 
Definition~\ref{D:2tiltorcl} (3) it admits kernels and by (4) it admits cokernels.
\end{remark}

 Let us start by studying the case of a $2$-quasi-abelian category $(\cE,{\cE}x_{\rm split})$
whose Quillen exact structure is the minimal one
(i.e., any conflation splits).

\begin{proposition}\label{P:2assts}
Let $\cE$ be an additive category admitting kernels and cokernels.
The category $K(\cE)$ admits 
a canonical $2$-tilting  pair of $t$-structures $(\cR\cK_\cE,\cL\cK_\cE)$
such that $\cE=\cR\cK(\cE)\cap\cL\cK(\cE)$ and so
$\cE\hookrightarrow \cR\cK(\cE)$ is a $2$-tilting torsion class
while $\cE\hookrightarrow \cL\cK(\cE)$ is a $2$-cotilting torsion-free class.
\end{proposition}
%
\begin{proof}
We can endow $\cE$ with its minimal Quillen exact structure (split short exact sequences).
So $(\cE,{\cE}x_{\rm split})$ is a 2-quasi-abelian category whose derived category
$D(\cE,{\cE}x_{\rm split})=K(\cE)$.
In \ref{num:lefttstr} we provided the construction of the left and right $t$-structures on 
$K(\cE)$ for $\cE$ a $1$-quasi-abelian category. 
This construction is based on the existence of kernels and cokernels and so it works
unchanged thus
providing the $t$-structures
$\cL\cK_\cE$ and
$\cR\cK_\cE$   on $K(\cE)$ 
whose associated truncated functors are those 
described in  \ref{num:lefttstr}.
Moreover $\cR\cK_\cE^{\leq -2}\subseteq\cL\cK_\cE^{\leq 0}\subseteq \cR\cK_\cE^{\leq 0}$.
The heart of $\cL\cK_\cE$ (resp. $\cR\cK_\cE$) is denoted by $\cL\cK(\cE)$ 
(resp. $\cR\cK(\cE)$). 
Any short exact sequence $0\to K\stackrel{\alpha}\to L\stackrel{\pi}\to E\to 0$
 in  $\cL\cK(\cE)$ with $E\in\cE$ is a  distinguished triangle in $K(\cE)$
 and it induces the exact sequence 
 $
\cK(\cE)(E,L)\to \cK(\cE)(E,E)\to \cK(\cE)(E,K[1])
 $, since $K[1]$ is a complex in $\cL\cK_\cE^{\leq 0}$ (with $0$ entries in degrees greater or equal to $0$) we get
  $\cK(\cE)(E,K[1])=0$
and so $\pi$ is split 
epimorphism. 
Hence $\cE$ coincides with the class of  projective objects in $\cL\cK(\cE)$. Moreover any object  $L\in \cL\cK(\cE)$ 
can be represented as a complex 
$L\cong C(d):=[\Ker_{\cE}(d)\stackrel{\alpha}\to X \stackrel{d}\to \stackrel{\point}Y]\in K(\cE)$
(the other entries of the complex are $0$ see Appendix~\ref{A:not})
since $L\cong \tau_{\cL}^{\geq 0}\tau_{\cL}^{\leq 0} L$
(see \ref{num:lefttstr}).
Thus $L$
has a projective resolution of at most length $2$
in the following way:
the distinguished triangles (where $ C(\alpha):=[\Ker_{\cE}(d)\stackrel{\alpha}\to \stackrel{\point}X]$)
\begin{equation}\label{E:DT}
\Ker_{\cE}(d)[0]\longrightarrow X[0]\longrightarrow C(\alpha){\stackrel{+}\to} 
\qquad
C(\alpha) \longrightarrow Y[0] \longrightarrow C(d){\stackrel{+}\to} \quad
\hbox{ in }
\cK(\cE)
\end{equation}
give the short exact sequences
\[\small{0\to \Ker_{\cE}(d)\to X\to C(\alpha) \to 0}
\quad
\small{0\to C(\alpha)\to Y \to C(d)\to 0\; ;}
\quad
\hbox{ in }
\cL\cK(\cE)
\]
from which we obtain 
the  projective resolution $0\to \Ker_{\cE}(d)\to X\to Y \to C(d)\to 0$
of $C(d)$ in $\cL\cK(\cE)$. 
By Lemma~\ref{L2.2} we have $K(\cE)\simeq D(\cL\cK(\cE))$
which proves that $(\cR\cK_\cE,\cL\cK_\cE)$ is a $2$-tilting pair of $t$-structures, hence by Proposition~\ref{P:2t}
$\cE$ is a $2$-tilting torsion (resp. $2$-cotilting torsion-free) class in $\cR\cK(\cE)$
(resp. $\cL\cK(\cE)$). In particular $\cE$ coincides with the class of injective objects in
$\cR\cK(\cE)$ (resp. projective objects in $\cL\cK(\cE)$).
\end{proof}
%


\begin{corollary}\label{C:L=rcoh}
Let $(\cE,{\cE}x_{\rm split})$ be a $2$-quasi-abelian category 
(with its minimal Quillen exact structure).
Hence 
$\cL\cK(\cE)\simeq \rcoh\cE$ and $ \cR\cK(\cE)\simeq (\cE\lcoh)^\circ $.
\end{corollary}
\begin{proof}
Since $\cE$ has kernels and cokernels it is a coherent category
(see Definition~\ref{D1.3} and Proposition~\ref{P1.3}).
Both $\rcoh\cE$ and $\cL\cK(\cE)$ are abelian categories 
whose projective objects coincide with $\cE$ and such that
any object has a projective resolution of at most length $2$.
The functor $I_{\cL}:\cE\to \cL\cK(\cE)$ extends uniquely to a functor
$I_{\cL}^c:\rcoh\cE\to \cL\cK(\cE)$ cokernel preserving
(see \ref{N:A.8}) which is an equivalence of categories
(any object in $L\in\cL\cK(\cE)$ has a projective resolution
therefore $I_{\cL}^c$ is essentially surjective and
fully faithful since any object in $\cE$ is projective in $\cL\cK(\cE)$).
Thus the left heart is equivalent to the category of right coherent functors.
The right statement follows dually.
\end{proof}
%
%
%

Let us now turn to the case of a general $2$-quasi-abelian category $(\cE,{\cE}x)$:

\begin{lemma}\label{L:toQES}
Given any $2$-quasi-abelian category $(\cE,{\cE}x)$  the left and right $t$-structu\-res on $\cK(\cE)$
induce a $2$-tilting pair of $t$-structures $(\cR\cD_{(\cE,{\cE}x)},\cL\cD_{(\cE,{\cE}x)})$ on the
derived category $D(\cE,{\cE}x)$ such that $\cE=\cR\cH(\cE,{\cE}x)\cap\cL\cH(\cE,{\cE}x)$.
\end{lemma}
%
\begin{proof}
Let us denote by $\cN_{{\cE}x}$ 
the null system  of acyclic complexes with respect to 
 $(\cE,{\cE}x)$ (see Definition~\ref{D3.1}).
Let us prove that the $t$-structure 
$\cL\cK_\cE$ on $K(\cE)$ satisfies the hypothesis of 
Proposition~\ref{Prop:Schn} thus inducing (passing trough the quotient)
the $t$-structure $\cL\cD_{(\cE,{\cE}x)}$ on $D(\cE,{\cE}x):=K(\cE)/\cN_{{\cE}x}$.
We have to prove that
given any distinguished triangle
$Y^\point \to X^\point \to N^\point {\stackrel{+}\to}$ in $K(\cE)$
such that $Y^\point \in \cL\cK_\cE^{\geq 1}$,
$X^\point \in  \cL\cK_\cE^{\leq 0} $
and $N^\point\in \cN_{{\cE}x}$
we have $Y^\point, X^\point \in\cN_{{\cE}x}$.
We can suppose $Y^\point=\tau_{\cL}^{\geq 1}Y^\point$
and $X^\point\in K^{\leq 0}(\cE)$.
Let consider the following commutative diagram:
{\tiny{\[
\xymatrixcolsep{0.01in}\xymatrixrowsep{0.7em}\xymatrix{
 \ar[r]& 0\ar[rr]\ar[d] & &  \Ker(d^0_Y)\ar[rr]\ar[d] & & 
Y^0\ar[rr]^{d_Y^0}\ar[d] &&Y^{1}\ar[d]   \\
 \ar[r] &X^{-2}\ar[rr]^{d_X^{-2}}\ar[d] & & 
X^{-1}\ar[rr]^{d_X^{-1}}\ar[d]& & X^0\ar[rr]\ar[d]& &0\ar[d] 
 \\
 \ar[r]  &\Ker(d^0_Y)\oplus X^{-2}\ar[rr]\ar@{->>}[rd]& & Y^0\oplus X^{-1}\ar[rr] \ar@{->>}[rd]
& &Y^1\oplus X^0\ar[rr]\ar@{->>}[rd]
&& Y^2  \\
 \im(d^{-3}_X)\ar@{>->}[ru] 
& &  \Ker(d^{0}_Y)\oplus \Coker(d^{-3}_X)\ar@{>->}[ru] & & 
\Ker(d^1_Y)\oplus X^0\ar@{>->}[ru]&
& \Ker(d^2_Y) \ar@{>->}[ru]&   \\
}
\]}}one has to start looking the last row,
for $i\leq -3$ we have $N^i=X^i$ while for $j\geq 1$
we have $N^j=Y^{j+1}$ so we can write $\im(d^{-3}_X)$
on the left and $\Ker(d^2_Y)$ on the right, hence we complete taking
resp. the cokernel and the kernel and we are able to decompose
$Y^\point$ and $X^\point$ via conflations.
The following pullback diagram 
$$
{\small{\xymatrixrowsep{1.5em}\xymatrix{
\Ker(d^{0}_Y)\oplus \Coker(d^{-3}_X)\ar@{>->}[r] &
Y^0\oplus X^{-1}\ar@{->>}[r]& \Ker(d^1_Y)\oplus X^0\\
\Ker(d^{0}_Y)\oplus \Coker(d^{-3}_X)\ar[u]\ar@{>->}[r] &
\Ker(d^{0}_Y)\oplus X^{-1}\ar[u]\ar@{->>}[r]& X^0\ar[u]\\
}}}
$$ 
permits to conclude that $\xymatrix{
\Coker(d^{-3}_X)\ar@{>->}[r] &X^{-1}\ar@{->>}[r]& X^0 }$
is a conflation, hence
 $X^\point  \in\cN_{{\cE}x}$ which implies that
$Y^\point \in\cN_{{\cE}x}$ too.

Therefore we obtain a pair of $t$-structures
$(\cR\cD_{(\cE,{\cE}x)},\cL\cD_{(\cE,{\cE}x)})$ on  $D(\cE,{\cE}x)$ 
such that 
$\cR\cD_{(\cE,{\cE}x)}^{\leq -2}\subseteq \cL\cD_{(\cE,{\cE}x)}^{\leq 0}
\subseteq \cR\cD_{(\cE,{\cE}x)}^{\leq 0}$.

Clearly $\cE\subseteq \cR\cH(\cE,{\cE}x)\cap\cL\cH(\cE,{\cE}x)$.
If $E^\point\in \cR\cH(\cE,{\cE}x)\cap\cL\cH(\cE,{\cE}x)$, we can suppose
$E^\point\in K^{\leq 0}(\cE)$ and,   since it belongs to $ \cR\cH(\cE,{\cE}x)$, we have that
$E^\point \to \tau_{\cR}^{\geq 0}E^\point=\Coker_{\cE}d^{-1}_{E^\point}[0]\in\cE$
is a quasi-isomorphism (i.e., its mapping cone belongs to $\cN_{{\cE}x}$)
and so $E^\point\in\cE$.

It remains to prove that the derived category of the heart is equivalent to
$D(\cL\cH(\cE,{\cE}x))\simeq D(\cE,{\cE}x)=:K(\cE)/\cN_{{\cE}x}$.
Now $\cE$ is a full subcategory of $\cL\cH(\cE,{\cE}x)$ and a sequence
$S:0\to E_1\to E \to E_2\to 0$ is exact for the Quillen exact structure $(\cE,{\cE}x)$
if and only if the triangle $E_1[0]\to E[0]\to E_2[0]{\stackrel{+}\to}$ is distinguished in
$D(\cE,{\cE}x)$, hence (since any term is in $\cL\cH(\cE,{\cE}x)$)
if and only if $S$ is exact in $\cL\cH(\cE,{\cE}x)$.
We note that given any morphism $f:E\to F$ in $\cE$ 
we have $\Ker_{\cL\cH(\cE,{\cE}x)}(f)=H^{0}_{\cL\cH(\cE,{\cE}x)}
([\stackrel{\point}E\to F])\in \cE$ (due to the inclusion 
$\cR\cD_{(\cE,{\cE}x)}^{\leq -2}\subseteq \cL\cD_{(\cE,{\cE}x)}^{\leq 0}$)
and so $\Ker_{\cL\cH(\cE,{\cE}x)}(f)=\Ker_{\cE} (f)$.
Hence any 
 complex in $K(\cE)$ which is acyclic in $D(\cL\cH(\cE,{\cE}x))$ 
 can be decomposed into short exact sequences in $\cL\cH(\cE,{\cE}x)$
 whose terms belong to $\cE$ and so we deduce that
$\cN_{{\cE}x}=\cN_{\cL\cH(\cE,{\cE}x)}\cap K(\cE)$.
Moreover any object $X^\point\in\cL\cH(\cE,{\cE}x)$ can be represented as a complex 
$X^\point\in K^{\leq 0}(\cE)$ such that $\tau_{\cL}^{\geq 0}X^\point\cong
X^\point$
and so (as in the proof of Proposition~\ref{P:2assts})
it can be represented by a complex 
$C(d):=[\Ker_{\cE}(d)\stackrel{\alpha}\to X\stackrel{d}\to \stackrel{\point}Y]
\in\cL\cH(\cE,{{\cE}x})$ whose terms belong to $\cE$.
This suggests that $\cL\cH(\cE,{{\cE}x})$ is a Gabriel quotient of the heart $\cL\cK(\cE)$
as we will see in Theorem~\ref{T:HvsCF}.
The same argument of Proposition~\ref{P:2assts} \eqref{E:DT}
proves that the exact sequence $0\to \Ker_{\cE}(d)\to X\to Y \to C(d)\to 0$
is exact
 in $\cL\cH(\cE,{\cE}x)$, hence any object in the left heart
 admits a $\cE$-resolution of length at most $2$.
 Therefore the subcategory $\cE$ in $\cL\cH(\cE,{\cE}x)$ satisfies the
hypotheses of \cite[Prop.~13.2.6]{KS} (see Proposition~\ref{KS}), hence
${K(\cE)\over{\cN_{{\cE}x}}}\simeq D(\cL\cH(\cE,{\cE}x))$.
\end{proof} 

Now we have a definition for any property appearing in  Theorem~\ref{Th:1tilt}
whose generalization is the following theorem:

\begin{theorem}\label{Th:2tilt}
Let $(\cE,{\cE}x)$ be an additive category endowed with a  Quillen exact structure ${\cE}x$. The following properties are equivalent:
\begin{enumerate}
\item $\cE$ is a $2$-cotilting torsion-free class in an abelian category $\cC$ (and sequence in ${\cE}x$ are short exact sequence in
$\cC$ whose terms belong to $\cE$);
\item $\cE$ is a $2$-tilting torsion class in an abelian category $\cC'$  (and sequence in ${\cE}x$ are short exact sequence in
$\cC'$ whose terms belong to $\cE$);
\item $(\cE,{\cE}x)$  is a $2$-quasi-abelian category;
\item $\cE$ is the intersection of the hearts $\cH_\cD\cap\cH_\cT$
of a $2$-tilting pair of $t$-structures on $D(\cE,{\cE}x)$.
\end{enumerate}
Moreover $\cC\simeq \cL\cH(\cE)$, $\cC'\simeq\cR\cH(\cE)$ and $(\cD,\cT)=(\cR\cD_\cE,\cL\cD_\cE)$.
\end{theorem}
%
\begin{proof}
By Proposition~\ref{P:2t} given any $2$-tilting pair of $t$-structures
$(\cD,\cT)$ we obtain that $\cE$ is a $2$-tilting torsion 
(resp. $2$-cotilting torsion-free) class in $\cH_\cD$ 
(resp. $\cH_\cT$) and by Remark~\ref{R:ttc-2qa}
$\cE$ is $2$-quasi-abelian. So (4) implies (1), (2) and (3).
By Theorem~\ref{T:3.5} given $\cE$ a $2$-tilting torsion class in
$\cH_\cD$  the pair $(\cD,\cT)$ (on $D(\cH_\cD)$) is a $2$-tilting pair of $t$-structures
(where $\cT$ is the $t$-structure obtained by tilting $\cD$ with respect to
$\cE$) and by Proposition~\ref{P:DT=RL} the pair 
$(\cD,\cT)$ coincides with $(\cR\cD_{(\cE,{\cE}x)},\cL\cD_{(\cE,{\cE}x)})$.
So (2) implies (4) and by the dual of Theorem~\ref{T:3.5}  (1) implies (4).
Given $(\cE,{\cE}x)$ a $2$-quasi-abelian category endowed with a Quillen exact
structure by Lemma~\ref{L:toQES} one can associate the
$2$-tilting pair of $t$-structures $(\cR\cD_{(\cE,{\cE}x)},\cL\cD_{(\cE,{\cE}x)})$
on $D(\cE,{\cE}x)$ such that $\cE=\cR\cH(\cE,{\cE}x)\cap\cL\cH(\cE,{\cE}x)$.
So (3) implies (4).
\end{proof}

\begin{remark}
We have proved that for any $n$-quasi-abelian category  $(\cE,{\cE}x)$  with $n\in\{1,2\}$ we have a derived equivalence
$D(\cL\cD(\cE,{\cE}x))\simeq D(\cR\cD(\cE,{\cE}x))$
even if the category $\cE$ does not contain a tilting object.
\end{remark}
%

\section{Effaceable functors}\label{S:4}

We prove that  the left $\cL\cH(\cE,{{\cE}x})$ is a Gabriel quotient of the heart $\cL\cK(\cE)\simeq  \rcoh\cE$
(as suggested in Lemma~\ref{L:toQES}).
This section is devoted to the tool of effaceable functors
which we  will use in Section~\ref{S:n} to define a Serre subcategory of $\rcoh\cE$.

\begin{proposition}\label{P:5.1}
Let $\cE$ be a projectively complete category and let $\rfp\cE$ be
the Freyd category of (contravariant) finitely presented functors.
The maximal Quillen exact structure on $\rfp\cE$ 
is the one whose conflations are 
$0\to F_1\rightarrowtail F\twoheadrightarrow F_2\to 0$ such that for any $E\in \cE$
the sequences of abelian groups
$0\to F_1(E)\to F(E)\to F_2(E)\to 0$ are exact.
\end{proposition}
\begin{proof}
Let us recall that 
$\rfp\cE$ admits cokernels which are calculated pointwise  and if a morphism
admits a kernel it is also computed pointwise; moreover
by Proposition~\ref{PA.1} any functor which is (pointwise or in $\rMod\cE$) extension of 
finitely presented functors is finitely presented too.
Hence the push-out of any inflation is an inflation, resp. the pull-back of any deflation is a deflation
and they are stable by compositions so these conflations define a 
Quillen exact structure on $\rfp\cE$.
For any other Quillen exact structure on $\rfp\cE$ a conflation
$0\to G_1\rightarrowtail G\twoheadrightarrow G_2\to 0$
 is a kernel-cokernel sequence and so for any $E\in\cE$ 
 we get a short exact sequence
$0\to G_1(E)\to G(E)\to G_2(E)\to 0$ of abelian groups.
\end{proof}

Let us recall the definition of \emph{right filtering} subcategory
of an exact category and some related results due to 
Schlichting (\cite{ScM}). Let us recall that for any inflation
$A\rightarrowtail B$ the object $A$ is called an admissible subobject
of $B$.

\begin{definition}\cite[Def.~1.3.]{ScM}\label{D:rfil}
Let $\cU$ be an exact category (i.e., an additive category with
a Quillen exact structure)
and $\cA\subset \cU$. 
Then the inclusion $\cA\subset \cU$ is called \emph{right filtering} and $\cA$ is called 
\emph{right filtering in} $\cU$ if:
\begin{enumerate}
\item $\cA$ is  an extension closed full subcategory of $\cU$;
\item $\cA$ is closed under taking admissible subobjects and admissible quotients;
\item
every map $f:U\to A$ with $U\in\cU$ and $A\in\cA$ admits a factorisation  $f=g\pi$
\xymatrix{U\ar@{->>}[r]^{\pi} & B\ar[r]^{g} & A} with $B\in\cA$ and $\pi$ a deflation.
\end{enumerate}
\end{definition}
%

\begin{definition}\cite[Def.~1.12.]{ScM}
Let $\cU$  be an exact category and $\cA\subset \cU$ be an extension closed full subcategory.
A $\cU$-morphism is called a \emph{weak isomorphism}
if it is a finite composition of inflations with cokernel in $\cA$ and deflations with kernel in 
$\cA$. We write $\Sigma_{\cA\subset \cU}$ for the class of weak isomorphisms. 
\end{definition}
%

\begin{lemma}\cite[Lem.~1.13.]{ScM}\label{L:5.1}
If $\cA$ is right filtering in $\cU$ then $\Sigma_{\cA\subset \cU}$ admits a 
calculus of right fractions.
\end{lemma}
By passing to the opposite category one obtains the dual results in the left filtering case.

In the following, we will  define a  right filtering subcategory $\reff_{{\cE}x}\cE$ of 
 $\rfp\cE$ whose objects are the quotients in $\rfp\cE$ of deflations in ${\cE}x$,  they
 are called
 effaceable functors following the classical definition
(\cite[p.14]{Swan}, \cite[p.28]{We} and  \cite[p.4]{KRdaf}).
When $\cA$ is an abelian category,
 the right orthogonal class of $\reff\cA$ 
 coincides with the full subcategory of coherent functors
which respects the monomorphisms , hence the quotient category ${\rcoh\cA}\over{\reff\cA}$
is
the category of coherent left exact functors.
Following Krause's denomination the equivalence 
$\cA\simeq {{\rcoh\cA}\over{\reff\cA}}$  is called Auslander's formula (\cite[Th.~2.2]{KRdaf}).

This procedure is analog to the procedure one needs to do in order to define the category of sheaves
in abelian groups
associated to a topological space. One first defines the localizing Serre subcategory of
pre-sheaves which have stalk $0$ at any point, hence its right orthogonal class is formed
by separated pre-sheaves, while the quotient category provides the category of sheaves  in abelian groups.

It turns out that the approach via Quillen exact structures is equivalent 
to the one via Grothendieck topologies as recently explained
by Kaledin and Lowen in their paper \cite[2.2, 2.5]{Lowen2}. The deflations (resp. the inflations)
of a Quillen exact structure provide a Grothendieck pre-topology in $\cE$ 
(resp. in $\cE^\circ$). 
In this equivalence the notion of pre-sheaf with stalk $0$ at any point would
give rise to the notion of weak effaceable functor which is equivalent to the
notion of effaceable functor in the finitely presented case
(see Proposition~\ref{P4.1}).

Following the analogy with abelian sheaves on a topological space $X$, a pre-sheaf
$\cF$ has stalk $0$ in any point $x\in X$ if and only if for any $U$ open subset of $X$  and
$\eta \in \cF(U)$  there exists an open covering $p:\bigsqcup_{i\in I}U_i\twoheadrightarrow U$ such that 
the restriction
$\cF(p)(\eta)=\prod_{i\in I}\eta_{|U_i}=0$. In the additive context we have the following
counterpart: let $\cE$ be a projectively complete category endowed with a Quillen exact structure
$(\cE,{\cE}x)$ and
$\rfp\cE$ its Freyd category.
We denote by
$
\reff_{{\cE}x}\cE:=\{\Coker_{\rfp\cE}(q)\; |\; q \hbox{ is a deflation in } {\cE}x\}
$
the full subcategory of
 $\rfp\cE$ whose objects are cokernels 
 of morphisms induced by deflations of ${\cE}x$.
We call the elements of $\reff_{{\cE}x}\cE$ \emph{effaceable functors}.

\begin{proposition}\label{P4.1}

Let $F\in\rfp\cE$; the following are equivalent:
\begin{enumerate}
\item  $F$ is effaceable;
\item for any $U\in\cE$ and $\eta\in F(U)$, there exists a
deflation 
$p:Y\twoheadrightarrow U$ such that $F(p)(\eta)=0$ (weak effaceable).
\end{enumerate}
\end{proposition}
\begin{proof}
Let us prove that $(1)\Rightarrow (2)$. We have to prove that for any 
$\eta\in F(U)\cong \Hom_{\rfp\cE}(\cE_U,F)$ there exists a deflation $p:Y\twoheadrightarrow U$ such that $F(p)(\eta)=0$.
Let consider
$\cE_{E_1}\stackrel{q}\to\cE_{E_2}\stackrel{\gamma}\to F\to 0$
with $q :E_1\twoheadrightarrow E_2$ a deflation in $\cE$, then there exists
$\cE_{U}\stackrel{h}\to \cE_{E_2}$  (since  $\cE_U$ is projective in $\rfp\cE$) such that 
$\gamma h=\eta$.
Let consider the following commutative diagram where $Y:= E_1\times_{E_2}U$ and $p$ is a deflation
since it is the pull-back of a deflation (the axiomatic of Quillen exact structure guarantees the existence of the fiber product $Y$),   : 
{\footnotesize{\[\xymatrixrowsep{1.8 em}
\xymatrix{
\cE_{Y}\ar[r]^{p}\ar[d]&
\cE_U\ar[rd]^{\eta}\ar[d]_{h} & &\\
\cE_{E_1}\ar[r]^{q}\ar@/_1pc/[rr]_{0}
& \cE_{E_2} \ar[r]^{\gamma}& 
F \ar[r] & 0. \\
}
\]}}hence $F(p)(\eta)=\eta p=0$.

Let us prove that $(2)\Rightarrow (1)$. 
Since $F\in\rfp\cE$ is finitely presented there exists $f\in\cE(E_1,E_2)$ such that 
$\cE_{E_1}\stackrel{f}\to\cE_{E_2} \stackrel{\eta}\to F\to 0$ and by
hypothesis (2)
there exists a deflation $p:Y\twoheadrightarrow E_2$ such that 
$\eta p=0$  hence (since $\cE_{E_1}\twoheadrightarrow \Ker_{\rfg\cE}(\eta)$
and $\cE_Y$ is projective in $\rfg\cE$) there exists
$g:Y\to E_1$ such that $p=fg$ which proves
(by Remark~\ref{R3.1}) that $f$ is a deflation.
\end{proof}

\begin{remark}\label{R5.1}
Following $(2)$ implies $(1)$ in the previous Proposition~\ref{P4.1}
we have also proved that, given any presentation
$\cE_{E_1}\stackrel{f}\to\cE_{E_2} \stackrel{\eta}\to F\to 0$
 of an effaceable functor, the map $f$ is a deflation.
\end{remark}

\begin{proposition}\label{P:SerreSub}
Let consider $\rfp\cE$ endowed with its maximal Quillen exact structure.
The full subcategory
$\reff_{{\cE}x}\cE\subset \rfp\cE$ is right filtering;
if $\cE$ is right coherent, hence $\reff_{{\cE}x}\cE$
is a Serre subcategory of the abelian category $\rfp\cE=\rcoh\cE$.
Dually 
$\cE\leff_{{\cE}x}\subset \cE\lfp$ is left filtering in  $\cE\lfp$ and
if $\cE$ is left coherent, hence $\cE\leff_{{\cE}x}$
is a Serre subcategory of the abelian category $\cE\lfp=\cE\lcoh$.
\end{proposition}
\begin{proof}
Let us prove that $\reff_{{\cE}x}\cE\subset \rfp\cE$ is right filtering; by  Definition~\ref{D:rfil}
we have to verify:
\begin{enumerate}
\item $\reff_{{\cE}x}\cE$ is  an extension closed full subcategory of $\rfp\cE$;
\item $\reff_{{\cE}x}\cE$ is closed under taking admissible subobjects and admissible quotients
in $\rfp\cE$;
\item
every map $f:U\to A$ with $U\in\rfp\cE$ and $A\in\reff_{{\cE}x}\cE$ admits a factorisation  $f=g\pi$
with
$U\stackrel{\pi}{\twoheadrightarrow}  B\stackrel{g}\to  A$,  $\pi$ a deflation and $B\in\reff_{{\cE}x}\cE$.
\end{enumerate}
Let us verify that $\reff_{{\cE}x}\cE$ is closed under extension in $\rfp\cE$.
Let consider a conflation 
$0\to T_1\rightarrowtail T\twoheadrightarrow T_2\to 0$  such that  both $T_1, T_2$ are effaceable functors
and let us prove that $T$ satisfies condition (2) of
Proposition~\ref{P4.1}.
Given
 $\eta\in T(U)\cong\rfp\cE(\cE_U,T)$ with $U\in\cE$,  let us consider the following commutative diagram (explained below):
{\small \[\xymatrixrowsep{1em}
\xymatrix{
\cE_W\ar[r]^{q}\ar[rdd]_{0}& \cE_Y\ar[dd]^{\xi}\ar[rd]^{p} \ar@/^2pc/[rrdd]^{0}& & & \\
&&\cE_U\ar[d]^{\eta} \ar[rd]^{\beta\eta}& &\\
0\ar[r] &T_1\ar[r]^{\alpha} &T \ar[r]^{\beta} &T_2 \ar[r]& 0 \\
}
\]}
Because $T_2$ is effaceable,
there exists a deflation $p:Y\twoheadrightarrow U$ such that 
$\beta\eta p=0$,
and so $\eta p=T(p)(\eta)$ factors through $\alpha$ via $\xi\in\rfp\cE(\cE_Y,T_1)$. 
Now, since  $T_1$ is effaceable,
there exists $q:W\twoheadrightarrow Y$ such that $\xi q=T_1(q)(\xi)=0$, hence
$0=\alpha\xi q=\eta pq=T(pq)(\eta)$. We remark that $pq$ is a deflation since it is 
a composition of two deflations, therefore the previous construction proves that $T$ is 
effaceable. 

Let us prove that $\reff_{{\cE}x}\cE$ is closed under admissible subobjects and admissible
quotients.
Let   $0\to T_1\rightarrowtail T\twoheadrightarrow T_2\to 0$  be a conflation
in  $\rfp\cE$ with $T\in \reff_{{\cE}x}\cE$. 
Given $U\in \cE$ and
$\eta\in T_1(U)$,  
there exists a deflation 
$p:Y\twoheadrightarrow U$ such that $\alpha(Y)(T_1(p)(\eta))=T(p)(\alpha(U)(\eta))=0$,
which proves that $T_1(p)(\eta)=0$ (since $\alpha(Y)$ is a monomorphism of
abelian groups by Proposition~\ref{P:5.1}).
Given  an object $V$ of $\cE$ and
$\xi\in T_2(V)\cong\rfp\cE(\cE_V,T_2)$,
 there exists 
$\sigma:\cE_V\to T$ such that $\xi=\beta\sigma$ (because $\beta$ is a deflation).
Since $T$ is effaceable, there exists $q:W\twoheadrightarrow V$ such that
$\sigma q=0$,
which implies $\xi q=T_2(q)(\xi)=0$ and so $T_2$ is effaceable.

Let consider: $f:U\to A$, 
$\cE_{U_2}\stackrel{h}\to \cE_{U_1}\to U\to 0$  a presentation of $U\in\rfp\cE$ and 
$\cE_{A_2}\stackrel{p}\to \cE_{A_1}\to A\to 0$   a presentation of
$A\in\reff_{{\cE}x}\cE$ with $p$  a deflation.
Since representable functors are projective in $\rfp\cE$ there exists 
$f_i:\cE_{U_i}\to\cE_{A_i}$ with $i\in\{1,2\}$ such that $p f_2=f_1 h$.
Hence  the following diagram commutes: {\small{\[
\xymatrixrowsep{1.4em}
\xymatrix{
\cE_{U_2}\ar@/_1pc/@{..>}[rr]_(0.3){f_2} \ar[r]^(0.4){r}\ar[d]_{h}& \cE_{U_1\times_{A_1}A_2}\ar[d]^{q}\ar[r] & \cE_{A_2}\ar[d]^{p} \\
\cE_{U_1}\ar[d]\ar@{=}[r] & \cE_{U_1}\ar[d]\ar[r]^{f_1}& \cE_{A_1}\ar@{->>}[d] \\
U\ar@/_1pc/@{..>}[rr]_(0.3){f} \ar[r]^(0.3){\pi}& \Coker_{\rfp\cE}(q)\ar[r]^(0.65){g}  & A \\
}
\]}}Thus $\ker\pi$ belongs to $\rfp\cE$ because
$\cE_{U_2}\stackrel{r}\to \cE_{U_1\times_{A_1}A_2}\to \ker\pi\to 0$ is exact and
the sequence $0\to \ker\pi\to U\stackrel{\pi}\to \Coker_{\rfp\cE}(q)\to 0$ is a conflation
since it is pointwise exact. This proves that
 $\Coker_{\rfp\cE}(q)$ belongs to $\reff_{{\cE}x}\cE$ ($q$ is a deflation since it is the pullback of $p$).
 
When $\cA$ is an abelian category conditions (1) and (2) prove that $\reff_{{\cE}x}\cA$
is a Serre subcategory of $\rcoh\cA$.
The left statement holds true by duality (in $\cE^\circ$).
\end{proof}
%

\section{$n$-coherent categories}\label{S4}

We have seen that the main difference
between  $1$-quasi-abelian categories and $2$-quasi-abelian ones is the need of Quillen exact structures.
The passage from $n=2$  to $n\geq 3$
requires a new technicality due to the possible absence of
kernels and cokernels.
Let $(\cE,{\cE}x)$ be a projectively complete category endowed with a Quillen exact structure.
We are looking for a definition of $n$-quasi-abelian category which permits us 
to associate to  $(\cE,{\cE}x)$ a $n$-tilting pair of $t$-structures on
$D(\cE,{\cE}x):=K(\cE)/{\cN_{{\cE}x}}$.
By Proposition~\ref{P:DT=RL} 
we know that if these $t$-structures exist they are the left and right $t$-structures:
\[\cL\cD^{\leq 0}_{(\cE,{\cE}x)}:=\{X^\point\in K(\cE)\; |\; X^\point\cong E^\point_{\leq 0}\; \hbox{in }
D(\cE,{\cE}x) \hbox{ with }
E^\point_{\leq 0}\in K^{\leq 0}(\cE)\}
\]
\[\cR\cD^{\geq 1}_{(\cE,{\cE}x)}:=\{X^\point\in K(\cE)\; |\; X^\point\cong E^\point_{\geq 1}\; \hbox{in }
D(\cE,{\cE}x) \hbox{ with }
E^\point_{\geq 1}\in K^{\geq 1}(\cE)\}.\]

In the following we will use the notions  of coherent functor, coherent category 
(Definition~\ref{D1.3}) weak kernels and cokernels; we refer to Appendix~\ref{Ap1}
for more details.
First of all we study the case of $(\cE,{\cE}x_{\rm split})$
endowed with its minimal Quillen exact structure
so that
$D(\cE,{\cE}x_{\rm split})=K(\cE)$.

\begin{proposition}\label{P:4.1}
The followings hold:
\begin{enumerate}
\item the class
$\cL\cK_\cE^{\leq 0}$ is an aisle in $K(\cE)$ if and only if
$\cE$ is right coherent;
\item the class
$\cR\cK_\cE^{\geq 1}$ is a co-aisle in $K(\cE)$ if and only if
$\cE$ is left coherent.
\end{enumerate}
If $\cE$ is a right coherent category 
we have $\cL\cK(\cE)\simeq \rcoh\cE$  dually if $\cE$ is a left-coherent category 
$\cR\cK(\cE)\simeq \left(\cE\lcoh\right)^\circ$. Moreover given $\cE$ a coherent category 
$\cR\cK_\cE^{\leq -n}\subseteq \cL\cK_\cE^{\leq 0}\subseteq \cR\cK_\cE^{\leq 0}$
(with $n$ minimal) 
if  and only if $\rcoh\cE$ (or equivalently $\cE\lcoh$) has projective dimension 
$n$.
\end{proposition}
 \begin{proof}
Statement (2) is dual to (1).
 Let us recall that by Proposition~\ref{P1.3}
 $\cE$ is right coherent if and only if it admits
 weak kernels.
 
 Let us suppose that $\cL\cK_\cE^{\leq 0}$
is an aisle (we denote by $\tau_{\cL}^{\leq 0}$ its truncation functor) and let us prove that
$\cE$ is right coherent.
Let $d:E_0\to E_1$ be a morphism in
$\cE$ and let us regard it as a complex $E^\point:=:E_0\stackrel{d}\to E_1$
(with $E^0$ placed in degree $0$).
The universal property in $K(\cE)$ of the truncation 
$[\cdots\to K^{-1}\to \stackrel{\point}{K^{0}}\to 0\to\cdots]=\tau_{\cL}^{\leq 0}(E^\point)
\stackrel{\alpha^\point}\to E^\point$
proves that $(K^0,\alpha^0)$ is a weak kernel for $d$.

On the other side let us suppose that $\cE$ is right coherent and let us prove that $\cL\cK_\cE^{\leq 0}$
is an aisle in $K(\cE)$.
We note that $\cL\cK_\cE^{\leq 0}$ is 
 extension closed in $K(\cE)$ and $\cL\cK_\cE^{\leq 0}[1]\subseteq \cL\cK_\cE^{\leq 0}$.
Since $\cE$ is right coherent the Freyd category
of (contravariant) finitely presented functor is abelian $\rfp\cE=\rcoh\cE$ 
(Proposition~\ref{P1.3})
and $\cE$
coincides with the class of projective objects in $\rcoh\cE$; hence 
$D^{-}(\rcoh\cE)\simeq K^{-}(\cE)$.
Moreover $D^{\leq 0}(\rcoh\cE)\simeq \cL\cK_\cE^{\leq 0}$, hence their hearts are equivalent: $\cL\cK(\cE)\simeq \rcoh\cE$. 
The category $\rcoh\cE$ has finite projective dimension $n$ if and only if
given any $E^\point\in \cR\cK_\cE^{\geq 0}$
the kernel $\Ker_{\rcoh\cE}(d^0_E)$ admits a resolution of length 
at most $n-2$ (since $0\to \Ker_{\rcoh\cE}(d^0_E) \to E^0\to E^1\to \Coker_{\rcoh\cE}(d^0_E)\to 0$
is exact and any  projective resolution of  $\Coker_{\rcoh\cE}(d^0_E)$ 
has at most length $n$).
This is equivalent to require that 
$\tau_\cL^{\geq 1}X^\point\cong \tau_\cL^{\geq 1}X^{\geq 0}\subseteq \cR\cK_\cE^{\leq -n+2}$
for any  $X^\point\in K(\cE)$  (see \ref{num:lefttstr})
which is equivalent to 
$\cR\cK_\cE^{\leq -n}\subseteq \cL\cK_\cE^{\leq 0}\subseteq \cR\cK_\cE^{\leq 0}$.
 
 In this case $n$ is called the \emph{global dimension} of $\cE$.
\end{proof}
%


\begin{definition}\label{D:n-cohcat}
A coherent category of global dimension at most $n$
will be said \emph{$n$-coherent}.
For example the category $\rproj R$ of projective (right) modules of finite type   
on a coherent ring $R$ with global dimension $n$ is 
 $n$-coherent.
\end{definition}
\section{$n$-tilting torsion classes for $n>2$}\label{S:n}
%


%
%
%

\begin{definition}\label{Def:prek}
Let $(\cE,{\cE}x)$ be a projectively complete category endowed with a Quillen exact structure
and let $f:A\to B$ be a morphism in $\cE$.
A \emph{$(\cE,{\cE}x)$-pre-kernel} of $f$ is a map $i:K\to A$ in $\cE$ such that
$f\circ i=0$ and for any $j:X\to A$ such that $f\circ j=0$ there exist (possibly many)
a deflation $\pi:N\twoheadrightarrow X$ and a map $k:N\to K$
such that $j\pi=ik$:
{\small{\[\xymatrixrowsep{1.8em}
\xymatrix{
N\ar@{..>>}[r]^{\pi}\ar@{..>}[d]_{k} & X\ar[d]_{j}\ar[rd]^{0} \\
K\ar[r]^{i} \ar@/_1pc/[rr]_{0}& A\ar[r]^{f} & B. \\ }
\]}}
The category $(\cE,{\cE}x)$ has \emph{$(\cE,{\cE}x)$-pre-pull-back squares}  if, 
given any pair $f_i:X_i\to Y$ with $i=1,2$, there exist
 an object $Z$ and  $g_i:Z\to X_i$ such that $f_1 g_1=f_2 g_2$ and, for any pair of arrows
 $\alpha_i:W\to X_i$ with $i\in\{1,2\}$ such that $\alpha_1 f_1=\alpha_2 f_2$,
 there exist (not necessary unique)
a deflation $\pi:N\twoheadrightarrow W$ and a map $k:N\to Z$
such that the  diagram below commutes:
{\small{\begin{equation}\label{E1}\xymatrixrowsep{1.8em}
\xymatrix{
W\ar@/^1.5pc/[rrr]^{\alpha_1} \ar[rrd]_{\alpha_2} &N\ar@{..>>}[l]_{\pi}\ar@{..>}[r]^{k} &Z\ar@{-->}[r]^{g_1} \ar@{-->}[d]_{g_2} 
& X_1\ar[d]^{f_1} \\
&& X_2\ar[r]^{f_2}& Y.\\
}
\end{equation}}}
Passing throughout the opposite category one obtains the dual notion of 
\emph{$(\cE,{\cE}x)$-pre-cokernel } and \emph{$(\cE,{\cE}x)$-pre-push-out square}.
\end{definition}
%


\begin{remark}\label{R:pk}
The notion
 of $(\cE,{\cE}x)$-pre-kernel (resp.
$(\cE,{\cE}x)$-pre-cokernel) is not functors due to the lack of the unicity in their definition.
Nevertheless its existence is equivalent to require the existence of kernels
in the quotient category ${ \rfp\cE\over \reff_{{\cE}x}\cE}$ (see Remark~\ref{R:pk}) which is a necessary and sufficient condition to prove
that this quotient category is an abelian category (see Theorem~\ref{T:HvsCF}).
When the Quillen exact structure coincides with the minimal one (split short exact sequences)
we have $D(\cE,{\cE}x_{\rm split})=K(\cE)$ and
the previous definitions reduce to the notions 
\emph{weak kernel}
and \emph{weak pull-back square} (see Definition~\ref{D1.4}).

If $\cE$ admits weak kernels 
it admits $(\cE,{\cE}x)$-pre-kernels for any Quillen exact structure on $\cE$ since
any weak kernel is also a $(\cE,{\cE}x)$-pre-kernel.
More generally if $\cE$ admits $(\cE,{\cE}x)$-pre-kernels
given any other Quillen exact structure ${\overline{{\cE}x}}$
containing the conflations of ${\cE}x$ we have that 
$\cE$ admits $(\cE,{\overline{{\cE}x}})$-pre-kernels.
\end{remark}
%

\begin{lemma}\label{L:Dextc}
Let $(\cE,{\cE}x)$ be a projectively complete category endowed with a Quillen exact structure
and
$D(\cE,{\cE}x):=K(\cE)/{\cN_{{\cE}x}}$ its derived category.
The classes
\[\cL\cD^{\leq 0}_{(\cE,{\cE}x)}:=\{X^\point\in K(\cE)\; |\; X^\point\cong E^\point_{\leq 0}\; \hbox{in }
D(\cE,{\cE}x) \hbox{ with }
E^\point_{\leq 0}\in C^{\leq 0}(\cE)\}
\]
\[\cR\cD^{\geq 1}_{(\cE,{\cE}x)}:=\{X^\point\in K(\cE)\; |\; X^\point\cong E^\point_{\geq 1}\; \hbox{in }
D(\cE,{\cE}x) \hbox{ with }
E^\point_{\geq 1}\in C^{\geq 1}(\cE)\}\]
are extension closed in $D(\cE,{\cE}x)$ and
$
\cL\cD^{\leq 0}_{(\cE,{\cE}x)}[1]\subseteq \cL\cD^{\leq 0}_{(\cE,{\cE}x)}$;
$\cR\cD^{\geq 1}_{(\cE,{\cE}x)}[-1]\subseteq \cR\cD^{\geq 1}_{(\cE,{\cE}x)}$.
\end{lemma}
\begin{proof}
Let us prove that $\cL\cD^{\leq 0}_{(\cE,{\cE}x)}$ is extension closed; the analog result
for $\cR\cD^{\geq 1}_{(\cE,{\cE}x)}$ follows dually.
We have to prove that, given a distinguished triangle 
$X^\point \to M^\point \to Y^\point \stackrel{+}\to \,$
with $X^\point, Y^\point \in \cL\cD^{\leq 0}_{(\cE,{\cE}x)}$, the middle term
$ M^\point$ belongs to $ \cL\cD^{\leq 0}_{(\cE,{\cE}x)}$.
This is equivalent to prove that the mapping cone of the connecting morphism
$Y^\point[-1]\stackrel{\delta}\to X^\point$ belongs to $\cL\cD^{\leq 0}_{(\cE,{\cE}x)}$.
Since $\delta $ is a morphism in $D(\cE,{\cE}x):=K(\cE)/{\cN_{{\cE}x}}$  by definition it can be represented 
as 
$\xymatrix@-10pt{Y^\point[-1] \ar[r]^(0.6){\delta'}&F^\point  & X^\point\ar[l]^{\cong}_{\alpha} }$
with $\delta'$ and $\alpha$ morphism in $K(\cE)$ such that
$Y^\point[-1]=[\cdots \to E^{i}\to E^{i+1}\to \cdots \to E^0 \to E^1\to 0\cdots\to ]$,
$F^\point \in K(\cE)$,
$X^\point=
[\cdots \to X^{i}\to X^{i+1}\to \cdots \to X^0 \to  0\to \cdots]\in C^{\leq 0}(\cE)$
and  $M(\alpha)=[\cdots\to X^{-1}\oplus F^{-2}\to X^0\oplus F^{-1}\to \stackrel{\point}{F^0}
\to F^1\to \cdots]$, mapping cone of $\alpha$,
in ${\cN_{{\cE}x}}$.
So for any $n\geq 0$ the differential 
$d^n_{M(\alpha)}=d^n_F=m_{n} e_{n}:
\xymatrix@-10pt{F^n\ar@{->>}[r]^{e_n} & D^n\ar@{>->}[r]^{m_n} & F^{n+1}}$
with $e_n$ a deflation and $m_n$ an inflation. 
Thus $M^\point$ is isomorphic to  $M(\delta')$ whose first differential $d^1_M=d^1_F$, hence
its   zero differential  factors through $D^0\cong \Ker(d^1_F)$ 
as $d^0_M:
\xymatrix@-10pt{E^1\oplus F^0\ar[r]^(0.6){e} & D^0\ar@{>->}[r]^{m_0} & F^{1}}$; moreover
 $e$ is a deflation because the composition
$e_0:F^0\to E^1\oplus F^0\stackrel{e}\to D^0$ is a deflation. We obtain the following morphism
in $K(\cE)$
{\small{\[\xymatrixrowsep{1.2em}
\xymatrix{
 &[ \cdots \ar[r]\ar[d] & E^0\oplus F^{-1}\ar[r] \ar[d] & 
\Ker(e) \ar[r] \ar[d] &0 \ar@{>->}[d] \ar[r] & 0\ar[r]\ar[d] & \cdots] \\
M\cong & [ \cdots \ar[r]& E^0\oplus F^{-1}\ar[r]^{d_M^{-1}} &
E^1\oplus F^0 \ar[r]^(0.6){d_M^{0}} & F^1\ar[r] & F^2\ar[r] & \cdots] \\
}\]}}
whose mapping cone belongs to ${\cN_{{\cE}x}}$ thus proving that 
$M\in \cL\cD^{\leq 0}_{(\cE,{\cE}x)}$.
\end{proof}
%

\begin{lemma}
\label{L:wkc}
Let $(\cE,{\cE}x)$ be a projectively complete category endowed with a Quillen exact structure.
\begin{enumerate}
\item 
The subcategory $\cL\cD^{\leq 0}_{(\cE,{\cE}x)}$ is an aisle in $D(\cE,{\cE}x)$
if and only if  $\cE$ has $(\cE,{\cE}x)$-pre-kernels.
\item The subcategory $\cR\cD^{\geq 1}_{(\cE,{\cE}x)}$ is a co-aisle in $D(\cE,{\cE}x)$
if and only if  $\cE$ has $(\cE,{\cE}x)$-pre-cokernels.
\end{enumerate}
If the previous conditions are satisfied 
 $\cE=\cL\cD^{\leq 0}_{(\cE,{\cE}x)}\cap \cR\cD^{\geq 0}_{(\cE,{\cE}x)}$
 and  any object in the heart $\cL\cH(\cE,{\cE}x)$ can be represented as
a complex
$K^\point\in C^{\leq 0}(\cE)$ such that $K^i=(\cE,{\cE}x)$-pre-kernel of $d_K^{i+1}$
for any $i\leq -2$.
Dually objects in  $\cR\cH(\cE,{\cE}x)$ are complexes
$C^\point\in C^{\geq 0}(\cE)$ such that $C^i=(\cE,{\cE}x)$-pre-cokernel of $d_C^{i-2}$
for any $i\geq 2$.
\end{lemma}
\begin{proof}
(1). Let us suppose that $ \cL\cD^{\leq 0}_{(\cE,{\cE}x)}$ is an aisle in $D(\cE,{\cE}x)$.
Hence the
inclusion functor $i_{\leq 0}:\cL\cD^{\leq 0}_{(\cE,{\cE}x)}\hookrightarrow D(\cE,{\cE}x)$ 
admits a right adjoint $\tau^{\leq 0}_{\cL}:D(\cE,{\cE}x)\to \cL\cD^{\leq 0}_{(\cE,{\cE}x)}$
(by abuse of notation we will denote by $\tau^{\leq 0}_{\cL}$ also the composition
$i_{\leq 0}\tau^{\leq 0}_{\cL}$).
Let $f:A\to B$ be a morphism in $\cE$ and let regard it as a complex:
$M^\point:=[\stackrel{\point}A\stackrel{f}\to B] \in C^{\geq 0}(\cE)$.
Let $\alpha: \tau^{\leq 0}(M^\point)\to M^\point$ 
be the co-unit of the adjunction (it is a morphism in $D(\cE,{\cE}x)$) and 
$\tau^{\leq 0}_{\cL}(M^\point)=[ \cdots \to K^{-1}\to \stackrel{\point}{K^{0}}] \in C^{\leq 0}(\cE)$.
By Lemma~\ref{L:mD=mK} the morphism $\alpha$ is a morphism in $K(\cE)$ 
and we denote by $i:=\alpha^0:K^0\to A$. Let us prove that $K^0\stackrel{i}\to A$
is a  $(\cE,{\cE}x)$-pre-kernel for $f$.
Let consider $j:X\to A$ a morphism in $\cE$ such that $fj=0$, hence $j$ induces a morphism 
 $X[0]\stackrel{j}\to M^\point$
 in $D(\cE,{\cE}x)$. Since $X[0]\in  \cL\cD^{\leq 0}_{(\cE,{\cE}x)}$ there exists an unique morphism
 $\beta: X[0]\to \tau^{\leq 0}(M^\point)$ 
 in $D(\cE,{\cE}x)$ such that $\alpha\beta=j$ i.e., there exists a resolution
 $\cdots \to N^{-1}\to N^{0}\stackrel{\pi}{\twoheadrightarrow} X\to 0$  (which is a complex in 
 the null system $\cN_{{\cE}x}$) 
 and a morphism of complexes
 $k^\point:N^\point\to \tau^{\leq 0}(M^\point)$
 such that $j\pi:N^0\stackrel{k^0}\to K^0\stackrel{i}\to A$
 which proves that $\cE$ has $(\cE,{\cE}x)$-pre-kernels.
 
On the other side 
let us suppose that $\cE$ has $(\cE,{\cE}x)$-pre-kernels. The full subcategory
$\cL\cD^{\leq 0}_{(\cE,{\cE}x)}$
of $D(\cE,{\cE}x)$ is closed by $[1]$ and 
it is closed under extensions  (Lemma~\ref{L:Dextc}). 
It remains to prove that
the inclusion functor $i_{\leq 0}:\cL\cD^{\leq 0}_{(\cE,{\cE}x)}\hookrightarrow D(\cE,{\cE}x)$ admits a right adjoint
$\tau^{\leq 0}_{\cL}:D(\cE,{\cE}x)\to \cL\cD^{\leq 0}_{(\cE,{\cE}x)}$.
Let first compute the $\tau^{\leq 0}_{\cL}(L^\point)$ in the case
$L^{\point}:=[\stackrel{\point}{L^{0}}\stackrel{d_L^0}\to L^1\to L^2\to \cdots]\in C^{\geq 0}(\cE)$
(hence  $\tau^{\leq 0}_{\cL}(L^\point)\in\cL\cH(\cE,{\cE}x)$ since $ L^{\point}\in \cL\cD^{\geq 0}_{(\cE,{\cE}x)}$).
Let $K^0\stackrel{i}\to L^0$ be a  $(\cE,{\cE}x)$-pre-kernel of $d^0_L$,
$K^{-1}\stackrel{d^{-1}_K}\to K^0$ a  $(\cE,{\cE}x)$-pre-kernel of $i$ and
recursively let  $K^{-i-1}\stackrel{d^{-i-1}_K}\to K^{-i}$
be a  $(\cE,{\cE}x)$-pre-kernel of $d^{-i}_K$ with $i\geq 1$.
We
pose $\tau^{\leq 0}_{\cL}(L^{\point}):=
[\cdots\stackrel{d^{-2}_{K}}\to K^{-1}\stackrel{d^{-1}_{K}}\to K^0] \in \cL\cD^{\leq 0}_{(\cE,{\cE}x)}$ with 
$h:K^\point\to L^\point$ the morphism of complexes induced by $i:K^0\to L^0$. 
For any 
other complex $X^\point\in C^{\leq 0}(\cE)$ let us prove that the morphism
$D(\cE,{\cE}x)(X^\point, K^\point)\stackrel{h\circ \_}\to D(\cE,{\cE}x)(X^\point, L^\point)\cong K(\cE)(X^\point, L^\point)$ 
(by Lemma~\ref{L:mD=mK})
is an isomorphism. 
For any $\alpha\in K(\cE)(X^\point, L^\point)$ by definition of
$(\cE,{\cE}x)$-pre-kernel we can construct the following commutative diagram:

{\small\[\xymatrixrowsep{1.5 em}\xymatrixcolsep{1.5em}
\xymatrix{ 
\cdots \ar[r] & X^{-2}\ar@{=}[r] & X^{-2}\ar[r]^{d^{-2}_X} & X^{-1}\ar@{=}[r] & X^{-1}\ar[r]^{d^{-1}_X}  & X^{0}
\ar[rd]^{\alpha^0} \ar@/^1pc/[rrd]^{0}& & \\
  \cdots \ar[r] & N^{-2}\ar[d]_{k^{-2}} \ar@{->>}[u]^{\pi^{-2}}\ar@{->>}[r] & X^{-2}\times_{X^{-1}}N^{-1}\ar[r] \ar@{->>}[u]
& N^{-1}  \ar@{->>}[u]^{\pi^{-1}}
  \ar@{->>}[r]\ar[d]_{k^{-1}} & X^{-1}\times_{X^{0}}N^{0}\ar@{->>}[u]
 \ar[r] & N^{0}\ar@{->>}[u]^{\pi^{0}} \ar[d]_{k^0}  & L^0\ar[r]^{d^0_L}& L_1. \\
 \cdots \ar[r] & K^{-2}\ar[rr]^{d^{-2}_K}  & &  K^{-1}\ar[rr]^{d^{-1}_K} &  & K^{0} \ar[ru]_{i}\ar@/_1pc/[rru]^{0}& & \\ 
  }
\]}
Looking at the right side of the diagram, since $K^0\stackrel{i}\to L^0$ is $(\cE,{\cE}x)$-pre-kernel of $d^0_L$,
we can find $\pi^0$ and $k^0$ such that the diagram commutes.
The morphism $X^{-1}\times_{X^{0}}N^{0}\to K^0\to L^0$ is zero since $\alpha^0 d^{-1}_X=0$, thus we can find 
$\pi^{-1}$ and $k^{-1}$ such that the diagram commutes.
The composition $X^{-2}\times_{X^{-1}}N^{-1}\to X^{-1}\times_{X^{0}}N^{0}\to N^0$ is zero since 
the first map factors uniquely through  $X^{-2}\times_{X^{-1}}N^{-1}\to X^{-2}\times_{X^{0}}N^{0 }\to X^{-1}\times_{X^{0}}N^{0}$ and
$X^{-2}\times_{X^{0}}N^{0 }\to X^{-1}\times_{X^{0}}N^{0}\to N^0$ is zero. Iterating this procedure one can construct the previous commutative diagram.

The morphism
 $\pi:N^\point\to X^\point$ is a qis in $D(\cE,{\cE}x)$ and  $k^\point:N^\point\to K^\point$
is a morphism in $K(\cE)$.
Let us denote by $k\in D(\cE,{\cE}x)(X^\point, L^\point)$ the morphism $k=k^\point \pi^{-1}:$ 
\xymatrix{X^\point  &N^\point\ar[l]_{\pi}^{\cong} \ar[r]^{k^\point} & K^\point \\ }; hence
$\alpha =h  k$ which proves that 
$D(\cE,{\cE}x)(X^\point, K^\point)\stackrel{h\circ \_}\to D(\cE,{\cE}x)(X^\point, L^\point)$ 
is an epimorphism.
Given a morphism $k\in D(\cE,{\cE}x)(X^\point, L^\point)$ represented by  a qis
$\pi:N^\point\to X^\point$ and a morphism $k^\point:N^\point\to K^\point$ in $K(\cE)$,
the morphism $hk=0$ if and only if $ik_0=0$.
Since $d^{-1}_K$ is the   $(\cE,{\cE}x)$-pre-kernel of $i$ 
there exists $p^0:W^0\twoheadrightarrow N^0$ and $s^0:W^0\to K^{-1}$
such that $k^0 p^0=d^{-1}_ks^0$.
Let consider the cartesian square:
{\small{ $$\xymatrixrowsep{1.8 em}
\xymatrix{
\widetilde{N}^{-1}\ar[r]^{\tilde{d}^{-1}} \ar@{->>}[d]_{\tilde{p}^0}& W^0\ar@{->>}[d]_{p^0} \\
N^{-1}\ar[r]^{d_N^{-1}} & N^0 \\
}
$$}}and let us denote by
$\phi_{-1}:=k^{-1} \tilde{p}^0 - s^0 \tilde{d}^{-1}$. We have
$d_k^{-1}\phi_{-1}=0$ and, since $d^{-2}_k$ is the   $(\cE,{\cE}x)$-pre-kernel of $d_k^{-1}$,
 there exists $q^{-1}:W^{-1}\twoheadrightarrow \widetilde{N}^{-1}$ and $s^{-1}:W^{-1}\to K^{-2}$
such that $\phi_{-1} q^{-1}=d^{-2}_ks^{-1}$ which gives
$k^{-1} \tilde{p}^0q^{-1} =d^{-2}_ks^{-1}+ s^0 \tilde{d}^{-1}q^{-1}$. Let define
$p^{-1}:= \tilde{p}^0q^{-1}$ and $d_W^{-1}:=\tilde{d}^{-1}q^{-1}$.
Iterating the argument we construct a qis $W^\point\stackrel{p^\point}\to N^\point$
 such that $k p^\point$ is null homotopy 
(via the homotopy  $s^i$) which proves that 
$k=0$ in $D(\cE,{\cE}x)$.
 
Let $E^\point:=[\cdots\to E^{-1}\to \stackrel{\point}{E^0}
\to E^1\to\cdots]$ be a complex in $D(\cE,{\cE}x)$. Let 
consider the following commutative digram whose rows and columns are distinguished triangles
and let put by definition $\tau^{\leq 0}_{\cL}(E^\point)$ to be the mapping cone in
$D(\cE,{\cE}x)$ of the morphism $\tau^{\leq 0}_{\cL}(d^{-1}_{E^\point})$:
{\small\begin{equation}\label{E1.t-str}
\xymatrixrowsep{1 em}
\xymatrix{
E^{\leq -1}[-1] \ar[r]^{id}\ar[d]^{\tau^{\leq 0}(d^{-1}_{E^\point})} & E^{\leq -1}[-1] \ar[r]
\ar[d]^{d^{-1}_{X^\point}}
& 0\ar[r]^(0.6){+}\ar[d]  & \\
\tau^{\leq 0}_{\cL}(E^{\geq 0}) \ar[r] \ar[d] & E^{\geq 0} \ar[r]\ar[d] 
& \tau^{\geq 1}_{\cL}(E^{\geq 0}) \ar[r]^(0.6){+}\ar[d]^{\cong}& \\
\tau^{\leq 0}_{\cL}(E^\point) \ar[r] \ar[d]^{+} &  E^{\point}\ar[r]\ar[d]^{+}& \tau^{\geq 1}_{\cL}( E^{\point}) 
\ar[r]^(0.6){+}\ar[d]^{+} & \\
& & & \\ } 
\end{equation}}
For any $X^\point\in K^{\leq 0}(\cE)$ we have that 
$D(\cE,{\cE}x)(X^\point, E^\point)\cong D(\cE,{\cE}x)(X^\point,\tau^{\leq 0}_{\cL}(E^\point))$
since $D(\cE,{\cE}x)(X^\point,\tau^{\geq 1}_{\cL}(E^{\geq 0}))=0$.
An object in the heart $\cL\cH(\cE,{\cE}x)$ can be represented as
a complex
$K^\point\in K^{\leq 0}(\cE)$ such that $\tau^{\leq -1}_{\cL}(K^\point)\in\cN_{{\cE}x}$ and so
$K^i=D(\cE,{\cE}x)$-kernel of $d_K^{i+1}$
for any $i\leq -2$. Statement 
(2) is dual to (1).

If $\cE$ admits $(\cE,{\cE}x)$-pre-kernels and $(\cE,{\cE}x)$-pre-cokernels,
let us denote by $\tau^{\leq 0}_{\cL}$ (resp. $\delta^{\geq 1}_{\cR}$) the truncation functor
with respect to $\cL\cD^{\leq 0}_{(\cE,{\cE}x)}$
(resp.  $\cR\cD^{\geq 1}_{(\cE,{\cE}x)}$).
Hence  $E^\point\in\cL\cD^{\leq 0}_{(\cE,{\cE}x)}\cap \cR\cD^{\geq 0}_{(\cE,{\cE}x)}$
if and only if the composition 
$$\gamma: K^\point:=\tau^{\leq 0}_{\cL}E^\point\to E^\point \to \delta^{\geq 0}_{\cR}(E^\point)
=:C^\point$$
is an isomorphism in $D(\cE,{\cE}x)$ i.e., if and only if the mapping cone 
$M(\gamma)\in\cN_{{\cE}x}$:
{\small\[\xymatrixrowsep{0.8 em}\xymatrixcolsep{1 em}
\xymatrix{
K^\point\ar[d]^{\gamma}: & \cdots \ar[rr] \ar[d] && K^{-1}\ar[rr]^{d^{-1}_K}\ar[d] &&
K^{0}\ar[d]^{\gamma^0}\ar[rr] & &0 \ar[rr] \ar[d] && \cdots\ar[rr]\ar[d] && \\
C^\point : \ar[d] & \cdots\ar[rr] \ar[d]&&0\ar[rr] \ar[d]&&
C^{0}\ar[rr]^{d^{0}_C} \ar[d]&& C^1 \ar[rr]^{d^{1}_C} \ar[d] && \cdots\ar[rr]\ar[d]&& \\
M(\gamma): & K^{-1}\ar@{->>}[rd]\ar[rr]^{d^{-1}_K} && 
K^0\ar@{->>}[rd]\ar[rr]^{\gamma^0} && C^0\ar@{->>}[rd]\ar[rr]^{d^0_C} &&
C^1\ar[rr]^{d^1_C} \ar@{->>}[rd]&&\cdots\ar[rr] &&  \\
& &W^{-1}\ar@{>->}[ru] & & W^0\ar@{>->}[ru] & &W^1\ar@{>->}[ru] &&
W^2 & \cdots && \\
}\]}this proves that $K^\point \cong W^0[0]\in\cE$.
\end{proof}
%

\begin{lemma}\label{L:6.4}
Let us suppose that $(\cE,{\cE}x)$ admits $(\cE,{\cE}x)$-pre-kernels and $(\cE,{\cE}x)$-pre-cokernels.
Hence
$\cR\cD^{\leq -n}_{(\cE,{\cE}x)}\subseteq \cL\cD^{\leq 0}_{(\cE,{\cE}x)}\subseteq \cR\cD^{\leq 0}_{(\cE,{\cE}x)}$ (with $n\geq 2$) if and only if one of 
the following equivalent conditions holds:
\begin{enumerate}
\item 
given any complex 
$
K^{[-n+1,0]}:=\xymatrix@-7pt{
K^{-n+1}\ar[r]^{d_K^{-n+1}}& K^{-n+2}\ar[r] &\cdots \ar[r]^{d_K^{-1}} & K^0\\ }$
with $K^i=(\cE,{\cE}x)$-kernel of $d_K^{i+1}$ for any $i\leq -2$,
the morphism $d_K^{-n+1}$ has a kernel in $\cE$;
\item
given any complex
$
C^{[-n+1,0]}:=\xymatrix@-7pt{
C^{-n+1}\ar[r]^{d_C^{-n+1}}&C^{-n+2}\ar[r] &\cdots \ar[r]^{d_C^{-1}} & C^0\\ }$
with $C^i=(\cE,{\cE}x)$-cokernel of $d_C^{i-2}$ for any  $i\geq -n-1$,
the morphism $d_C^{-1}$ has a cokernel in $\cE$.
\end{enumerate}
In this case the sequence in (1) is exact in $ \cL\cH(\cE,{\cE}x)$
while the one in (2) is exact in
$\cR\cH(\cE,{\cE}x)$ and
 the pair $(\cR\cD_{(\cE,{\cE}x)},\cL\cD_{(\cE,{\cE}x)})$ is a $n$-tilting
pair of $t$-structures on $D(\cE,{\cE}x)$ .
\end{lemma}
\begin{proof}
Let us suppose that $n\geq 2$ and $\cR\cD^{\leq -n}_{(\cE,{\cE}x)}\subseteq \cL\cD^{\leq 0}_{(\cE,{\cE}x)}\subseteq \cR\cD^{\leq 0}_{(\cE,{\cE}x)}$.
Given  a complex 
$K^{[-n+1,0]}$
with $K^i=(\cE,{\cE}x)$-kernel of $d_K^{i+1}$ for any $i\leq -2$,
by the proof of Lemma~\ref{L:wkc} the complex
$\tau^{\leq -n+1}_{\cL}( K^{[-n+1,0]})$
 is constructed taking in degree $-n+1$ the  $(\cE,{\cE}x)$-kernel of $d_K^{-n+1}$ and taking in degrees
$i<-n+1$ the  $(\cE,{\cE}x)$-kernel of the differential $i+1$. Thus 
$\tau^{\geq -n+2}_{\cL}( K^{[-n+1,0]})\cong 
\tau^{\geq 0}_{\cL}( K^{[-n+1,0]})\in\cL\cH(\cE,{\cE}x)
$
since any term of this complex is  a $(\cE,{\cE}x)$-kernel of its successive differential
and
 the complex
$\xymatrix@-7pt{
K^{-n+1}\ar[r]^{d_K^{-n+1}}& K^{-n+2}\ar[r] &\cdots \ar[r]^{d_K^{-1}} & K^0\\ }$
is exact in $\cL\cH(\cE,{\cE}x)$.

By hypothesis $\cL\cH(\cE,{\cE}x)\subseteq \cR\cD^{\geq -n}_{(\cE,{\cE}x)}$ and
since $ K^{[-n+1,0]}\in \cR\cD_{(\cE,{\cE}x)}^{[-n+1,0]}$ we get
$\tau^{\leq -n+1}_{\cL}( K^{[-n+1,0]})\in\cR\cD^{\geq -n+1}_{(\cE,{\cE}x)}\cap
\cL\cD^{\leq -n+1}_{(\cE,{\cE}x)}=\cE[n-1]$.

The dual argument proves that  (2) holds true and  the sequence in  (2) is exact in
$\cR\cH(\cE,{\cE}x)$.

On the other side if (1) holds true  given $X^\point\in D(\cE,{\cE}x)$ we have
$\tau_\cL^{\geq 1}X^\point\cong \tau_\cL^{\geq 1}X^{\geq 0}
\subseteq \cR\cD^{\geq -n+1}_{(\cE,{\cE}x)}$
 (since $\tau_\cL^{\leq 0}X^{\geq 0}\in \cR\cD^{\geq -n}_{(\cE,{\cE}x)}$), hence 
$\cR\cD^{\leq -n}_{(\cE,{\cE}x)}\subseteq \cL\cD^{\leq 0}_{(\cE,{\cE}x)}\subseteq \cR\cD^{\leq 0}_{(\cE,{\cE}x)}$.
Therefore 
any object $K^\point\in\cL\cH(\cE,{\cE}x)$ can be represented as a complex 
$K^\point\in K^{\leq 0}(\cE)$ such that $\tau_{\cL}^{\leq -1}(K^\point)\in\cN_{{\cE}x}$
and so it can be represented by a complex 
\[
C(d_K^{-n},\dots,d_K^{-1}):=[
\xymatrix@-10pt{\Ker(d_K^{-n+1})\ar[r]^(0.6){d_K^{-n}} & K^{-n+1}\ar[r]^(0.6){d_K^{-n+1}}& K^{-n+2}\ar[r] &\cdots \ar[r]^{d_K^{-1}} & 
\stackrel{\point}{K^0} \\ }]
\]
such that $K^i$ is a $(\cE,{\cE}x)$-pre-kernel of $d_K^{i+1}$ for any $i\leq -2$.
The following distinguished triangle in $D(\cE,{\cE}x)$ provides a short exact sequence in $\cL\cH(\cE,{\cE}x)$
\[\xymatrix@-10pt{
0\ar[r] & C(0,d_K^{-n},\dots,d_K^{-2}) \ar[r] & K^0[0] \ar[r]&
C(d_K^{-n},\dots,d_K^{-1})\ar[r] & 0 \\ }
\]
which proves that $\cE$ generates $\cL\cH(\cE,{\cE}x)$.
Hence $K(\cE)/\cN_{{\cE}x}\simeq D(\cL\cH(\cE,{\cE}x))$ since the full subcategory $\cE$ in $\cL\cH(\cE,{\cE}x)$
satisfies the hypotheses of 
Proposition~\ref{KS}.
Dually $K(\cE)/\cN_{{\cE}x}\simeq D(\cR\cH(\cE,{\cE}x))$.
\end{proof}

\begin{definition}\label{D:n-qa}
A  projectively complete category $(\cE,{\cE}x)$ endowed with a Quillen exact structure
is called \emph{$n$-quasi-abelian} (for $n\geq 2$) if it admits $(\cE,{\cE}x)$-pre-kernels and 
$(\cE,{\cE}x)$-pre-cokernels
and one of the following equivalent conditions holds:
\begin{enumerate}
\item 
For any complex 
$
K^{[-n+1,0]}:=\xymatrix@-7pt{
K^{-n+1}\ar[r]^{d_K^{-n+1}}& K^{-n+2}\ar[r] &\cdots \ar[r]^{d_K^{-1}} & K^0\\ }$
such that $K^i$ is $(\cE,{\cE}x)$-pre-kernel of $d_K^{i+1}$ for any $i\leq -2$
the morphism $d_K^{-n+1}$ has a kernel in $\cE$.
\item
For any complex
$
C^{[-n+1,0]}:=\xymatrix@-7pt{
C^{-n+1}\ar[r]^{d_C^{-n+1}}&C^{-n+2}\ar[r] &\cdots \ar[r]^{d_C^{-1}} & C^0\\ }$
such that $C^i$ is $(\cE,{\cE}x)$-pre-cokernel of $d_C^{i-2}$ for any  $i\geq -n-1$
the morphism $d_C^{-1}$ has a cokernel in $\cE$.
\end{enumerate}

Whenever the exact structure is not specified we will consider $\cE$ endowed with
its maximal Quillen exact structure.
\end{definition}
Theorem (see \ref{n:lista}) and
Definition~\ref{D:2tiltorcl} suggest the following $n$-level
generalization of the notion
of $1$-tilting torsion class in an abelian category:

\begin{definition}\label{D:ntiltorcl}
Let $\cA$ be an abelian category.
A full subcategory $\cE\hookrightarrow \cA$ is a
 \emph{$n$-tilting torsion class} if
\begin{enumerate}
\item $\cE$ cogenerates $\cA$;
 \item $\cE$ is extension closed  in $\cA$, hence
 it is endowed with the Quillen exact structure ${\cE}x$ whose conflations are
 sequences in $\cE$ which are exact in $\cA$;
\item $\cE$ has $(\cE,{\cE}x)$-pre-kernels;
\item for any exact sequence in $\cA$: 
$
0\to A\to  X_{1}\stackrel{d_X^{1}}\to  \cdots \stackrel{d_X^{n-1}} \to
X_n\to B\to 0$
with $X_{i}\in \cE$ for any $1\leq i\leq n$ \and $A, B\in \cA$
we have $B\in\cE$.
\end{enumerate}

Dually a \emph{$n$-cotilting torsion-free class} in $\cA$ is  a full generating 
extension closed subcategory $\cE$ of $\cA$
admitting $(\cE,{\cE}x)$-cokernels and  such that 
for any exact sequence in $\cA$: 
$
0\to A\to  Y_{1}\stackrel{d_Y^{1}}\to  \cdots \stackrel{d_Y^{n-1}} \to
Y_n\to B\to 0$
with $Y_i\in \cE$  
we have $A\in\cE$.
\end{definition}
%

\begin{remark}\label{R6}
Given $(\cE,{\cE}x)$ a $n$-quasi-abelian category 
by Lemma~\ref{L:6.4} 
$\cE$ is a $n$-tilting torsion class in $\cR\cH(\cE,{\cE}x)$
and  $(\cR\cD_{(\cE,{\cE}x)},\cL\cD_{(\cE,{\cE}x)})$ is
a $n$-tilting
pair of $t$-structures
 on $D(\cE,{\cE}x)$.

On the other hand, given a $n$-tilting
pair of $t$-structures $(\cD,\cT)$ on $\cC$ 
by Proposition~\ref{P:DT=RL} and Lemma~\ref{L:wkc},  the category
$\cE=\cT^{\leq 0}\cap \cD^{\geq 0}$ (with the Quillen exact structure induced by $D(\cE,{\cE}x)$) admits $(\cE,{\cE}x)$-pre-kernels
and $(\cE,{\cE}x)$-pre-cokernels, hence it is  $n$-quasi-abelian. 
\end{remark}
%

%

\begin{theorem}\label{T:6.8}
Any $n$-tilting torsion class $\cE$ 
in $\cA$, endowed with the Quillen exact structure induced by $\cA$, is  $n$-quasi-abelian;
$\cA\simeq \cR\cH(\cE,{\cE}x)$ and 
${K(\cE)\over{\cN_{{\cE}x}}} \simeq D(\cA)$.
%
\end{theorem}
\begin{proof}
Conditions (1) and (4) of Definition~\ref{D:ntiltorcl} imply that 
$\cE$ satisfies  the hypotheses of Proposition~\ref{KS} and so
$K(\cE)/\cN_{{\cE}x}\simeq D(\cA)$. Since
$D^{\geq 0}(\cA)\simeq \cR\cD_{(\cE,{\cE}x)}^{\geq 0}$
we obtain that $\cR\cH(\cE,{\cE}x)\simeq \cA$.

By Lemma~\ref{L:wkc} $\cE$ has $(\cE,{\cE}x)$-pre-cokernels and by point (3) of Definition~\ref{D:ntiltorcl} 
$\cE$ admits $(\cE,{\cE}x)$-pre-kernels.
Moreover by  Lemma~\ref{L:6.4} the complex
$
C^{[-n+1,0]}:=\xymatrix@-7pt{
C^{-n+1}\ar[r]^{d_C^{-n+1}}&C^{-n+2}\ar[r] &\cdots \ar[r]^{d_C^{-1}} & C^0\\ }$
 is exact in $\cA$ if and only if $C^i$ is a $(\cE,{\cE}x)$-pre-kernel of $d_C^{i-2}$ for any  $i\geq -n-1$, hence by Definition~\ref{D:ntiltorcl},
the morphism $d_C^{-1}$ has a cokernel in $\cE$.
\end{proof}
%

\begin{corollary}\label{Cor:n3.5}
Let $\cD$ be the natural $t$-structure on the triangulated category $D(\cH_\cD)$
and $i:\cE\to \cH_\cD$ a $n$-tilting torsion class on $\cH_\cD$.
Hence $\cT^{\leq 0}:=\cD^{\leq -n}\star \cE\star \cE[1]\star \cdots \star\cE[n-1]$ 
is an aisle in $D(\cH_\cD)$ such that $\cE=\cH_\cD\cap\cH_\cT$ and the
pair $(\cD,\cT)$ is a $n$-tilting pair of $t$-structures.
We will say that the $t$-structure $\cT$ is obtained by tilting
$\cD$ with respect to the $n$-tilting torsion class $\cE$.
\end{corollary}
%
\begin{proof}
By Theorem~\ref{T:6.8},
the $n$-tilting torsion class $\cE$ is a $n$-quasi-abelian category and
$(\cR\cD_{(\cE,{\cE}x)},\cL\cD_{(\cE,{\cE}x)})$ is a
$n$-tilting
pair of $t$-structures
 on $D(\cE,{\cE}x)\simeq D(\cH_{\cD})$.
The right $t$-structure coincides with the natural one on $D(\cH_\cD)$ (i.e.,
 $\cR\cD_{(\cE,{\cE}x)}=\cD$) while the left $t$-structure satisfies
$\cL\cD_{(\cE,{\cE}x)}^{\leq 0}\subseteq \cT^{\leq 0}$.
On the other hand,
since $\cD^{\leq -n}\simeq \cR\cD_{(\cE,{\cE}x)}^{\leq -n}\subseteq \cL\cD_{(\cE,{\cE}x)}^{\leq 0}$
and $\cE[i]\subseteq \cL\cD_{(\cE,{\cE}x)}^{\leq 0}$, for any $1\leq i\leq n-1$,
we deduce that $\cT^{\leq 0}\subseteq \cL\cD_{(\cE,{\cE}x)}^{\leq 0}$.
This proves that $\cT^{\leq 0}\simeq \cL\cD_{(\cE,{\cE}x)}^{\leq 0}$ is an aisle,  $\cE=\cH_\cD\cap\cH_\cT$
and $(\cD,\cT)=(\cR\cD_{(\cE,{\cE}x)},\cL\cD_{(\cE,{\cE}x)})$ is a $n$-tilting pair
of $t$-structures on $D(\cH_\cD)$.
\end{proof}

\begin{theorem}\label{T:HvsCF}
Let $(\cE,{\cE}x)$ be a $n$-quasi-abelian category. One has the following equivalences of categories
\[
\cL\cH(\cE,{\cE}x)\simeq {\rfp\cE\over \reff_{{\cE}x}\cE}; 
\qquad
\cR\cH(\cE,{\cE}x)\simeq \left({\cE\lfp\over \cE\leff_{{\cE}x}}\right)^\circ.
\]
In the special case of an abelian category  endowed with its maximal Quillen exact structure
$(\cA,{\cA}x_{\rm max})$, these equivalences give the Auslander's formulas:
\[
\cA\simeq {\rcoh\cA\over \reff\cA};\qquad 
\cA\simeq  \left(\cA\lcoh\over \cA\leff\right)^\circ .
\]
\end{theorem}
\begin{proof}
The second statement is dual to the first one.
By the universal property of the Freyd category $\rfp\cE$,
there exists a unique functor $L$ cokernel preserving such that the diagram below commutes:
{\small\[\xymatrix@-10pt{
& \cE \ar[ld]\ar[rd] & \\
\rfp\cE \ar[rr]^(0.4){Q_L} & & \cL\cH(\cE,{\cE}x).\\
}\]}
If $F=\Coker_{\rfp\cE}(f)$, we have $Q_L(F)=\Coker_{ \cL\cH(\cE,{\cE}x)}(f)$.
The functor $Q_L$ is essentially surjective since 
any object $L\in\cL\cH(\cE,{\cE}x)$ 
admits a resolution $0\to  K^{-n} \stackrel{d^{-n}_K}\to \cdots \stackrel{d^{-1}_K}\to K^0\to L\to 0$
(due to the fact that $\cE$ is a $n$-cotilting torsion-free class in $\cL\cH(\cE,{\cE}x)$),
thus $L=\Coker_{\cL\cH(\cE,{\cE}x)}(d^{-1}_k)$ and 
$L\cong[K^{-n} \stackrel{d^{-n}_K}\to \cdots \stackrel{d^{-1}_K}\to \stackrel{\point}{K^0}]
=:C(d_K^{-n},\dots,d_K^{-1})$
in $D(\cE,{\cE}x)$.

We notice that $K=\Coker_{\rfp\cE}(g)$ satisfies
$Q_L(K)=0$ if and only if $g$ is a deflation in $\cE$, hence
$K\in\reff_{{\cE}x}\cE$.
This proves that the functor $Q_L$ induces a canonical faith and essentially surjective functor
$\overline{Q_L}$ such that the following diagram commutes:
{\small\[\xymatrix@-10pt{
& &\cE \ar[lld]\ar[rrd] & &\\
\rfp\cE \ar[rr]^{\pi}  \ar@/_1.3pc/[rrrr]_{Q_L}&&{\rfp\cE\over \reff_{{\cE}x}\cE}\ar[rr]^{\overline{Q_L}}&& \cL\cH(\cE,{\cE}x).\\
}\]}

It remains to prove that $\overline{Q_L}$ is full.
Given  $K$ and $L$  in $\rfp\cE$ with  presentations
$K^{-1}\stackrel{d_K^{-1}}\to K^0\to K\to 0$ and 
$L^{-1}\stackrel{d_L^{-1}}\to L^0\to L\to 0$,
their images under  $Q_L$ are
$C(d_K^{-n},\dots,d_K^{-1})$ and $C(d_L^{-n},\dots,d_L^{-1})$ resp.. 
A morphism $Q_L(K)\stackrel{h}\to Q_L(L)$ is 
a
morphism in $D(\cE,{\cE}x)$; hence
there exists a complex
$C^\point\in K^{\leq 0}(\cE)$ (up to truncation)
and morphisms
$
\xymatrix@-10pt{Q_L(K)  &C^\point\ar[l]_(0.4){\varphi}^(0.4){\simeq} \ar[r]^(0.4){\tilde{h}} & Q_L(L) \\ }
$
such that the mapping cone 
$M(\varphi)\in \cN_{{\cE}x}$.
The zigzag
\begin{equation}\label{E:c4}
\xymatrix{\pi(K)& &\Coker_{\rfp\cE}(d^{-1}_{C^\point})\ar[ll]_(0.6){\Coker_{\rfp\cE}(\varphi^0)} \ar[rr]^(0.6){\Coker_{\rfp\cE}(\tilde{h})} & &\pi(L) \\ }
\end{equation}
viewed as a
a morphism in ${\rfp\cE\over \reff_{{\cE}x}\cE}$, is sent to $h$ by  ${\overline{Q_L}}$.

Since $M(\varphi)\in \cN_{{\cE}x}\cap K^{\leq 0}(\cE)$, its $-1$ differential 
$K^{-1}\oplus  C^0\stackrel{(d_K^{-1}, \varphi^0)}\longrightarrow  K^0$
has to be a deflation and
the sequence $K^{-2}\oplus C^{-1}\to K^{-1}\oplus C^0\to K^0$
is exact. Therefore  the sequence
\[
0 \to \Coker_{\rfp\cE}(d^{-1}_{C^\point})\rightarrowtail  \Coker_{\rfp\cE}(d_K^{-1})
\twoheadrightarrow \Coker_{\rfp\cE}(d_K^{-1}, \varphi^0)\to 0
\]
is a conflation and $\Coker_{\rfp\cE}(d_K^{-1}, \varphi^0)\in \reff_{{\cE}x}\cE$ which proves that \eqref{E:c4} is a morphism in
 ${\rfp\cE\over \reff_{{\cE}x}\cE}$ and by construction it maps to $h$ by  ${\overline{Q_L}}$.
\end{proof}
%

\begin{corollary}
Let $(\cE,{\cE}x)$ be a $n$-quasi-abelian category and
$\overline{{\cE}x}$ a Quillen exact structure on $\cE$ finer than ${\cE}x$.
Hence the class 
\[\overline{ \reff_{{\cE}x}\cE}:=\{ \Coker_{\cL\cH(\cE,{\cE}x)}(g) \; |\; g \hbox{ is a deflation in }
\overline{{\cE}x}\}\]
is a Serre subcategory of $\cL\cH(\cE,{\cE}x)$ and
$\cL\cH(\cE,\overline{{\cE}x})\simeq {\cL\cH(\cE,{\cE}x)\over \overline{ \reff_{{\cE}x}\cE}}$.
\end{corollary}

\begin{corollary}
Let $\cE$ be a $1$-quasi-abelian category. Hence
\[
\cL\cH(\cE)\simeq {\rcoh\cE\over \reff\cE}
\qquad
\cR\cH(\cE)\simeq \left({\cE\lcoh\over \cE\leff}\right)^\circ.
\]
In this case the Serre subcategories of effaceable functors are:
\[
\reff\cE:=\{\Coker_{\rcoh\cE}(q)\; |\; q \hbox{ is a cokernel map in } {\cE}\}\]
\[
\cE\leff:=\{\Coker_{\cE\lcoh}(i)\; |\; i \hbox{ is a kernel map in } {\cE}\}
\]
since any cokernel map is a deflation (resp. any kernel map is an inflation)
if and only if $\cE$ is a $1$-quasi-abelian category.
\end{corollary}

\begin{remark}
Let consider $(\cE,{\cE}x)$  a $n$-quasi-abelian category, with $n\geq 3$,
which is not a $n-1$-quasi-abelian category (i.e.,
such that $\cR\cD^{\leq -n}_{(\cE,{\cE}x)}\subseteq \cL\cD^{\leq 0}_{(\cE,{\cE}x)}$
but $\cR\cD^{\leq -n+1}_{(\cE,{\cE}x)}\not\subseteq \cL\cD^{\leq 0}_{(\cE,{\cE}x)}$). Hence for any Quillen exact structure $\overline{{\cE}x}$ on $\cE$
finer than ${\cE}x$ (i.e., which contains the conflations of ${\cE}x$)
we have that $(\cE,\overline{{\cE}x})$ is a $n$-quasi-abelian category
which is not a $n-1$-quasi-abelian category.
Otherwise if  $\cR\cD^{\leq -n+1}_{(\cE,\overline{{\cE}x})}\subseteq 
\cL\cD^{\leq 0}_{(\cE,\overline{{\cE}x})}$,
any object $L\in\cL\cH(\cE,{\cE}x)$
which has a presentation
$0\to K^{-n}\stackrel{d_K^{-n}}\to\cdots \stackrel{d_K^{-1}}\to K^0\to L\to 0$ would short in 
$\cL\cH(\cE,\overline{{\cE}x})$ i.e.,
$d_K^{-n+2}$ would have a kernel (computed in $\cL\cH(\cE,\overline{{\cE}x})$)
which belongs to  $\cE$ but (since $\cE$ is fully faithful  in $\cL\cH(\cE,{\cE}x)$
this would be a kernel for $d_K^{-n+2}$ also in $\cL\cH(\cE,{\cE}x)$
which contradicts the hypothesis.

So for $n\geq 3$ the index $n$ of quasi-abelianity for $\cE$ is independent from the Quillen exact structure on $\cE$, hence it can be computed using the maximal Quillen exact structure.
\end{remark}

We are now able to prove the $n$ version of Theorem~\ref{Th:1tilt}.


\begin{theorem}\label{Th:ntilt}
Let $(\cE,{\cE}x)$ be an additive category endowed with a Quillen exact structure. The following properties are equivalent:
\begin{enumerate}
\item $\cE$ is a $n$-cotilting torsion-free class in an abelian category $\cA$;
\item $\cE$ is a $n$-tilting torsion class in an abelian category $\cA'$;
\item $(\cE,{\cE}x)$  is a $n$-quasi-abelian category;
\item $\cE$ is the intersection of the hearts $\cH_\cD\cap\cH_\cT$
of a $n$-tilting pair of $t$-structures in $D(\cE,{\cE}x)$.
\end{enumerate}
Moreover $\cA\simeq \cL\cH(\cE,{\cE}x)$, $\cA'\simeq\cR\cH(\cE,{\cE}x)$ and $(\cD,\cT)=(\cR\cD_\cE,\cL\cD_\cE)$.
\end{theorem}
%
%
\begin{proof}
If $(\cD,\cT)$ is a $n$-tilting pair of $t$-structures in $\cC$,
by Remark~\ref{R6} 
$\cE=\cH_{\cD}\cap \cH_\cT$ is a $n$-quasi-abelian category
and by Proposition~\ref{P:DT=RL} $(\cD,\cT)=(\cR\cD_{(\cE,{\cE}x)},\cL\cD_{(\cE,{\cE}x)})$.

Let $(\cE,{\cE}x)$ be a $n$-quasi-abelian category.
By Lemma~\ref{L:6.4} and Remark~\ref{R6}  the pair of $t$-structures $(\cR\cD_{(\cE,{\cE}x)},\cL\cD_{(\cE,{\cE}x)})$ is  $n$-tilting
 and 
$\cE$ is a $n$-tilting torsion class in $\cR\cH(\cE,{\cE}x)$
(resp. $\cE$ is a $n$-cotilting torsion-free class in $\cL\cH(\cE,{\cE}x)$).

By Theorem~\ref{T:6.8}, if $\cE$ is a $n$-tilting torsion class in $\cA'$ (resp. $n$-cotilting torsion-free class in $\cA$), 
$\cA'\simeq\cR\cH(\cE,{\cE}x)$ (resp. $\cA\simeq\cL\cH(\cE,{\cE}x)$) and
$\cE\simeq \cR\cH(\cE,{\cE}x)\cap \cL\cH(\cE,{\cE}x)$ which concludes the proof.
\end{proof}

We can visualize  the links between properties (1) to (4) by the following diagram:
\begin{equation*}\tcbhighmath{\xymatrix@-20pt{
\{\hbox{$n$-tilting torsion classes}\}\ar@{<->}@<-7ex>[dddd]
\ar@{<->}[rr] && \{\hbox{$n$-cotilting torsion-free classes}\}\\
  \qquad\qquad\cE \hbox{ in } \cR\cH(\cE,{\cE}x)\ar@{<->}[rr]& & \cE \hbox{ in } \cL\cH(\cE,{\cE}x)
\ar@<-5ex>[u]\ar@<5ex>[ddd] \\ &&\\
  &&\\
\{\hbox{$n$-quasi-abelian categories}\}\ar[uuu] \ar[d]\ar@{<->}[rr]& & \{\hbox{$n$-tilting pairs of $t$-structures}\}\ar[d] \ar[uuu]\\
\cE=\cR\cH(\cE,{\cE}x)\cap\cL\cH(\cE,{\cE}x)
\ar@{<->}[rr]& &(\cR\cD_{(\cE,{\cE}x)},\cL\cD_{(\cE,{\cE}x)}) \hbox{ on } \cC=D(\cE,{\cE}x) . \\
}}
\end{equation*}

\subsection{Examples}

\begin{example} \label{Ex4.1}
 Given $R$  a commutative ring, the following categories are $1$-quasi-abelian:
\begin{itemize}
\item
The category of filtered modules over any ring (\cite[Exam.~1.2.13]{Yves}).
\item
The category of torsion-free coherent sheaves over a reduced irreducible
analytic space or algebraic variety X . For X  a normal
curve, the previous category is that of vector bundles (of finite rank) (\cite[Exam.~1.2.13]{Yves}).
\item In the contest of $\mathcal D$-modules the category of strict relative coherent ${\mathcal D}_{X\times S/S}$-modules
with $X\times S$ a complex manifold and $\dim S=1$ (\cite{FMF}).
\item
Let $R$ be a (left and right) coherent ring
with global dimension ${\rm gl.dim }(R)=1$
and $\cE:={\mathrm{add}}(R)$ (see Appendix~\ref{R1.2}).
The maximal Quillen exact structure on $\cE$ coincides
with the minimal one and $\cE$ is a $1$-quasi-abelian category;
its left heart is $\cL\cK(\cE)\simeq\rcoh R$ 
(and so $\cE=\rproj\cE$ is $1$-cotilting torsion-free class with its minimal
Quillen exact structure) while $\cR\cK(\cE)\simeq(\cE\lcoh)^\circ$. 
\end{itemize}

The following list contains more examples of $1$ and $2$-quasi-abelian categories:
\begin{itemize}

\item
Let $R$ be a (left and right) coherent ring
with global dimension ${\rm gl.dim }(R)=2$
and $\cE:={\mathrm{add}}(R)$.
Hence  $\cE$ admits kernels and cokernels:
given a morphism $f:P_1\to P_2$ in $\cE$, its kernel
$\Ker_{\cE}(f)=\Ker_{\rcoh R}(f)\in\cE$ 
while $\Coker_{\cE}(f)=(\Ker_{R\lcoh}(f^\ast))^\ast$
where $(\_)^\ast:=\Hom_R(\_,R)$.
Therefore for any Quillen exact structure $\cE$ is $2$-quasi-abelian.
In \cite{Rump} Rump constructed a tilted algebra $A$, of type $\Bbb E_6$, such that its
category of projective modules of finite type has kernels and cokernels
(since $A$ has global dimension $2$), but it is not $1$-quasi-abelian.
\item
Let us consider the affine
plane $\Bbb A^2_k={\rm Spec}(R)$ with $R=k[x,y]$ and $k$ a field; hence 
$R$ has projective dimension $2$ and it is Noetherian therefore coherent; this assures
that $\cE:={\mathrm{add}}(R)$ has kernels and cokernels.
In this case $\cE$ coincides with the category
of free $R$-modules of finite type (this result was proved
by Seshadri in \cite{Sesh}, while  the general statement, known
as Serre problem, was proved by
Quillen and Suslin 
\cite{Qu}, \cite{Su}). Its left heart as a $2$-quasi-abelian category
endowed with its minimal Quillen exact structure, is the category
 ${\rm {Coh}}({\cO}_{\Bbb A^2_k})$
of coherent sheaves on the affine plane
$\Bbb A^2_k$. 
A sequence
$0\to \cE_1\stackrel{\alpha}\to \cE_2\stackrel{\beta}\to \cE_3\to 0$ is exact in $\cE$ 
for its maximal Quillen exact structure
if and only if  $\cE_3\cong (\Ker_{R}(\beta^\ast))^\ast$ and so the cokernel of $\beta$ 
in ${\rm {Coh}}({\cO}_{\Bbb A^2_k})$ is a torsion sheaf whose support has dimension $0$
(finite union of closed points). On the other side any coherent sheaf supported on a 
finite union of closed points can be represented as a cokernel of such a $\beta$.
Let us denote by $\cT_0$ the  class of torsion sheaves supported on points; this  is a Serre subcategory of
${\rm {Coh}}({\cO}_{\Bbb A^2_k})$ and the functor
$I_{\cL}:\cE\to {{\rm {Coh}}({\cO}_{\Bbb A^2_k})/\cT_0}$ is fully faithful  
and $\cE$ is a $1$-cotilting torsion-free class in ${{\rm {Coh}}({\cO}_{\Bbb A^2_k})/ \cT_0}$ and so 
$\cE$ is $1$-quasi-abelian category
(an hence
 the left heart of $\cE$ as a $1$-quasi-abelian category
is the quotient abelian category ${\rm {Coh}}({\cO}_{\Bbb A^2_k})/\cT_{0}$).
\end{itemize}
\end{example}
%

\begin{example}
Let $\cE$ be the category of free abelian groups of finite type. It is
a $1$-quasi-abelian category and
its maximal Quillen exact structure coincides with the minimal one
(split short exact sequences).  
Its left heart 
$\cL\cK(\cE)$ is the whole category of finitely generated abelian groups 
while $\cR\cK(\cE)=(\cE\lcoh)^\circ$ is equivalent to the opposite category
of  the category of abelian groups of finite type.
The derived equivalence $D({\cA}b)\simeq D({\cA}b^\circ)$
is given by $\R\Hom_{\Bbb Z}(\_, \Bbb Z)$ and  the intersection
of the hearts is given by the finitely generated abelian groups $F$ such that 
$\R\Hom_{\Bbb Z}(F, \Bbb Z)=\Hom_{\Bbb Z}(F, \Bbb Z)$ which are the free abelian groups of finite type.
One can also interpret the right heart as the   tilt
of the abelian category of finitely generated abelian groups 
with respect to the cotilting torsion-free class of  free abelian groups of finite type:
i.e., objects are complexes $d:F_0\to F_1$ (in degree $0$ e $1$) of free
abelian groups such that $\Coker(d)$ is a torsion group.
\end{example}
%

\begin{example}
\cite[Exam.~ 3.6.(5), Exer.~3.7.(12)]{Bayer}.
Let $X$ be a smooth projective curve, $\mu\in \Bbb R$ a real number
and let $A_{\geq \mu}$ be the full subcategory of ${\cC}{\mathrm {oh}}(\cO_X)$ 
generated by torsion sheaves and vector bundles whose HN-filtration
quotients have slope $\geq \mu$. Hence $A_{\geq \mu}$ is a tilting torsion
class in ${\cC}{\mathrm {oh}}(\cO_X)$.
In particular let $X={\Bbb P}_{k}^1$
 the projective line  over a field $k$.
 Let us recall that any coherent sheaf $\cF\in {\cC}{\mathrm {oh}}(\cO_{{\Bbb P}_k^1})$
 decomposes as
$\cF\cong \cF_{\mathrm{tor}}\oplus \cF_{\mathrm{free}}$
and, by the
Birkhoff-Grothendieck theorem, the torsion-free part is a direct sum of line bundles 
$\cO_{{\Bbb P}_k^1}(d_i)$.
So $\cE:=A_{\geq 0}$ is a a $1$-tilting torsion
class in ${\cC}{\mathrm {oh}}(\cO_{{\Bbb P}_k^1})$
(hence it is a $1$-quasi-abelian category).
In this case the maximal Quillen exact structure on $\cE$ does not coincide with the minimal one
since the sequence 
$0\to \cO_{{\Bbb P}_k^1} \to \cO_{{\Bbb P}_k^1}(1)^2 \to \cO_{{\Bbb P}_k^1}(2)\to 0$ does not split
(i.e., $\Ext^1_{\cO_{{\Bbb P}_k^1}}(\cO_{{\Bbb P}_k^1}(2),\cO_{{\Bbb P}_k^1}) \neq 0$).
So we have a right heart (as a $2$-quasi-abelian category with
$\cE$
endowed with the split exact structure) in $K(\cE)$
which is the category $(\cE\lcoh)^\circ$
while its right heart in $D(\cE)$ as $1$-quasi-abelian category is the category
of coherent sheaves ${\cC}{\mathrm {oh}}(\cO_{{\Bbb P}_k^1})$
(since $\cE$ is a $1$-tilting torsion class in it).
Concerning the left heart $\cL\cD(\cE)$ its objects are complexes
$X=[\cE^{-1}\stackrel{d}\to\cE^0]$ with $\cE^i\in\cE$ and $d$ a monomorphism in
$\cE$. Since any object in $\cE$ admits a finite resolution
whose terms are direct factors of  finite direct sums of 
$\cO_{{\Bbb P}_k^1}\oplus\cO_{{\Bbb P}_k^1}(1)$
(and so in ${\mathrm{add}}(\cO_{{\Bbb P}_k^1}\oplus\cO_{{\Bbb P}_k^1}(1))$
see Appendix~\ref{R1.2})
we can represent $X$ as a bounded complex
$X=[X^{-m}\to \cdots \to X^0]\in 
K^{\leq 0}({\mathrm{add}}(\cO_{{\Bbb P}_k^1}\oplus\cO_{{\Bbb P}_k^1}(1)))$.
Thus  for any $X\in\cL\cD(\cE)$ and for any $i>0$
we have $\Ext^i(\cO_{{\Bbb P}_k^1}\oplus\cO_{{\Bbb P}_k^1}(1),X)\cong
\cD(\cE)(\cO_{{\Bbb P}_k^1}\oplus\cO_{{\Bbb P}_k^1}(1),X[i])=0$ and
 (via the associated distinguished triangle) we get a short exact sequence
in the left heart
$0\to X^{[-m,-1]}[-1]\to X^0\to X\to 0$  which proves
that
$T=\cO_{{\Bbb P}_k^1}\oplus\cO_{{\Bbb P}_k^1}(1)$
is a projective generator of the left heart $\cL\cD(\cE)$.
Hence $\cL\cD(\cE)$ is equivalent
to the category of left modules of finite type on the ring
$R:=\End(\cO_{{\Bbb P}_k^1}\oplus\cO_{{\Bbb P}_k^1}(1))$ which is 
the path algebra of 
 the Kronecker quiver $Q$
\[
\xymatrix{
\point\ar@<0.5ex>[r]\ar@<-0.5ex>[r] & \point \\
}
\]
The derived equivalence 
$D^b({\cC}{\mathrm {oh}}((\cO_{{\Bbb P}_k^1}))
\simeq D^b({\mathrm{Rep}_k(Q)})$
(which holds true also in the unbounded derived categories by Theorem~\ref{Th:1--tilt})
 is due to A. Beilinson
and $T=\cO_{{\Bbb P}_k^1}\oplus\cO_{{\Bbb P}_k^1}(1)$
is an example of \emph{tilting sheaf}.
\end{example}

\begin{example}
Given $\cA$  a Grothendieck category and $T$ a classical $n$-tilting object
in $\cA$ one can associate to $T$ the $t$-structure:
\[\cT^{\leq 0} :=\{X^\point\in D(\cA)\; |\; \Hom_{D(\cA)}(T,X^\point)=0\; \hbox{for all  } i>0 \}
\]
\[\cT^{\leq 0} :=\{X^\point\in D(\cA)\; |\; \Hom_{D(\cA)}(T,X^\point)=0\; \hbox{for all  } i>0 \}.\]

The pair $(\cD,\cT)$ is a $n$-tilting pair of $t$-structures. 
The intersection $\cE$ of their hearts is
the full subcategory of $\cA$ whose objects are $n$-presented by $T$.
It is  a $n$-tilting torsion class in $\cA$
(see \cite[Prop.~6.2]{FMT} for more details).
\end{example}

\section{Perverse coherent sheaves}\label{CY}

This section provides a  generalisation of
Bridgeland categories of perverse coherent sheaves
by the use of $n$-tilting torsion classes.

This problem has been studied in 
\cite{TU} where the authors proposed a category of perverse coherent sheaves
via the used of iterated $1$-tilting classes (see also \cite{FMT} for a general treatment of this iterated Happel Reiten Smal{\o} procedure). 
The construction in \cite{TU} requires the use of a tilting complex, which is proved to exist in the
case of relative dimension $2$ under a technical assumption.
In our approach we will follow Bridgeland paper
and we will define (for $n= 2$) a category of perverse coherent sheaves without the use of a tilting
complex. 

Let $X$ be a Noetherian scheme  over $\Bbb C$, we denote by $\QcX$ (resp. $\chX$) the category of quasi-coherent (resp. coherent) sheaves on $X$
and by $D(\QcX)$ its derived category. We denote by $D(X)$ the derived category of
$\chX$ and we recall that it is equivalent to the derived category of quasi-coherent
sheaves with coherent cohomologies $D_{\mathrm {coh}}(\QcX)$.

\numero{\bf Assumptions.}\label{Ass}
\noindent 
For the rest of this section we will assume that  $f:Y\to X$ is a  projective birational morphism 
 of Noetherian locally $\Bbb Q$-factorial semiseparated 
  schemes over $\Bbb C$ such that $\bR f_*(\cO_Y)=\cO_X$
with relative dimension $n$. 
The condition of being Noetherian locally $\Bbb Q$-factorial semiseparated 
assures that the schemes involved have the resolution property i.e.;
every coherent sheaf is a quotient of some vector bundle.
Moreover 
we get:
\begin{itemize}
\item for any coherent $\cO_Y$-module $\cG$ we
have $\bR f_*(\cG)\in D^{[0,n]}(X)$;
\item $\id_{D(X)}\simeq \bR f_* \bL f^*$, hence the functor $ \bL f^*$ is  fully faithful;
\item $\bR f_* f^!\simeq \id_{D(X)}$, therefore the functor $f^!$ is  fully faithful;

\item $f^!(D^{\geq 0}(X))\subseteq D^{\geq -n}(Y)$
(this is the $n$-version of \cite[Lem. 3.1.4]{VdB} whose proof coincides with that one with $-1$ replaced by $-n$ and
$-2$ replaced by $-n-1$ at the beginning of the proof).
 \end{itemize}
 \smallskip 
 
 In the case of $n=1$ Van den Bergh proved in 
\cite[Lem. 3.1.2, Lem. 3.1.3, Lem. 3.1.5]{VdB}
(following  \cite[Prop. 5.1]{Bri1}) that
the following classes
\begin{align*}
\cT_{0}&= \{T\in \chY \; | \, \bR^1f_* T=0\} & ;\quad  & \cF_{0}= \{F\in \chY  \; | \,  F\stackrel{\phi_F}\hookrightarrow
H^{-1}f^!\bR^1f_*F\} \hfill \\
\cT_{-1}&= \{T\in \chY \; | \, \eta_T: f^*f_*T\twoheadrightarrow T\} &  ;\quad  &  \cF_{-1}= \{F\in \chY  \; | \, f_*F=0\} \hfill \\
\end{align*}
define torsion pairs in $\chY$
(which we will prove to be  tilting in Lemma~\ref{L:Tcog}).
We recall that $\eta_T:f^*f_*T\to T$ is the
co-unit of the adjunction $(f^*,f_*)$ while the map
$\phi_F: F\to
H^{-1}f^!\bR^1f_*F$ is the morphism obtained by taking the zero cohomology of the composition 
$F\to f^! \bR f_*F \to f^! \bR^1f_*F[-1]$ (where the first map is the unit of the adjunction $(\bR f_*,f^!)$).
Notice that $\cT_{-1}=\cT_{0}\cap \cX$
where $\cX:=\{\cF\in\chY\; | \, \Hom(\cF,C)=0\; \forall\; C\in \chY : \bR f_*C=0\}$.
The heart of the $t$-structure obtained by tilting the natural $t$-structure with respect to
the tilting   torsion pair $(\cT_{-1},\cF_{-1})$ 
(resp. $(\cT_{0},\cF_{0})$) is called $\Pvu$ (resp. $\Pvz$).
Hence $D(Y)\simeq D(\Pvu)\simeq D(\Pvz)$.

\numero {\bf Higher analog of $\cT_{-1}$ and $\cT_0$}.
Let $\cG$ be a coherent $\cO_Y$-module.
In the case of relative dimension $n>1$ we propose the following generalization of the 
previous tilting torsion classes:
\[
\cT_{0}= \{T\in \chY \; | \, \bR f_*T\cong f_*T\}\qquad
\cT_{-1}=  \{T\in \cT_{0} \; | \; 
\eta_T: f^*f_*T\twoheadrightarrow T\} .
\]

\begin{conjecture}\label{MC}
We conjecture that under the previous assumptions the
classes $\cT_{0}$ and $\cT_{-1}$ are  $n$-tilting in $\chY$.
\end{conjecture}
We will prove that for any $n$ these classes satisfy conditions (1), (2) and (4) of
Definition~\ref{D:ntiltorcl}.
For $n=1$ they  are tilting torsion classes by Lemma~\ref{L:Tcog}. We will prove in Theorem~\ref{T:dtiltingt}
that for $n=2$ they are  $2$-tilting in $\chY$.

Let us  prove that under the assumptions of \ref{Ass}
the class
 $\cT_{-1}$ cogenerates $\chY$; 
this statement is the relative version of McMurray Price's Lemma  \cite[Lem.~5.2]{Price} which we prove with the same argument in the following Lemma.

\begin{lemma}\label{L:Tcog}
Let $f:Y\to X$ be a projective morphism as in \ref{Ass} and let $\cL$ be an $f$-ample vector bundle.
For any $\cF\in \chY$ there exists 
a  monomorphism $\alpha: \cF \hookrightarrow T$ with $T\in \cT_{-1}$.
\end{lemma}
\begin{proof}
The relative Serre vanishing Theorem (\cite[Ch. III.5]{Har}), 
guarantees that given $\cF\in \chY$ for  $m\gg 0$ we have:
$\bR^i f_*(\cF\otimes_{\cO_Y} \cL^m)=0$ for any $i>0$ and  
the counit  $f^*f_*(\cF\otimes_{\cO_Y} \cL^m)\twoheadrightarrow \cF\otimes_{\cO_Y} \cL^m$ 
of the adjunction $(f^\ast, f_\ast)$ is an epimorphism;
which is equivalent to require that $\cF\otimes_{\cO_Y} \cL^m\in \cT_{-1}$.
Let $\cF\in \chY$ and let consider $m$ big enough such that  both $\cL^m$ and $\cF\otimes_{\cO_Y} \cL^m$ belong to
$ \cT_{-1}$. Let $\cE\twoheadrightarrow f_*(\cL^m)$ be an epimorphism in $\chX$ with $\cE$ a locally free
$\cO_X$-module of finite rank (it exists since $X$ has the resolution property).
Hence the composition 
\[\eta: f^*(\cE)\twoheadrightarrow f^*f_*(\cL^m)\twoheadrightarrow \cL^m\] 
is  a locally splitting epimorphism
since $ \cL^m$ is a locally free sheaf, hence its dual
$\eta^\vee: \cL^{-m}\to\cHom_{\cO_Y}(f^*(\cE),\cO_Y)$ is a locally splitting monomorphism
and so it is  pure (i.e., universally injective)
which implies that the morphism $\delta:=\cF\otimes_{\cO_Y}\cL^m\otimes_{\cO_Y}\eta^\vee$ is injective
\[\xymatrix{
\delta:\cF\ar@{^(->}[r] & 
\cF\otimes_{\cO_Y}\cL^m\otimes_{\cO_Y}\cHom_{\cO_Y}(f^*(\cE),\cO_Y)\cong \cHom_{\cO_Y}(f^*(\cE),\cF \otimes_{\cO_Y}\cL^m). \\
}
\]
Moreover
\[
\begin{matrix}
\hfill \bR f_\ast \cHom_{\cO_Y}(f^*(\cE),\cF \otimes_{\cO_Y}\cL^m)
\cong &\bR f_\ast \bR\cHom_{\cO_Y}(\bL f^*(\cE),\cF \otimes_{\cO_Y}\cL^m)\cong \hfill \\
\hfill \cong\bR\cHom_{\cO_X}(\cE,\bR f_\ast(\cF \otimes_{\cO_Y}\cL^m))\cong & 
\cHom_{\cO_X}(\cE, f_\ast(\cF \otimes_{\cO_Y}\cL^m))\cong \hfill \\
\hfill \cong & f_\ast \cHom_{\cO_Y}(f^*(\cE),\cF \otimes_{\cO_Y}\cL^m).\hfill \\
\end{matrix}
\]
The first isomorphism holds true since both $\cE$ and $f^\ast(\cE)$ are locally free coherent sheaves, hence
$f^\ast(\cE)$ is $\cHom_{\cO_Y}(\_, \cF \otimes_{\cO_Y}\cL^m)$-acyclic, while $\cE$ is $f^\ast$-acyclic. The second isomorphism is induced by the adjunction
$(\bL f^*,\bR f_*)$. Since $\cE$ is locally free $\bR \cHom_{\cO_X}(\cE, f_\ast(\cF \otimes_{\cO_Y}\cL^m))\cong
\cHom_{\cO_X}(\cE, f_\ast(\cF \otimes_{\cO_Y}\cL^m))$, hence the third isomorphism is deduced by the fact that we choose $m$ such that 
$\bR f_\ast(\cF \otimes_{\cO_Y}\cL^m)\cong f_\ast(\cF \otimes_{\cO_Y}\cL^m)$. The last isomorphism is induced by the adjunction
$(f^*,f_*)$.
It remains to prove that the counit of the adjunction
$f^*f_* \cHom_{\cO_Y}(f^*(\cE),\cF \otimes_{\cO_Y}\cL^m)\to \cHom_{\cO_Y}(f^*(\cE),\cF \otimes_{\cO_Y}\cL^m)$
is an epimorphism. By 
 the last isomorphism of the previous list we have
\[
\begin{matrix}
f^*f_* \cHom_{\cO_Y}(f^*(\cE),\cF \otimes_{\cO_Y}\cL^m)\cong &
f^*\cHom_{\cO_X}(\cE, f_\ast(\cL^m\otimes_{\cO_Y}\cF))\cong  \\
\cHom_{\cO_Y}(f^*(\cE),f^*f_*(\cL^m\otimes_{\cO_Y}\cF))
\end{matrix}
\]
\noindent and
by hypothesis
the counit  $f^*f_*(\cF\otimes_{\cO_Y} \cL^m)\twoheadrightarrow \cF\otimes_{\cO_Y} \cL^m$ 
 is an epimorphism which implies that 
 $\cHom_{\cO_Y}(f^*(\cE),f^*f_*(\cL^m\otimes_{\cO_Y}\cF)) \twoheadrightarrow
 \cHom_{\cO_Y}(f^*(\cE),\cL^m\otimes_{\cO_Y}\cF)$ (because $f^*(\cE)$ is locally free).
\end{proof}

%
%

\begin{lemma}\label{Lemma:cue}
The full subcategories 
$\cT_{i}$, with $i\in \{0, -1\}$,  are closed under extensions in  $\chY$. 
\end{lemma}
\begin{proof}
Let us prove  that $\cT_0$ is closed under  extensions in  $\chY$. 
Given any short exact sequence
\begin{equation}\label{Eq:7.1}
0\to T_1\to \cF \to T_2 \to 0\qquad \hbox{with } T_1, T_2\in \cT_0; \hbox{ and } \cF\in\chY
\end{equation}
 we get a distinguished triangle $\bR f_*T_1\to \bR f_*\cF\to \bR f_*T_2 \stackrel{+}\to$
with $\bR f_*T_1, \bR f_*T_2$ coherent $\cO_X$-modules (thus complexes concentrated in degree $0$)
which proves that $\bR f_*\cF\cong f_*\cF$ is a complex concentrated in degree $0$.

Let us prove  that $\cT_{-1}$ is closed under  extensions in  $\chY$. Let us start with
the short exact sequence \eqref{Eq:7.1} by supposing that $T_1, T_2\in \cT_{-1}$.
Since $\cT_{-1}\subseteq \cT_{0}$, we deduce by the previous argument that $\cF\in\cT_0$. Thus
the sequence $0\to f_*T_1\to f_*\cF\to f_*T_2\to 0$ is exact.
Hence the following diagram commutes
\begin{equation}\label{Eq:cdi}
\xymatrix{
&  f^*f_*(T_1)\ar[r]\ar@{->>}[d] & f^*f_*(\cF)\ar[r]\ar[d] &  f^* f_*(T_2)\ar[r]\ar@{->>}[d] & 0 \\ 
0\ar[r] & T_1\ar[r] &\cF \ar[r]& T_2\ar[r] &0 \\ 
}
\end{equation}
therefore the canonical map $f^*f_*\cF\twoheadrightarrow \cF$ is an epimorphism.
\end{proof}

\begin{lemma}\label{Lemma:ncoker}
(Under the assumptions \ref{Ass}), 
the full subcategories $\cT_{i}$, $i\in \{0, -1\}$, satisfy condition (4) of Definition~\ref{D:ntiltorcl}, namely:

for any exact sequence in $\chY$
\begin{equation}\label{Eq:lesx}
\xymatrix{
0\ar[r] &A\ar[r]& X_{1}\ar[r]^{d_X^{1}} & \cdots \ar[r]^{d_X^{n-1}} & 
X_n\ar[r] & B\ar[r] & 0\\ }
\end{equation}with $X_{j}\in\;\cT_{i}$ for any $1\leq j\leq n$ \and $A, B\in \chY$ 
we have $B\in\; \cT_{i}$.
\end{lemma}
\begin{proof}
Consider $X^\point:= [\cdots\to 0 \to X_1 \to\cdots \to \stackrel{\point}X_n\to 0\to \cdots]$
where $ X_n$ is placed in degree $0$. 
The sequence \eqref {Eq:lesx}
produces the distinguished triangle $A[n-1]\to X^\point \to B[0]\stackrel{+}\to\; $
which
induces the distinguished triangle 
$\bR f_*(A)[n-1]\to \bR f_*(X^\point)\to \bR f_*(B)\stackrel{+}\to .$
Since $f$ has relative dimension $n$, $\bR f_*(A)[n-1]\in D^{\leq 1}(X)$.
Hence $\bR f_*(B)\in D^{\leq 0}(X)$ (since  $\bR f_*(X^\point)\in  D^{\leq 0}(X)$), therefore $B$ belongs to $\cT_0$.

If $X_j\in \cT_{-1}$ for any $1\leq j\leq n$, by the
previous argument we deduce that $B$ belongs to $\cT_{0}$ and, since 
it is a quotient of $X_n$, $f^*f_*(B)\twoheadrightarrow B$.
\end{proof}

\begin{theorem}\label{T:dtiltingt}
For $n=2$ the  classes
\[
\cT_{0}=\{T\in \chY \; | \, f_*T=\bR f_*T\}
\qquad
\cT_{-1}=\{T\in \cT_{0} \; | \, f^*f_*T\twoheadrightarrow T\}
\]
are $2$-tilting torsion classes in $\chY$.
\end{theorem}
\begin{proof}
Points (1), (2) and (4) of Definition~\ref{D:2tiltorcl}
have been proved in Lemma~\ref{L:Tcog}, Lemma~\ref{Lemma:cue} and
Lemma~\ref{Lemma:ncoker} 
respectively.
We have to prove that $\cT_{-1}$ and $\cT_{0}$ have kernels.

The full subcategory  ${\cX_0}=\{T\in \QcY \; | \bR^2 f_*T=0\}$ is a $1$-tilting torsion class in $\QcY$ i.e.;
it is closed under direct sums, extension, quotients and it cogenerates $\QcY$ (since it contains any injective sheaf)

Given $\cF\in\QcY$ we will denote by $t_{\cX_0}(\cF)$ its torsion part (i.e.; the biggest subsheaf of $\cF$ lying in ${\cX_0}$). Notice that if
$\cF\in\chY$ even  $t_{\cX_0}(\cF)\in\chY$.

{\it Step 1.} Let us prove that $\cT_{-1}$ admits kernels.

Give any 
locally free sheaf $\cE\in\chX$, the sheaf $f^*\cE\in\cT_{-1}$ 
since, by \ref{Ass}, we have $\cE=\bR f_* \bL f^*\cE\cong \bR f_*f^*\cE\cong f_*f^*\cE$.
Hence, given any
coherent sheaf $\cM\in \chX$, the sheaf
$f^*\cM$ belongs to $\cT_{-1}$ (since it is the cokernel of a map  in $\cT_{-1}$).

Let $\cE_1\stackrel{\alpha}\to \cE_2$ be a morphism in $\cT_{-1}$ whose kernel in  $\chY$ is $\cK:=\Ker\alpha$.
Let us denote by  $\eta_{\cK}:f^*f_*\cK\to \cK$  the
counit of the adjunction $(f^*,f_*)$.
The short exact sequence
$0\to t_{\cX_0}(\Ker\eta_\cK)\stackrel{j}\to f^*f_*\cK\to {\overline{\cK}}\to 0$,
($\overline{\cK}:=\Coker j$), induces the distinguished triangle
$\bR f_*(t_{\cX_0}(\Ker\eta_\cK))\to f_*f^*f_*\cK\to \bR f_* ({\overline{\cK}})\stackrel{+}\to
$
which proves that $\bR f_* ({\overline{\cK}})\in D^{\leq 0}(X)$, therefore
${\overline{\cK}}\in \cT_{-1}$.
Let us 
verify that $\overline{\cK}=\ker_{\cT_{-1}}(\alpha)$.
Let $\cL\stackrel{\phi}\to \cE_1$ be a morphism in $\cT_{-1}$ such that $\alpha\phi=0$ and  consider the following functorial
commutative diagram obtained by the universal property of the kernel and by the adjunction
$(f^*,f_*)$:
\[\xymatrixrowsep{1 em}
\xymatrix{
&&\Ker\eta_{\cL}\ar@{^(->}[r]\ar[ldd]_{\gamma}\ar[lld]& f^*f_*\cL\ar@{->>}[r]^(0.6){\eta_{\cL}}
\ar[ldd]_{f^*f_*\beta}&\cL\ar[rdd]^(0.6){\phi}\ar[rrdd]^{0}\ar[dd]^{\exists !}_{\beta} \ar[lddd]_(0.4){\overline{\beta}}|(0.65)\hole && & \\
 t_{\cX_0}(\Ker\eta_\cK) \ar[rd] \ar[rrd]|(0.6)\hole& & & & & \\
&\Ker\eta_{\cK}\ar@{^(->}[r]& f^*f_*\cK\ar[rr]^(0.53){\eta_{\cK}}\ar@{->>}[rd]^{\pi}&& \cK\ar[r] & 
\cE_1\ar[r] ^{\alpha}&\cE_2 \\ 
& & & \overline{\cK}\ar[ru] & && \\
}
\]
we note that $\ker\eta_{\cL}\in \cX_0$ since $\cL\in\cT_{-1}$, hence $\gamma$ factors through 
$ t_{\cX_0}(\Ker\eta_\cK)$. Therefore there exists a unique $\overline{\beta}:\cL\to \overline{\cK}$ such that the diagram commutes.

{\it Step 2.} Let us prove that $\cT_{0}$ admits kernels.

Let $\cF_1\stackrel{\alpha}\to \cF_2$ be a morphism in $\cT_{0}$ whose kernel in  $\chY$ is $\cK:=\Ker\alpha$.
 Let $\cM[1]$  be the mapping cone of
 $\chi: \cK\to f^!\bR f_*\cK\to f^!\delta^{\geq 1}\bR f_* \cK$.
By \ref{Ass} we have $f^!\delta^{\geq 1}\bR f_* \cK\in \cD^{\geq -1}(Y)$, hence
 $\cM\in \cD^{\geq 0}(Y)$.  
 
 Let us prove that ${\overline\cK}:=t_{\cX_0}(H^{0}(\cM))$ belongs to $\cT_0$ and $\overline{\cK}=\ker_{\cT_{0}}(\alpha)$.
 Let consider the following commutative diagram with distinguished rows and columns:
 \[ \xymatrixrowsep{0.8em}
\xymatrix{
H^{0}(\cM)\ar[r] \ar[d] &\cM\ar[r]\ar[d] & 
\delta^{\geq 1}(\cM)\ar[r]^(0.7){+}\ar[d] & \\
\cK\ar[r]\ar[d] & \cK\ar[r]\ar[d]^{\chi} & 
0\ar[r]^(0.7){+}\ar[d] & \\
\cN\ar[r]\ar[d]^{+}&f^!\delta^{\geq 1}\bR f_* \cK\ar[r] \ar[d]^{+}&
\delta^{\geq 1}(\cM)[1]\ar[r]^(0.7){+} \ar[d]^{+}& \\
\; & \; & \; & \\
}
\] 
 By applying to it the functor $\bR f_*$ (using $\bR f_*f^!=\id_{\cD(X)}$) we deduce the following facts:
$\bR f_*\cN\in \cD^{\geq 1}(X)$, $f_*H^{0}(\cM)\cong f_*\cK$, the map
$\bR^1f_*\cK\to\bR^1f_*\cN\to \bR^1f_*\cK$ 
 (induced  by the sud-ovest square) is the identity, hence $\bR^1f_*(H^{0}(\cM))=0$.
Thus ${\overline{\cK}}\in\cT_0$
(since $f_* \left({H^{0}(\cM)\over {\overline{\cK}}}\right)=0$).
Any $\cL\stackrel{\phi}\to \cF_1$ in $\cT_{0}$ such that $\alpha\phi=0$
factors  uniquely through  
$\cL\stackrel{\phi'}\to\cK$.
In the exact sequence
{\footnotesize{$$
\xymatrixcolsep{1em}
\xymatrix{
\Hom_{D^b(Y)}^{-1}(\cL,f^!\delta^{\geq 1}\bR f_* \cK)
\ar[r]&\Hom_{D^b(Y)}(\cL,\cM)\ar[r]& 
\Hom_{D^b(Y)}(\cL,\cK)\ar[r]& \Hom_{D^b(Y)}(\cL,f^!\delta^{\geq 1}\bR f_* \cK)\\
}$$}}
the first and the last terms are zero since 
$$\Hom^i_{D^b(Y)}(\cL,f^!\delta^{\geq 1}\bR f_* \cK)\cong
\Hom^i_{D^b(X)}(f_*\cL,\delta^{\geq 1}\bR f_* \cK)=0\quad \forall i\in\{-1, 0\}.$$
This proves that $\Hom_Y(\cL,H^{0}(\cM))\cong\Hom_{D^b(Y)}(\cL,\cM)\cong \Hom_{Y}(\cL,\cK)$
(remember that $\cM\in D^{\geq 0}(Y)$). Thus
 we obtain that $\phi'$
 factors  uniquely through  
$\cL\stackrel{\phi''}\to H^{0}(\cM)$.
Therefore  the morphism $\phi''$  factors  uniquely through  
 $\cL\stackrel{\overline{\phi''}}\to \overline{\cK}$ (since $\cL\in\cX_0$).
 \end{proof}

 \begin{definition}
By Theorem~\ref{T:3.5} (for $n=2$ or supposing that  Conjecture~\ref{MC} holds true $n>2$), 
we define $({}^{i}\cD^{\leq 0}, {}^{i}\cD^{\geq 0})$ (with $i\in \{0, -1\}$) to be the 
$t$-structures obtained by tilting $\cD$ with respect to the $n$-tilting 
 torsion classes $\cT_{i}$.
 Their hearts are denoted by
 ${\;}^{i}{\rm Per}(Y/X)$ for $i\in \{-1,0\}$
 and objects in ${\;}^{i}{\rm Per}(Y/X)$ are called {\emph{perverse coherent sheaves}}.
 \end{definition}

 \begin{theorem}(Theorem~\ref{Th:ntilt}).
 For $n=2$ or assuming Conjecture~\ref{MC} 
 $$
 D(Y)\simeq D({\;}^{0}{\rm Per}(Y/X))\simeq D({\;}^{-1}{\rm Per}(Y/X)).
 $$
 \end{theorem}
 \begin{remark} 
 In higher dimension,
Toda remarked in \cite{T} that
Bridgeland proof, of the derived equivalence between $\cD^b(Y)$
 and $\cD^b(Y^+)$ via the intersection theorem, produces also the smoothness of the flop. Nevertheless 
 there are examples of $4$ dimensional flops which do not preserve the smoothness.
 We think that the use of the previous $n$-tilted torsion classes (which produce
equivalences $\cD(Y)\simeq \cD({\;}^{i}{\rm Per}(Y/X))$)
 could permit to
attack the problem of the equivalence $\cD({\;}^{-1}{\rm Per}(Y/X))\simeq \cD({\;}^{0}{\rm Per}(Y^+/X) )$ as in \cite{VdB}.
\end{remark}
%


%
%
%
%

\section{Comparison between $n$-abelian  and
$n+1$-quasi-abelian categories}\label{AvsQA}

Recently Jasso in \cite{Jasso} introduced the notion
of \emph{$n$-abelian category} whose 
basic example is an $n$-cluster-tilting subcategory of an abelian category. 
Let us briefly recall this definition and  the principal results of \cite{Jasso}.

Given $\cC$ an additive category and $d^0:X^0\to X^1$ a morphism in $\cC$
an \emph{$n$-cokernel of $d^0$} 
(\cite[Def.~2.2]{Jasso})
is a sequence
\[\xymatrix{
(d^1,\dots ,d^n):X^1\ar[r]^(0.7){d^1}& X^2\ar[r]^{d^2}&
\cdots \ar[r]^{d^n}& X^{n+1}}
\]
such that for all $Y\in\cC$ the sequence of abelian groups
\begin{equation}\label{EJ1}
\xymatrix{
0\ar[r]& \cC(X^{n+1},Y)\ar[r]^{d^n\circ \_}& \cC(X^n,Y)
\ar[r]^{d^{n-1}\circ \_}&\cdots\ar[r] &\cC(X^1,Y)\ar[r]^{d^{0}\circ \_}&
\cC(X^0,Y)}\end{equation}
is exact.
In terms of coherent functors in $\cE\lcoh$ the previous sequence \eqref{EJ1} 
proves that
(following the notation of Appendix~\ref{N1.0})
the kernel of the morphism ${\;}_{X^{1}}\cC\stackrel{d^{0}\circ \_}\longrightarrow
{\;}_{X^{0}}\cC$ is a coherent functor which admits a projective presentation 
\[\xymatrix{
0\ar[r] &{\;}_{X^{n+1}}\cC\ar[r]^{d^n\circ \_}& {\;}_{X^{n}}\cC\ar[r]^{d^{n-1}\circ \_}&\cdots \ar[r]& \Ker({\;}_{X^{1}}\cC\ar[r]^{d^{0}\circ \_}&{\;}_{X^{0}}\cC)\ar[r]& 0}\] of length $n$ in $\cE\lcoh$.
The dual concept of $n$-kernel implies that the  
 kernel of the morphism $\cC_{X^0}\stackrel{\_\circ d^{0}}\longrightarrow
\cC_{X^1}$ is a coherent functor admitting a projective presentation of length $n$
in $\rcoh\cE$.

An \emph{$n$-exact sequence in $\cC$} (\cite[Def.~2.4]{Jasso}) is a complex
$X^0\stackrel{d^0}\longrightarrow X^1 \stackrel{d^1}\longrightarrow\cdots \stackrel{d^{n-1}}\longrightarrow X^n\stackrel{d^n}\longrightarrow X^{n+1}$
such that $(d^0,\dots,d^{n-1})$ is a $n$-kernel of $d^n$ and $(d^1,\dots,d^n)$ is an
$n$-cokernel of $d^0$.

\begin{definition}(\cite[Def.~3.1]{Jasso}).
Let $n$ be a positive integer. An \emph{$n$-abelian category} is an additive category
${\mathcal M}$ satisfying the following axioms:
\begin{description}
\item[$(A0)$] the category ${\mathcal M}$ is projectively complete;
\item[$(A1)$] every morphism in ${\mathcal M}$ has an $n$-kernel and
an $n$-cokernel;
\item[$(A2)$] for every monomorphism $f^0:X^0\to X^1$ in ${\mathcal M}$ and
for every $n$-cokernel $(f^1,\dots,f^n)$ of $f^0$ the following sequence is
$n$-exact:
\[
\xymatrix{
X^0\ar[r]^{f^0} & X^1\ar[r]^{f^1} &\cdots \ar[r]^{f^{n-1}} & X^n\ar[r]^{f^n} &
X^{n+1}
}
\]
\item[$(A2^{op})$] 
for every epimorphism $g^n:X^n\to X^{n+1}$ in ${\mathcal M}$ and
for every $n$-kernel $(g^0,\dots,g^{n-1})$ of $g^n$ the following sequence is
$n$-exact:
\[
\xymatrix{
X^0\ar[r]^{g^0} & X^1\ar[r]^{g^1} &\cdots \ar[r]^{g^{n-1}} & X^n\ar[r]^{g^n} &
X^{n+1}
}
\]
\end{description}
\end{definition}

\begin{proposition}
Any $n$-abelian category ${\mathcal M}$ is an $n+1$-coherent category, hence  ${\mathcal M}$ is an
$n+1$-quasi-abelian category for any Quillen exact structure on  ${\mathcal M}$.
\end{proposition}
\begin{proof}
Axioms $(A0)$ and $(A1)$ prove that the category ${\mathcal M}$ is coherent 
(see Definition~{D1.3})
since
any kernel of a morphism between representable functors is finitely presented.
Thus $\rcoh\cE$ and $\cE\lcoh$ are abelian categories.
Moreover any coherent functor $F\in \rcoh\cE$ admits a presentation
$\xymatrix{\cC_{X^n}\ar[r] & \cC_{X^{n+1}}\ar@{->>}[r] & F\ar[r] & 0}$, hence by axiom $(A1)$ it admits a projective resolution of length ad most $n+1$
which proves that ${\mathcal M}$ is $n+1$-coherent 
(Definition~\ref{D:n-cohcat}). Therefore, by Definition~\ref{D:n-qa}, ${\mathcal M}$
endowed with its minimal Quillen exact structure
is an $n+1$-quasi-abelian category.
\end{proof}

There are $n+1$-coherent categories which are not $n$-abelian.
For example $1$-abelian categories are precisely abelian categories
while $2$-quasi-abelian categories are projective complete categories admitting
kernels and cokernels. For example $1$-quasi-abelian categories 
which are not abelian categories
are never $n$-abelian ones.

\appendix

\section{Maximal Quillen exact structure}\label{QES}

Let us recall the notion of Quillen exact structure on an
additive category $\cE$ and some results on the maximal 
Quillen exact structure on $\cE$. See \cite{KelQES} and \cite{Buh}.

\numero \label{N3.1}
An \emph{exact category} is the data $(\cE,{\cE}x)$ of an additive category $\cE$ and a class
${\cE}x$ of exact sequences 
$\xymatrix{A\ar@{>->}[r]^{i} & B \ar@{->>}[r]^p & C}$ 
called 
\emph{conflations}
($i$, called 
\emph{inflation} or \emph{admissible monomorphism}, is a kernel of $p$ while
$p$, called a \emph{deflation} or  \emph{admissible epimorphism}, is a cokernel of $i$) satisfying the following axioms:
\begin{description}
\item[$\rm{Ex}0$] The identity morphism of the zero object is a deflation.
\item[$\rm{Ex}1$] The composition of two deflations is a deflation.
\item[$\rm{Ex}1^\circ$] The composition of two inflations is an inflation.
\item[$\rm{Ex}2$] The push-out of an inflation along an arbitrary morphism exists and yields an inflation.
\item[$\rm{Ex}2^\circ$]  The pull-back of a deflation along an arbitrary morphism exists and yields
a deflation. 
\end{description}
We call ${\cE}x$ an \emph{exact structure} on $\cE$.

In general, an additive category $\cE$ can admit different exact structures. 
In particular 
any split short exact sequence is a conflation for any exact structure on $\cE$, moreover any additive category $\cE$ admits a 
\emph{minimal exact structure} whose conflations are split short exact sequences.

Recently many advances have been done also for the dual problem: does $\cE$ admit a \emph{maximal exact structure}?
Due to the definition of Quillen exact structure, the natural candidate for being the class of conflations for the
maximal exact structure on $\cE$ is the 
 class of all kernel-cokernel pairs stable by push-outs and pull-backs.
 In \cite{Rump2} Rump proved that any additive category admits a maximal Quillen exact structure
and, in \cite{SW}, Sieg and Wegner proved that for additive categories with
kernels and cokernels (or equivalently $2$-quasi-abelian categories)
this maximal exact structure coincides with the class of all stable kernel-cokernel pairs.
Crivei generalized the result of Sieg and Wegner as follows:

\begin{proposition}\label{P:Crivei}
\cite[Th.~3.5]{Crivei}
Let $\cE$ be a \emph{weakly idempotent complete} additive category (i.e.,  additive category in which every section has a cokernel, or equivalently, every retraction has a kernel).
The class of all kernel-cokernel pairs stable by push-outs and pull-backs is  the maximal  Quillen
exact structure on $\cE$.
\end{proposition}
%

\begin{remark}\label{R3.1}
Any projectively complete category $\cE$ is weakly idempotent complete and additive.
In particular any  $2$-quasi-abelian category is projectively complete.
If $(\cE,{\cE}x)$ is a Quillen exact structure on $\cE$ and  $gf$ is a deflation, hence  $g$ is a deflation (\cite[Prop. 7.6.]{Buh}).
\end{remark}
%

\begin{definition}\label{D3.1}
Let $(\cE,{\cE}x)$ be an exact category. A complex $X^\point$ with entries in $\cE$ is called
\emph{acyclic} if each differential $d^n:X^n\to X^{n+1}$ decomposes in $\cE$ as
$d^n=m_n\circ e_n:
\xymatrix{X^n\ar@{->>}[r]^{e_n} & D^n\ar@{>->}[r]^{m_n} & X^{n+1}}$
where $m_n$ in an inflation, $e_n$ is a deflation and the sequence
$\xymatrix{D^n\ar@{>->}[r]^(0.4){m_n} & X^{n+1}\ar@{->>}[r]^(0.4){e_{n+1}} & D^{n+1}}$ belongs to
${\cE}x$ for any $n\in\Bbb Z$.
\end{definition}
%

\numero{\bf The Derived Category of a projectively complete exact category}.\label{N3.2} 
Let us recall Neeman's construction of   the ``derived'' category to a projectively complete exact category$(\cE,{\cE}x)$ (\cite{NeeDEC}).  
Let $K(\cE)$ be the homotopy category of chain complexes in $\cE$ and
let $\cN_{{\cE}x}$ be the full subcategory of $K(\cE)$ whose objects are acyclic complexes.
By \cite[Lem.~1.1]{NeeDEC}, $\cN_{{\cE}x}$ is a  
triangulated subcategory. 
By \cite[Lem.~1.2 and Rem.~1.8]{NeeDEC} $\cN_{{\cE}x}$ is a thick full triangulated subcategory of $K(\cE)$ if and only if $\cE$ is projectively complete. 
The \emph{derived category} $D(\cE)$ of a projectively complete exact category $(\cE,{\cE}x)$ is, by definition, the quotient (as triangulated categories) of $K(\cE)$ by $\cN_{{\cE}x}$.
Whenever the exact structure on $\cE$ is not specified, we will endow $\cE$ with its maximal
Quillen exact structure (\ref{P:Crivei}).
%
 
\begin{lemma}\label{L:mD=mK}
Let $(\cE,{\cE}x)$ be a projectively complete category endowed with a Quillen exact structure.
For any $X^\point\in C^{\leq 0}(\cE); Y^\point\in C^{\geq 0}(\cE)$ we have
\[D(\cE,{\cE}x)(X^\point,Y^\point)=K(\cE)(X^\point,Y^\point).\]
\end{lemma}
%
\begin{proof}
Given $\alpha\in D(\cE,{\cE}x)(X^\point,Y^\point)$
 the composition $i:X^0[0] \to X^\point\stackrel{\alpha}\to Y^\point \to Y^0[0]$
produces a morphism $i:X^0\to Y^0$ in $\cE$ such that $d^0_{Y^\point}i=0$ and $id^{-1}_{X^\point}=0$.
Hence $i$ induces a morphism in $K(\cE)$ which represents $\alpha$.
\end{proof}
%
\section{Freyd categories and coherent functors}\label{Ap1}

We will consider $\cC$ a category in the classical terminology
(for which any homomorphism class $\cC(X,Y)$, with
$X,Y$ objects in $\cC$, is a set).
Some authors define this a locally small category
in order to underline that its homomorphism form a set.
The wider notion of category, which permits to consider
also homomorphism which does not form a set, is very convenient 
once working with localization procedures.

\begin{definition}\label{D1.1}
Let us recall that a category $\cC$ is called:\begin{enumerate}
\item \emph{pre-additive} if
any hom-set is a group and the composition in bilinear;
\item \emph{additive} if it is pre-additive with zero object and biproducts;
\item \emph{idempotent complete}
\footnote{
It also called Karoubian by some authors.}   if any idempotent splits; 
\item \emph{projectively complete}
\footnote{It is also called Cauchy complete in \cite{Str}, or amenable by \cite{Fr}.}  when it is additive and idempotent complete.
\end{enumerate}
\end{definition}

We will use the notation: 
$\cC$ for a pre-additive category, $\cB$ for an additive category,
$\cP$ for a  projectively complete category,
$\cE$ for a Quillen exact category (see \ref{N3.1})
and $\cA$ for an abelian category.

\numero\label{N1.0}


We denote by
{\bf $\rMod\cC$} the enriched category of additive contravariant functors 
(i.e., $F:\cC^\circ \to {\cA}b$)
from
$\cC$ to  the category ${\cA}b$ of abelian groups, and by
{\bf x} the one of covariant functors (see \cite{KRdaf} \cite{AK}, \cite{Str}, \cite{MiRSO}).
The following functors are the enriched version of the Yoneda ones:
\[
\begin{matrix}
Y_{\cC}: \cC & \longrightarrow &\rMod\cC\hfill  \\
\hfill X & \longmapsto &\cC_X:=\cC(\_,X) \\
\end{matrix}\qquad
\begin{matrix}
{}_{\cC}Y: \cC& \longrightarrow &(\cC\lMod)^{\circ}  \hfill  \\
\hfill X & \longmapsto & {}_X{\cC}:={\cC}(X,\_) \\
\end{matrix}
\]
they admit an additive analogue of the Yoneda Lemma.

\begin{remark}\label{R1.2} 
Let $\cC$ be a pre-additive category (not necessarily small). One can perform 
the  \emph{projective completion}  of $\cC$ formally adding the zero object and finitely direct sums of objects 
in $\cC$, hence taking its idempotent completion. We denote by
${\mathrm{add}}(\cC)$ the projective completion
of $\cC$ (see for example \cite{BalmerS}).
Let \emph{$\rproj\cC$} (resp. \emph{$\cC\lproj$}) be the full subcategory of $\rMod\cC$
(resp. of $\cC\lMod$)
whose objects are direct summands of 
 finite direct sums of representable functors (we note that natural transformations between
two such objects always form  a set).  Hence ${\mathrm{add}}(\cC)$, 
$\rproj\cC$ and $\cC\lproj$
are equivalent.
Whenever $\cP$ is a projectively complete category ,
  $\rproj\cP$ is equivalent to $\cP$ and it coincides
with the full subcategory of $\rMod\cP$ of representable functors.
Any additive functor $F:\cC^\circ \to {\cA}b$ can uniquely be extended to
an additive functor $\overline{F}:(\rproj\cC)^\circ \to {\cA}b$, thus
$\rMod\cC$ is equivalent to $\rMod\rproj\cC$.
\end{remark}
%

\numero{\bf Coherent Functors}
In  his paper for the Proc. Conf. Categorical Algebra (La Jolla, Calif., 
1965) \cite{Au1} Auslander introduced the study of \emph{coherent functors} in the category $\rMod\cA$ 
with $\cA$ an abelian category (a ``genetic" introduction to this theme can be found in \cite{HCF}).
In the same collection Freyd \cite{Fr} introduced the study of the \emph{Freyd category} of 
\emph{finitely presented  functors} associated to a projectively complete category $\cP$.
These theories, and the related vocabulary, are widely inspired by the theory of  finitely presented and coherent modules over a ring $R$ which is also the easiest case (pre-additive category with a single object see \ref{N1.0}). 

The basic idea is that whatever one knows on finitely presented (resp. coherent) modules over a ring
has its counterpart for finitely presented (resp. coherent) functors in $\rMod\cC$.
It is well known (\cite[Ch.I]{Bour}, \cite[\S 1.5]{Bosch}) that, given a ring $R$,
right coherent modules $\rcoh R$ form a full abelian subcategory of all right $R$ modules $\rMod R$,
while finitely presented modules $\rfp R$ form a full 
projectively complete
subcategory of $\rMod R$ 
admitting cokernels. 
 The category $\rfp R$ is an abelian 
subcategory of $\rMod R$ if and only if the ring $R$ is right coherent. In that case
coherent and finitely presented modules coincide: $\rcoh R=\rfp R$  (these theorems go back to Henri Cartan).
In general finitely generated modules form a  full projectively complete
subcategory of $\rMod R$ and $\rfg R$ is an abelian subcategory of $\rMod R$
if and only if the ring $R$
 is right Noetherian,  in this case coherent modules coincide with  finitely generated ones:
$\rcoh R=\rfp R=\rfg R$.

The proofs of these statements 
are based on the fact that $R^n$ is a projective compact object in $\rMod R$, hence the functor
$\Hom_R(R^n,\_)$ commutes with all colimits (and limits too). 
By the Yoneda's Lemma any direct summand of a finite direct sum of representable functors  in 
$\rMod\cC$ (i.e., an element in $\rproj\cC$) is projective and compact
hence 
Cartan's results extend to coherent  functors replacing 
the role of $R^n$ by that of objects in $\rproj\cC$.
By Remark~\ref{R1.2}, 
$\rMod\cC$ and $\rMod\proj(\cC)$ are equivalent,
hence, from now on, given any pre-additive category
we will pass to its projective completion
$\cP:=\proj(\cC)$.

Freyd's work \cite{Fr} has been further investigated and developed by Beligiannis in
his very inspiring paper \cite{Be} to which we  refer (see also \cite{Au1}). 

\begin{definition} \label{D:cohfun}
An object $F\in\rMod\cP$ is called 
\emph{finitely generated} if 
there exists an epimorphism
$\cP_X\twoheadrightarrow F$ with $X\in\cP$.
An object $F\in\rMod\cP$ is called 
\emph{finitely presented}  if it fits into an exact sequence in $\rMod\cP$:
$\xymatrix{
\cP_{X_1}\ar[r] & \cP_{X_2}\ar[r] & F\ar[r] & 0\\
}$
with $X_i\in\cP$ for $i=1,2$.
An object $F$ finitely generated is called \emph{coherent}
if any
subobject $G\hookrightarrow F$ finitely generated is
finitely presented. Hence any finitely generated subfunctor of a coherent functor  is coherent.
We will denote by {\bf $\rfg\cP$}, resp. {\bf $\rfp\cP$}, resp. {\bf $\rcoh\cP$} the full subcategory of $\rMod\cP$ 
whose objects are the finitely generated, resp. finitely presented, resp. coherent functors. 
Following Beligiannis \cite[Def.~3.1]{Be} the  categories $\rfp\cP$ and $(\cP\lfp)^{\circ}$ are called the \emph{Freyd categories}
of $\cP$. 

\end{definition}

We obtain the following commutative diagram of fully faithful functors:
\begin{equation}\label{Dia1}\xymatrixrowsep{1 em}
\xymatrix{
& \cP \ar@{^(->}[d]_{P_{\cP}} \ar[rrd]^{Y_{\cP}}& & \\
\rcoh\cP  \ar@{^(->}[r] & \rfp\cP\ar@{^(->}[r] &\rfg\cP\ar@{^(->}[r] & \rMod\cP \\
} \end{equation} 
where by definition $P_{\cP}$ is the Yoneda functor whose codomain is restricted to 
finitely presented functors.
(The class of
natural transformations between finitely generated functors is a set since, if $\cP_X\twoheadrightarrow F$
and $\cP_Y\twoheadrightarrow G$, any morphism $\alpha:F\to G$ can be lifted to a morphism
 $\cP_X\to\cP_Y\in\cP(X,Y)$).

%

\begin{numero}\label{N1.1}
Given $\cP$ a projectively complete category, Freyd proved in
\cite{Fr} that $\rfp\cP$ is projectively complete, it has cokernels and an object $F$
is projective in $\rfp\cP$ (i.e., for any epimorphism $p:G_1\twoheadrightarrow G_2$ in $\rfp\cP$ 
the map $\rfp\cP(F,G_1)\to \rfp\cP(F,G_2)$ is surjective) if and only if 
$F\cong \cP_X$.
%

\begin{definition}\label{D1.2b}
Let $\cC$ be a pre-additive category.
A family of objects $\cG$ is called a 
\emph{generating family} if, for any non zero
morphism $f:C\to D$ in $\cC$, there exists a morphism
$h:G\to C$ with $G$ in $\cG$, such that $f\circ h\not= 0$.
A co-generating family of $\cC$ is a generating  family of $\cC^\circ$.
\end{definition}
%

\begin{remark}\label{R1.2b}
Let $\cP$  be projectively complete, hence it
is a generating (resp. co-generating)
family of projective (resp. injective) objects for $\rfp\cP$
(resp. $(\cP\lfp)^{\circ}$).
\end{remark}
%

\numero\label{N:A.8}
In \cite{Be} Beligiannis, following Freyd, 
proved that the pair $(\rfp\cP,P_{\cP})$ 
is ``universal" between the projectively complete categories 
with cokernels  ``containing an image" of $\cP$:
(
 given any other projectively closed category $\cD$ with cokernels  and an additive functor
$F:\cP\to \cD$  there exists unique a functor $F^c:\rfp\cP\to \cD$ 
cokernel 
preserving
such that $F^{c}\circ  P_{\cP}=F$). 
In \cite[Lem.~3.3]{Be} the author proved that $P_{\cP}$ 
preserves kernels and admits a left adjoint $\Phi_{\cP}$
 if and only if $\cP$ has cokernels. \end{numero}
%

\begin{definition}\label{D1.3}(\cite[p. 103]{Fr},  \cite[\S 4]{Be}).
A projectively complete category $\cP$ is called \emph{right (resp. left) coherent}
if for any $X\in\cP$ the functor $\cP_X$ (resp. ${}_X\cP$) is coherent.
$\cP$ is called \emph{coherent}
\footnote{We remark that the notion of coherent additive category has nothing to do 
with the one proposed by Peter Johnstone for a general category.} if it is both left and right coherent.
A pre-additive category $\cC$ is called (resp. right, resp. left) coherent if and only if the category
$\proj(\cC)$ is (resp. right, resp. left) coherent. 
\end{definition}
%
%

This statement, which is probably originally due to H. Cartan, is proposed 
 in its version for a ring $R$, as an exercise in Bourbaki \cite[\S 2 Exer.~11]{Bour} and explained in  great detail in \cite[\S 1.5]{Bosch}.  Here we propose its
 translation in the language of pre-additive categories; a detailed proof can be found in
 \cite[Appendix B]{Fi1}.

\begin{proposition}\label{P:CohAb}
The category $\rcoh\cC$ is closed under extension in $\rMod\cC$.
Moreover $\rcoh\cC$ is an abelian category and 
the canonical functor  $\rcoh\cC\to \rMod\cC$ is  exact.
\end{proposition}

Let use recall that $\rMod(\cC^\circ)=\cC\lMod$ hence, passing to the opposite category, one can recover 
the previous results for left modules.

\smallskip
Given an additive category $\cB$, Freyd introduced in \cite[p.~99]{Fr}  
the notion of \emph{weak kernel}
 which permits to define the notion of \emph{weak pull-back square}.
 An additive category $\cB$ admits weak pull back square if and only if it admits weak kernels.
In \cite[Ch. 6, 6.1.1]{NeeTrCat} Neeman independently introduced the notion of
\emph{homotopy pull-back square}  which coincides with Freyd weak pull-back square.

\begin{definition}\label{D1.4}
Let $A\stackrel{f}\to B$ be a morphism in an additive category $\cB$.
A weak kernel of $f$ is a map $K\stackrel{i}\to A$ such that
$f i=0$ and, for any $X\stackrel{j}\to A$ with $f j=0$, there exists, possibly many,
$X\stackrel{\alpha}\to K$ such that $i \alpha=j$.
 $\cB$ has 
weakly  (or equivalently homotopy) pull-back squares if, 
given any pair $f_i:X_i\to Y$ with $i=1,2$, there exists
 an object $Z$ with the dashed arrows such that any commutative diagram of this type can be completed with
 (a not necessarily unique) dotted arrow:
\begin{equation}\label{E1}\xymatrixrowsep{1.2 em}
\xymatrix{
W\ar@{..>}[r] \ar[rd]\ar@/^1pc/[rr] &Z\ar@{-->}[r]^{g_1} \ar@{-->}[d]_{g_2} & X_1\ar[d]^{f_1} \\
& X_2\ar[r]^{f_2}& Y.\\
}
\end{equation}
One can define dually the notions of weak cokernel and weak push-out.
\end{definition}
%

\begin{proposition}\label{P1.3} (\cite[Prop.~4.5]{Be}).
Let $\cP$ be a projectively complete category.
The following are equivalent:
\begin{enumerate}
\item $\cP$ is right  coherent;
\item $\cP$ admits weak kernels;
\item $\rfp\cP=\rcoh\cP$  is an abelian exact full subcategory of $\rMod\cP$ whose
projective  objects are exactly the representable functors in $\cP$.
\end{enumerate}
Moreover:
\begin{itemize}
\item $\cP$ has kernels iff $\rfp\cP=\rcoh\cP$ is abelian with ${\rm gl.dim }(\rcoh\cP)\leq 2$;
\item $\rfp\cP=\rcoh\cP$ is abelian with 
${\rm gl.dim }(\rfp\cP)=0$ iff  $\cP\simeq \rcoh\cP$ is  semisimple;
\item  ${\rm gl.dim }(\rcoh\cP)=1$ iff $\cP$ is not abelian semisimple but for any morphism $f$ in
$\cP$ we have that $\Ker(f)$ is split monic. \end{itemize}
\end{proposition}
%

\section{$t$-structures}\label{Ap2}

\numero{\bf Horthogonal  classes.}\label{N:oc}
Let $\cC$ be a pre-additive category and $\cU\subseteq \cC$;
we set
$\cU^{\perp}=\{C\in\cC\; |\; \cC(U,C)=0\; \forall U\in\cU\}$ and 
${\;}^{\perp}\cU=\{C\in\cC\; |\; \cC(C,U)=0\; \forall U\in\cU\}$.
%

\numero{\bf $t$-structures.}\label{B2.1}
The notion of $t$-structure is the analog for triangulated categories of that of torsion pair
for abelian categories.
Given $\cC$ a triangulated category, we will denote by $[1]$ its suspension  functor,
by $[n]$ its $n^{\rm th}$-iterated functor with $n\in\Bbb Z$ 
and we will use the notations $X\to Y\to Z{\stackrel{+}\to} $ 
for a distinguished triangle.
When we say that $\mathcal{U}$ is a subcategory of $\cC$, we always mean that $\mathcal{U}$
is a full subcategory which is closed under isomorphisms, finite direct sums and
direct summands.
 
If $\mathcal{U},\mathcal{V}$ are full subcategories of $\cC$, then we denote by 
 $\mathcal{U}\star \mathcal{V}$ 
 the full subcategory of $\cC$ consisting of objects $X$ which may be included in a distinguished triangle $U\to X\to V \stackrel{+}\to\,$ in 
$\mathcal{C}$, with $U \in \mathcal{U}$ and $V \in \mathcal{V}$.
 By the octahedral axiom we have 
$(\mathcal{U}\star \mathcal{V})\star\mathcal{W}=\mathcal{U}\star (\mathcal{V}\star\mathcal{W})$ (\cite{IY}).
 $\mathcal{U}$ is called \emph{extension closed}
if $\mathcal{U}\star\mathcal{U}=\mathcal{U}$.
In general $\mathcal{U}\star\mathcal{V}$ is not idempotently complete
but by \cite[Prop.~2.1]{IY} if the subcategories are orthogonal
$\cC(\mathcal{U},\mathcal{V})=0$ hence $\mathcal{U}\star\mathcal{V}$ 
is closed under direct summands.

A \emph{$t$-structure} in a triangulated category $\cC$ (\cite{BBD}) is a pair $\cD:=(\cD^{\leq0},\cD^{\geq0})$ of full subcategories of $\cC$ such that
($\cD^{\leq n}:=\cD^{\leq0}[-n]$ and $\cDgg n:=\cDgg0[-n]$):
\begin{enumerate}
\item[\rm (i)] $\cDll0\subseteq\cDll1$ and $\cDgg0\supseteq\cDgg1$;
\item[\rm (ii)] $\Hom_\cC(X,Y)=0$, for every $X$ in $\cDll0$ and every $Y$ in $\cDgg1$;
\item[\rm (iii)] For any object $X\in\cC$ there exists a distinguished triangle in $\cC$
$A\to X\to B \stackrel{+}\to $
 with
$A\in\cDll0$ 
and $B\in\cDgg1$ (hence $\cC=\cD^{\leq 0}\star \cD^{\geq 1}$).
\end{enumerate}

By \cite[Prop.~1.3.3, Th.~1.3.6]{BBD}
the inclusion of $\cDll n$ in $\cC$ (resp.  of $\cDgg n$ in $\cC$) has a right adjoint 
$\delta^{\leq n}$ (resp. a left adjoint $\delta^{\geq n}$)
 called the {\rm truncation functor}. For every object $X$ in $\cC$  a unique morphism 
$d\colon\delta^{\geq 1}(X)\to \delta^{\leq 0}(X)[1]$ such that the triangle
$
\delta^{\leq 0}(X)\ra X\ra\delta^{\geq 1}(X)\overset{d}\ra \delta^{\leq 0}(X)[1]
$
is distinguished. 
This triangle is (up to a unique isomorphism) the unique distinguished triangle 
$(A,X,B)$ with $A$ in $\cDll0$ and $B$ in $\cDgg1$ and it is called
the \emph{approximating triangle} of $X$ (for the $t$-structure $\cD$).

The classes $\cD^{\leq 0}$ and $\cD^{\geq 0}$ are called the \emph{aisle} and the \emph{co-aisle} of the $t$-structure $\cD$.
The data of a $t$-structure is equivalent to the datum
of its aisle (which is by definition: ${\mathcal U}\hookrightarrow \cC$ full subcategory admitting a right adjoint, closed by $[1]$, stable under extension; see  \cite{KeVo} and \cite[7.2]{KeDerCarTil}).

The category $\cH_\cD:=\cDll0\cap\cDgg0$ is abelian and is called the {\rm heart} of the $t$-structure. 
The truncation functors induce functors
$\hH_\cD^i\colon \cC \to \cH_\cD$, $i\in\mathbb Z$, called the {\rm $t$-cohomological functors associated with the $t$-structure $\cD$}, defined as follows:
$\hH_\cD^0(X):=\delta^{\geq0}\delta^{\leq0}(X)\simeq\delta^{\leq0}\delta^{\geq0}(X)$ and for every $i\in \bZ$, 
$\hH_\cD^i(X):=\hH_\cD^0(X[i])$.

\numero{\bf Notation.}\label{A:not}
Given $\cD$ a $t$-structure on a triangulated category $\cC$ we will
denote by $\cD^{[a,b]}=\cD^{\geq a}\cap\cD^{\leq b}\subseteq \cC$ with $a\leq b$ in $\Bbb Z$.
Hence $\cD^{[a,a]}=\cH_\cD[-a]$.

Given an abelian category  $\cA$, we denote by
$[X^i\to X^{i+1}\to\cdots \to X^{i+n}]$ for $n\in \Bbb N$, the complex
in $C(\cA)$ in degrees $i$ to $i+n$ whose remains terms are $0$.
We  note by $X^{\geq n}$ (resp.  $X^{\leq n}$) the complex which coincides with 
$X^\point$ in degrees
greater than (resp. less than) or equal to $n$ 
and is zero otherwise.

Following \cite{BBD}  we will denote by
$\Hom^n_{\cC}(X,Y):=\Hom_{\cC}(X,Y[n])$.
Any distinguished triangle
$X\to Y\to Z{\stackrel{+}\to} $ provides $\forall T\in\cC$ the following  long exact sequences
{\small{ \[
\cdots \to \Hom^{-1}_{\cC}(T,Z)\to \Hom_{\cC}(T,X)\to 
\Hom_{\cC}(T,Y)\to \Hom_{\cC}(T,Z)\to \Hom_{\cC}^1(T,X)\to\cdots  \]\[
\cdots \to \Hom^{-1}_{\cC}(X,T)\to \Hom_{\cC}(Z,T)\to 
\Hom_{\cC}(Y,T)\to \Hom_{\cC}(X,T)\to \Hom_{\cC}^1(Z,T)\to\cdots
\]}}
\begin{definition}\label{Dnat}
Given an abelian category $\cA$, its derived category $D(\cA)$ 
admits  \emph{natural $t$-structure} whose aisle $D(\cA)^{\leq 0}$ 
(resp.  co-aisle $D(\cA)^{\geq 0}$) is the subcategory of complexes without cohomology in positive (resp. negative) degrees.
\end{definition}
%

\numero\label{N:point}
Let $\cP$ be a projectively complete category.
The notation 
 $[\cdots \to L\to \stackrel{\point}M\to N\cdots] $  indicates a complex in $K(\cP)$
 whose element $M$ is placed in degree zero.


\begin{definition}\label{D2.1} (\cite[4.5 and 4.6]{KrT}).
Let $\cC$ be a triangulated category.
A non-empty full subcategory $\cN$ of $\cC$ is called a 
\emph{thick triangulated subcategory}
if \begin{description}
\item[(TS1)]
for any $X\in\cN$ and $i\in\Bbb Z$ we have $X[i]\in\cN$;
\item[(TS2)] given any distinguished triangle $X\to Y\to Z{\stackrel{+}\to}$ in $\cC$
if two objects from $\{X, Y , Z\}$ belong to $\cN$, then also the third one is in $\cN$;  
\item[(TS3)]  $\cN$ is closed under direct factors.
\end{description}
\end{definition}

One can attach to any thick subcategory $\cN$ of $\cC$  its multiplicative
system (compatible with the triangulation) 
$\Sigma(\cN)$ containing all the morphisms $X\stackrel{f}\to Y$ in $\cC$ fitting
in a distinguished triangle $X\stackrel{f}\to Y\to Z\stackrel{+}\to$ with $Z\in\cN$.
Hence one can
perform the quotient category $\cC/\cN:=\cC[\Sigma(\cN)^{-1}]$ 
(which is a category in a wider sense since it could be not locally small)
endowed with the quotient
functor $Q:\cC\to \cC/\cN$ such that by \cite[Prop.~4.6.2]{KrT}:
\begin{enumerate}
\item  $\cC/\cN$ carries a unique triangulated structure such that $Q$ is exact;
\item a morphism in $\cC$ is annihilated by $Q$ if and only if it factors through an object
in $\cN$ and moreover $\cN=\Ker Q$ (since it is thick);
\item every exact functor $\cC \to \cU$ annihilating $\cN$ factors uniquely through $Q$ via an
exact functor $\cC/\cN\to \cU$.
\end{enumerate}

\begin{numero}\label{AN}
Given $\cA$ an abelian category 
the subcategory 
\[
\cN:=\{X^\point\in K(\cA)\; | \; H^i(X^\point)=0, \; \forall i\in\Bbb Z\}
\]
is a thick subcategory of $K(\cA)$ and the quotient $K(\cA)/\cN=:D(\cA)$
defines the derived category of $\cA$ (which might be non-locally small).
\end{numero}
%

\begin{remark}\label{R2.1}
Let $\cA$ be an abelian category. If $\cA$ 
has a generating family $\cP$ of projectives 
$\cN:=\{X^\point\in K(\cA)\; | \; \Hom_{K(\cA)}(P[i],X^\point)=0, \; \forall i\in\Bbb Z\;\hbox{ \rm and } \forall P \in
\cP\}.$

(Dually 
$ \cN:=\{X^\point\in K(\cA)\; | \; \Hom_{K(\cA)}(X^\point,I[i])=0, \; \forall i\in\Bbb Z\;\hbox{ \rm and } \forall I
\in\cI\}$ if $\cA$ has a cogenerating family of injectives $\cI$).
\end{remark}
\begin{lemma}\label{L2.2}(See \cite[Exer.~5.1.5]{KrT}).
Given $\cA$  an abelian category with enough projectives and finite global dimension
${\rm gl.dim }(\cA)=n$, 
let us denote by $\cP$ the projective objects in $\cA$.
The null system $\cN\cap K(\cP)=\{0\}$ 
and $K(\cP)\simeq D(\cA). $
%
\end{lemma}

Kashiwara and Schapira generalized the previous Lemma as follows:

\begin{proposition}\label{KS}(\cite[Prop.~13.2.6]{KS})
Let $\cA$ be an abelian category and $\cE$ a full additive subcategory of $\cA$ such that:
\begin{enumerate}
\item $\cE$ is cogenerating (resp. generating);
\item there exists $d>0$ such that, for any exact sequence {\small$Y_d\to \cdots \to Y_1\to Y\to 0$}
(resp. {\small$0\to Y_1\to \cdots \to Y_d$})
with $Y_j\in\cE$, we have $Y\in \cE$.
\end{enumerate}
The canonical functor below is a triangulated equivalence of categories
\[ {K(\cE)\over {K(\cE)\cap\cN}}\stackrel{\simeq}\longrightarrow D(\cA).\]
\end{proposition}

%
\begin{lemma}\label{Prop:Schn}\cite[Lem.~1.2.17]{Schn} Given a $t$-structure $\cT$
on a triangulated category $\cC$ and a  saturated null system  
$\cN$ with $Q:\cC\to\cC/\cN$ its canonical quotient functor; the essential images
$(Q(\cT^{\leq 0}),Q(\cT^{\geq 0}))$ form a $t$-structure on $\cC/\cN$ if and only if for any 
distinguished triangle $X_1\to X_0\to N\stackrel{+1}\to$ with $X_1\in\cT^{\geq 1}$,  $X_0\in\cT^{\leq 0}$ and
$N\in\cN$ we have
$X_1,X_0\in \cN$.
\end{lemma}

\backmatter


\end{document}